\documentclass[a4paper,10pt]{amsart}
\usepackage[utf8]{inputenc}
\usepackage{lmodern}
\usepackage[T1]{fontenc}
\usepackage{microtype} 	
\usepackage[all]{xy}
\usepackage{comment}
\usepackage{enumerate}
\usepackage{scalerel, stackengine}

\usepackage{soul}
\usepackage{amsmath,amssymb,amsfonts,amsthm,latexsym,mathrsfs,tikzsymbols}
\usepackage{mathtools}
\usepackage{ stmaryrd } 
\usepackage{xcolor}
\usepackage{hyperref}
\usepackage{cleveref}
\hypersetup{colorlinks=true,linkcolor=gray,citecolor=gray}
\usepackage[hyperpageref]{backref} 
\usepackage{booktabs}
\usepackage{longtable} 
\usepackage{pdflscape} 
\usepackage{array}

\newtheorem{lemma}{Lemma}[section]
\newtheorem{proposition}[lemma]{Proposition}
\newtheorem{theorem}[lemma]{Theorem}

\newtheorem{corollary}[lemma]{Corollary}

\newtheorem{maintheorem}{Theorem}

\newtheorem{maincorollary}[maintheorem]{Corollary}

\theoremstyle{definition}

\newtheorem{definition}[lemma]{Definition}
\newtheorem{conjecture}[lemma]{Conjecture}

\newtheorem{question}[lemma]{Question}

\newtheorem*{problem*}{Problem}

\theoremstyle{remark}
\newtheorem{remark}[lemma]{Remark}
\newtheorem{example}[lemma]{Example}
\newtheorem*{example*}{Example}

\newtheorem*{remark*}{Remark}

\numberwithin{equation}{section}

\newcommand\Ima{{\rm Im}}

\newcommand\diag{{\rm diag}}
\newcommand{\GL}{\operatorname{GL}}
\newcommand{\SL}{\operatorname{SL}}
\newcommand{\SO}{\operatorname{SO}}
\newcommand{\SU}{\operatorname{SU}}
\newcommand{\PSL}{\operatorname{PSL}}
\newcommand{\PGL}{\operatorname{PGL}}
\newcommand{\Ma}{\operatorname{M}}
\newcommand\End{\operatorname{End}}

\DeclareMathOperator{\Sym}{Sym}

\newcommand\Span{\operatorname{Span}}

\newcommand\hdim{\operatorname{hdim}}
\newcommand\Di{\operatorname{Di}}
\newcommand{\Out}{\operatorname{Out}}

\newcommand\Mod{{\rm mod}} 

\newcommand\Irr{\operatorname{Irr}}

\DeclareMathOperator{\Rep}{Rep}

\newcommand\cd{\operatorname{cd}}

\newcommand\Char{{\rm char}}

\newcommand{\ZZ}{\mathcal{Z}}
\renewcommand{\O}{\mathcal{O}}
\newcommand{\U}{{\mathcal U}}
\newcommand{\qa}[3]{\left(\frac{#1, #2}{#3}\right)}
\newcommand\PCI{\operatorname{PCI}}
\newcommand\Mexc{\operatorname{(M_{exc})}}
\newcommand\wMexc{\operatorname{(wM_{exc})}}

\newcommand\Cen{\operatorname{Cen}}
\newcommand{\Cliff}{\operatorname{Cliff}}
\newcommand{\LL}{\mathcal{L}}
\newcommand{\I}{\mathcal{I}}
\newcommand{\vFQ}{\operatorname{vFQ}}
\newcommand{\vQL}{\operatorname{vQL}}
\newcommand{\spec}{\operatorname{spec}}
\newcommand{\con}{\operatorname{Con}}
\newcommand\Br{{\rm Br}}

\newcommand\ot{\otimes}
\newcommand\op{\oplus}
\newcommand\mc{\mathcal}

\newcommand\wh{\widehat}
\stackMath
\newcommand\reallywidehat[1]{%
\savestack{\tmpbox}{\stretchto{%
  \scaleto{%
    \scalerel*[\widthof{\ensuremath{#1}}]{\kern-.6pt\bigwedge\kern-.6pt}%
    {\rule[-\textheight/2]{1ex}{\textheight}}
  }{\textheight}%
}{0.5ex}}%
\stackon[1pt]{#1}{\vstretch{1.5}{\tmpbox}}%
}
\DeclarePairedDelimiter\abs{\lvert}{\rvert}

\newcommand{\N}{{\mathbb N}}
\newcommand{\Z}{{\mathbb Z}}
\newcommand{\F}{{\mathbb F}}
\newcommand{\Q}{{\mathbb Q}}
\newcommand{\R}{{\mathbb R}}
\newcommand{\C}{{\mathbb C}}

\newcommand\vcd{\operatorname{vcd}}
\newcommand{\V}{\mathrm{V}}
\newcommand\Fit{\operatorname{Fit}}

 \newcommand\restr[2]{{
   \left.\kern-\nulldelimiterspace 
   #1 
   \right|_{#2} 
   }}

\let\oldtocsubsection=\tocsubsection

\renewcommand{\tocsubsection}[2]{\hspace{2em}\oldtocsubsection{#1}{#2}}

\definecolor{Vino}{rgb}{0.256,0,0}

\usepackage[margin=1.4in]{geometry}
\thanks{The first author is grateful for financial support from the FWO and the F.R.S.–FNRS under the Excellence of Science (EOS) program (project ID 40007542). The second author is grateful to Fonds Wetenschappelijk Onderzoek Vlaanderen - FWO (grant 88258), and le Fonds de la Recherche Scientifique - FNRS (grant 1.B.239.22) for financial support.}

\begin{document}
\title[Representing in Low Rank I: conjugacy, topological \& homological aspects]{Representing in Low Rank I: conjugacy, topological and homological aspects}

\author{Robynn Corveleyn}
\author{Geoffrey Janssens}
\author{Doryan Temmerman}

\address{(Robynn Corveleyn) \newline Institut de Recherche en Math\'ematique et Physique, UCLouvain, Chemin du Cyclotron 2, 1348 Louvain-la-Neuve, Belgium \newline Email address: {\tt robynn.corveleyn@uclouvain.be}}

\address{(Geoffrey Janssens) \newline Institut de Recherche en Math\'ematique et Physique, UCLouvain, Chemin du Cyclotron 2, 1348 Louvain-la-Neuve, Belgium and \newline Department of Mathematics and Data Science, Vrije Universiteit Brussel,
Pleinlaan 2, 1050 Elsene, Belgium
\newline Email address: {\tt geoffrey.janssens@uclouvain.be}}

\address{(Doryan Temmerman) \newline AI Lab, Vrije Universiteit Brussel, 
Pleinlaan 2, 1050 Elsene, Belgium\newline Email address: {\tt doryan.temmerman@vub.be}}

\dedicatory{Dedicated to the $70$\textsuperscript{th} birthday of Eric Jespers, with gratitude}

\begin{abstract}
 In this series of papers, we investigate properties of a finite group which are determined by its low degree irreducible representations over a number field $F$, i.e. its representations on matrix rings $\operatorname{M}_n(D)$ with $n \leq 2$. In particular we focus on representations on $\operatorname{M}_2(D)$ where $D$ is a division algebra having an order $\mathcal{O}$ such that $\mathcal{O}$ has finitely many units, i.e. such that $\operatorname{SL}_2(\mathcal{O})$ has arithmetic rank $1$. In this first part, the focus is on two aspects.

 One aspect concerns characterisations of such representing spaces in terms of Serre's homological goodness property, small virtual cohomological dimension and higher Kleinian-type embeddings. As an application, we obtain several characterisations of the finite groups $G$ whose irreducible representations are of the mentioned form. In particular, such groups $G$ are precisely those such that $\mathcal{U}(R G)$, with $R$ the ring of integers of $F$, can be constructed from groups which virtually map onto a non-abelian free group. Along the way we investigate the latter property for congruence subgroups of higher modular groups and its implications for the congruence kernel. This is used to obtain new information on the congruence kernel of the unit group of a group ring.

The second aspect concerns the conjugacy classes of the images of finite subgroups of $\mathcal{U}(R G)$ under the irreducible representations of $G$. More precisely, we initiate the study of a blockwise variant of the Zassenhaus conjectures and the subgroup isomorphism problem. Moreover, we contribute to them for the low rank representations above. 
\end{abstract}

\maketitle

\newcommand\blfootnote[1]{%
  \begingroup
  \renewcommand\thefootnote{}\footnote{#1}%
  \addtocounter{footnote}{-1}%
  \endgroup
}

{\small \tableofcontents}

\section{Introduction}

This paper concerns representation theory of finite groups $G$. Arguably the main aim of it is to determine which group theoretical invariants of $G$ can be retrieved from its category of representations $\Rep_F(G)$, where $F$ is some field of characteristic $0$. In this work we consider the general problem of retrieving information from the group algebra $FG$ viewed as an $F$-algebra. \smallskip 

This problem is highly dependent on the choice of the ground field $F$. For example if $F = \C$, then $FG \cong FH$ for $G$ and $H$ abelian exactly when $|G|=|H|$. On the other hand, if $F= \Q$ then the group algebra determines all the structure constants of an abelian group by a theorem of Perlis--Walker \cite{PW}. One reason that $\Q$ detects more is the theorem of Wedderburn--Artin, which says that there is an isomorphism of $F$-algebras
\begin{equation}\label{W-A decomposition intro}
FG \cong \Ma_{n_1}(D_1) \times \cdots \times \Ma_{n_{\ell}}(D_{\ell}),
\end{equation}
where $D_i$ are finite-dimensional division $F$-algebras. In particular, the Brauer group $\Br(F)$, i.e. the division algebras over $F$, has an influence. One way to view this is, the ‘smaller’ $F$, the ‘richer’ its Brauer group. Nevertheless, Dade \cite{Dade} constructed non-isomorphic groups $G$ and $H$ with isomorphic group algebras over any field of characteristic $0$, or in other words such that $\Q G \cong \Q H$. For a long time, it was believed that over the integers $\Z$ (equivalently over all group rings $RG$), a group $G$ is uniquely determined by its integral group ring $\Z G$. However, Hertweck constructed a counterexample to this \cite{Hertweck}. Nevertheless, if $\Z G \cong \Z H$, then $G$ and $H$ are cocycle twists of each under a Cech-type non-abelian cohomology theory constructed by Roggenkamp--Kimmerle \cite{KR}. Thus, although Hertweck managed to find solvable groups for which this cohomology group is non-trivial, it indicates that $G$ and $H$ are still very related to each other. \smallskip

The motivation to consider $FG$ lies in the fact that it is a distinguished object in $\Rep_F(G)$. Firstly, being the regular module, it is directly related to the set of irreducible $F$-representations of $G$. A second aspect is more categorical. Namely, $\Rep_F(G)$ viewed as a symmetric tensor\footnote{With tensor category we mean an abelian, $F$-linear, semisimple, monoidal and rigid category whose monoidal unit is simple.} category is known (by \cite{DM, Del}) to determine completely the underlying group $G$. The latter data is equivalent to determining $FG$ as a Hopf algebra.
However, solely viewed as a tensor category, there can be non-isomorphic finite groups $G$ and $H$ such that $\Rep_F(G) \simeq \Rep_F(H)$ as tensor categories, giving rise to the concept of isocategorical groups \cite{EG, Dav, Gal}. If $F$ is algebraically closed, then $G$ and $H$ are \emph{isocategorical} if and only if $FG$ and $FH$ are up to some Drinfeld twist isomorphic as Hopf algebras \cite[Proposition 3.1]{EG}. This highlights again the distinguished role of the group algebra. In our setting, one follows the natural desire to retrieve information on $G$ from much less (and more realistic to be determined) starting structure.\smallskip

Interestingly, if $R$ is the ring of integers of a number field, then\footnote{The technical condition on $R$ allowing this conclusion, is that $R$ is a \emph{$G$-adapted ring}. This means that $R$ is an integral domain of characteristic $0$ such that none of the prime divisors of $|G|$ are invertible in $R$.} the $R$-algebra $RG$ contains the same information as its unit group $\U (RG)$. In other words, if $RG \cong RH$ as $R$-algebras, then $\U(RG) \cong \U(RH)$ as groups. Therefore, for several decades now, there has been considerable interest in understanding how $G$ is embedded in the finitely presented group $\U(RG)$. One proposed line of research is to construct a subgroup $N$ such that (1) the construction of $N$ does not depend on knowing the basis $G$ of $RG$ (i.e. it is generic in $G$), (2) it is of finite index in $\U(RG)$ and ideally (3) it is normal and torsion-free. This allows to retrieve invariants from the finite group $\U(RG)/N$. Property (2) is related with the problem of generic constructions of invertible elements \cite{JJS}. 

Now recall that each simple factor $\Ma_{n_i}(D_i)$ in the Wedderburn--Artin decomposition \eqref{W-A decomposition intro} corresponds to a primitive central idempotent $e_i$ of $FG$; we denote the set of all such elements by $\PCI(FG)$.  It is a general fact that $\U(R G)$ has finite index in $\prod_{i=1}^{\ell}\U (RGe_i)$. As such, both property (2) and the normality in (3) are directly related with the description of normal subgroups of the groups $\U (R Ge_i)$. In turn, this heavily depends on the isomorphism type of $F Ge_i = \Ma_{n_i}(D_i)$. More precisely, the presence of the following type of simple algebras often breaks the standard methods used (see \cite[Section 6.1]{BJJKT} for more details):

\begin{definition}[Exceptional components] \label{def_exc_comps}
A finite-dimensional simple algebra $B$ is called {\it exceptional} if it is isomorphic to one of the following:
\begin{enumerate}
\item a non-commutative division algebra which is not a totally definite quaternion algebra over a number field,
\item a matrix algebra $\Ma_2(D)$ with $D \in \{\Q, \Q(\sqrt{-d}), \qa{-a}{-b}{\Q} \mid a,b,d \in \N_0 \}$.
\end{enumerate}
If $B$ is of the latter form, then we speak of an {\it exceptional matrix} algebra and in the former case of an {\it exceptional division} algebra. If $B \cong FGe$ for some $e \in \PCI (FG)$, then $B$ is called an exceptional component of $FG$.
\end{definition}
The name was coined in \cite{MR3195747} and the reason that in practice they are exceptional (i.e. require `other methods') is different. The exceptional matrix algebras are exactly those $\Ma_n(D)$ for which the $S$-rank of $\SL_n(D)$ is $1$, where $S$ is the set of Archimedean places of $\mc{Z}(D)$ (see \cite[Remark 6.7]{BJJKT} for details and \cref{definition SL_1} for the definition of $\SL_n$ over non-commutative rings). Consequently, they are exactly those matrix algebras having a maximal order\footnote{It is well-known that a maximal order in $\Ma_n(D)$ is of the form $\Ma_n(\O)$ with $\O$ a maximal order in $D$. Furthermore, the unit groups of two orders are commensurable. Thus this property does not depend on the chosen order.} $\Ma_n(\O)$ such that $\SL_n(\mc{O})$ does not satisfy the subgroup congruence problem. Moreover, there exist non-central normal subgroups which are not of finite index. The non-exceptional division algebras $D$ are exactly those having an order $\O$ such that $\SL_1(\O)$ is finite. The others are problematic because $\SL_1(D)$ has no unipotent elements, which are an ingredient for most generic constructions of units in $\U (R G)$. Additionally, the subgroup congruence problem for $\SL_1(D)$ is sadly still open. Therefore, for the lines of research described earlier, irreducible representations in such simple algebras require a separate study.\smallskip

This series of papers has three aims. \emph{A first aim} is to better understand irreducible representations into an exceptional matrix algebra $A$. Here the focus will be on the following topological and homological properties and invariants:

\begin{enumerate}
    \item \emph{Serre's good property:} a group $\Gamma$ is \emph{good} if the cohomology of $\Gamma$ coincides with the cohomology of its profinite completion  $\widehat{\Gamma}$, in any finite module (see \Cref{subsectie good} for a precise formulation).
    \item \emph{Virtual cohomological dimension:} for any discrete group $\Gamma$, this is the cohomological dimension of a torsion-free finite index subgroup of $\Gamma$, and we denote  it by $\vcd(\Gamma)$ (see \Cref{subsectie vcd} for a precise formulation).
    \item \emph{Higher Kleinian property:} a group $\Gamma$ \emph{has property $\Di_n$} if $n$ is the smallest integer such that $\Gamma$ can discretely be embedded into $\SL_n(\C)$ (see \Cref{def discrete in}). 
    \item  \emph{A large congruence subgroup:} a group $\Gamma$ is \emph{large} if it maps onto a virtually free group (see \Cref{sectie vFQ} for details).
\end{enumerate}

In \Cref{section hom and topo char} and \Cref{sectie vFQ} we will determine the above for the unimodular elements of an order in $A$. It will turn out that exceptional matrix algebras cannot be distinguished from other simple algebras through any of the above. However, if we assume that the matrix algebra is associated to an irreducible representation of a finite group, then they can be. In case of the large property, our approach will be geometric, by relating it to the setting of higher modular groups. Furthermore, as an application, we obtain new information on the congruence kernel of an integral group ring.

\emph{The second aim} is to characterise and classify the finite groups $G$ and number fields $F$ such that $FG$ has the following property: 
\begin{equation}
\tag{$\mathrm{M}_{\mathrm{exc}}$} \text{ All } \quad F Ge \cong \Ma_n(D), \quad  \text{ with } n \geq 2, \text{ are exceptional}. 
\end{equation}
So with respect to the general theory described earlier, they represent the `most degenerate groups'. As a second application of understanding the properties above, we obtain several characterisations of when $FG$ has $\Mexc$. Some of them extend the work in \cite{FreebyFree} and others are completely new. In the sequel to this paper, we will give a precise classification of the finite groups $G$ and number fields $F$ such that $FG$ has $\Mexc$, yielding a new case of Kleinert's virtual structure problem \cite{KleinertSurvey}. Moreover, we will solve the rational isomorphism problem for these groups, i.e. if $\Q G \cong \Q H$ and $\Q G$ has $\Mexc$, then $G \cong H$.

\emph{The third aim} is the study of conjugacy classes of finite subgroups through low rank irreducible representations. More precisely, let $\rho \colon \U(FG) \rightarrow \Ma_n(D)$ be an irreducible $F$-representation of $G$ with $\Ma_n(D)$ exceptional and let $H$ be a finite subgroup of $\U (R G)$.  Then we investigate the $\GL_n(D)$-conjugacy class of $\rho(H)$. In other words, we study a blockwise variant of the Zassenhaus conjectures and the subgroup isomorphism problem. \smallskip

We now explain in more detail the main results of this paper.

\subsection{Topological and homological characterisation of exceptional components and applications} \addtocontents{toc}{\protect\setcounter{tocdepth}{1}}
Let $A$ be a finite simple algebra over a number field $F$ and $\O$ an order in $A$. In \Cref{section hom and topo char} we characterise when any of the following occur:

\begin{itemize}
    \item $\vcd(\SL_1(\O))$ divides $4$,
    \item $\SL_1(\O)$ has property $\Di_n$ with $n \mid 4$,
    \item $\SL_1(\O)$ is good.
\end{itemize}
Each of the above is independent of the chosen order $\O$, i.e. they depend only on the isomorphism type of $A$. Moreover, they all hold for exceptional matrix algebras. See \Cref{simple with vcd | 4}, \Cref{discrete in SL4 description} and \Cref{good for exceptional} for the detailed statements.

The aforementioned results are not representation theoretical ones. Concretely, it may happen e.g. that $\vcd(\SL_1(\O)) \mid 4$, but that $\SL_1(\O)$ does not have $\Di_n$ with $n \mid 4$, and vice versa. However if we assume that $A$ is a simple component of a group algebra $FG$ having $\Mexc$, then the various properties are directly related.

\begin{maintheorem}[\Cref{block VSP main theorem}, \Cref{final theorem dis} \& \Cref{good for group rings}]\label{main characterisation theorem}
Let $G$ be a finite group and $F$ a number field with ring of integers $R$. Then the following are equivalent:
\begin{enumerate}[(i)]
    \item $FG$ has $\Mexc$,
    \item $[F:\Q]\leq 2$ and $\vcd (\SL_1(R Ge)) \mid 4$,
     for every $e \in \PCI(F G)$ such that $F Ge$ is not a division algebra,
    \item $[G:\Fit(G)] \leq 2$ and $\Fit(G)$ has class at most $3$ and
\[\SL_1(R Ge) \text{ has } \Di_n \text{ with } n \mid 4,\]
for every $e \in \PCI(F G)$ such that $F Ge$ is not a division algebra,
    \item  $\U ( R G)$ is good,
    \item $\SL_1(R G)$ is good.
\end{enumerate}
\end{maintheorem}

The fact that $\Q G$ having $\Mexc$ implies that $\U(\Z G)$ is good was already obtained in \cite[Theorem 1.5]{CdRZ}.

Two instrumental ingredients for \Cref{main characterisation theorem} are (1) the description of the exceptional division algebra components that a group algebra with $\Mexc$ can have and (2) the behaviour of the $\Mexc$ property over different ground fields. Both problems are resolved in \Cref{section description components}.

\begin{remark*}
In \cite{FreebyFree} a characterisation in terms of virtual cohomological dimension and the Kleinian property was obtained of the finite groups $G$ such that $\Q G$ has no exceptional division components and the non-division components are all of the form $\Ma_2(\Q(\sqrt{-d}))$ for some $d \in \N$. More precisely, they show that $\Q G$ has such components if and only if $\vcd(\SL_1(\Z G e)) \leq 2$ for all $e \in \PCI(\Q G)$. In turn the latter is equivalent to $\SL_1(\Z Ge)$ having $\Di_n$ with $n \mid 2$ for all $e \in \PCI(\Q G)$. Interestingly, including all exceptional matrix algebras adds some intricate properties on the Fitting subgroup.
\end{remark*}

In light of the aforementioned result from \cite{FreebyFree} (and its predecessors) and \Cref{main characterisation theorem}, it is natural to formulate the following general problem.

\begin{problem*}[Block Virtual Structure problem]
    Let $\mc{P}$ be a property. Classify the group algebras $FG$ such that, for each primitive central idempotent $e$ of $FG$ in some subset (depending on $\mc{P}$) of  $\PCI (FG)$, the simple algebra $FGe$ has property $\mc{P}$.
\end{problem*}

Later in the introduction we briefly recall the virtual structure problem and in \Cref{sectie block VSP} it is shown how it relates to the above problem.

\subsection{Virtually free quotients and congruence kernels for higher modular groups}

Let $R$ be a ring such that $R/I$ is finite for any ideal $I$ in $R$. Recall that a subgroup $\Gamma$ of $\SL_n(R)$ is said to be a \emph{principal congruence subgroup} if it is the kernel of a reduction map $\SL_n(R) \rightarrow \SL_n(R/I)$. Such subgroups are of finite index and any subgroup containing one is called a \emph{congruence subgroup}.

\subsubsection*{A finite index subgroup mapping onto free groups}

The first main result of \Cref{sectie vFQ} concerns \emph{largeness} for principal congruence subgroups of $\SL_2(\O)$ with $\O$ an order in $D$ such that $\Ma_2(D)$ is an exceptional matrix algebra. A group is called \emph{large} if it has a quotient which is virtually a non-abelian free group. Such groups are also referred to as having property $\vFQ$, see \cite{Lubot96}. However, we will avoid this notation to avoid confusion with another property, which we will name \emph{property vQL}. A group has property $\vQL$ if it has a finite index subgroup which maps onto a non-abelian free group, i.e. it virtually has a non-abelian free quotient. If a group is large, then it also has $\vQL$, since the inverse image of the free group provided by largeness is a finite index subgroup which has a free quotient. One reason for the terminology “large” is the fact that $\Gamma$ being large implies that $\Gamma$ is \emph{SQ-universal}, meaning that every countable group can be embedded in a quotient of $\Gamma$.

\smallskip
In \cite[Theorem 3.7]{Lubot96} it was proven that if $\Gamma_d$ is a torsion-free finite index subgroup of $\SL_2(\mc{I}_d)$, then $\Gamma_d$ is large. Moreover, if the prime $p$ does not split in $\mc{I}_d$, then the kernel of $\SL_2(\mc{I}_d) \rightarrow \SL_2(\mc{I}_d/p\mc{I}_d)$ maps onto a free group of rank $p^3+p-1$.  Note that the property $\vQL$ is stable under commensurability. Hence \cite[Theorem 3.7]{Lubot96} gave a constructive version of Grunewald--Schwermer's theorem \cite{GS}, stating that $\SL_2(\mc{I}_d)$ has $\vQL$. We obtain the analogous statement for $\SL_2(\O)$ with $\O$ an order in $\qa{u}{v}{\Q}$, where $u$ and $v$ are strictly negative integers.

\begin{maintheorem}[\Cref{infinite abelianization SL}]\label{main th large}
Let $\qa{u}{v}{\Q}$ be a totally definite quaternion algebra with centre $\Q$ and let $\LL_{u,v}$ be the $\Z$-order with basis $\{1,i,j,k\}$. Then every torsion-free principal congruence subgroup of $\SL_2(\LL_{u,v})$ is large. 
\end{maintheorem}

Combining the aforementioned results in the literature with \Cref{main th large} yields the following. 

\begin{maincorollary}[\Cref{vQL for exceptional}]
Let $A$ be an exceptional matrix algebra, $\O$ an order in $A$ and $\Gamma$ a group commensurable with $\SL_1(\O)$. Then $\Gamma$ has a finite index subgroup mapping onto a non-abelian free group.
\end{maincorollary}

The proof of \Cref{main th large} is geometric and uses tools from \cite{Lubot96} and \cite{Kiefer}. More precisely, in \Cref{congr_subgroups_SL2} we relate the principal congruence subgroups of $\SL_2(\LL_{u,v})$ to congruence subgroups of certain groups of isometries of the hyperbolic $5$-space, denoted $\SL_+\left(\Gamma_4^{u,v}(\Z)\right)$. To do so, in \Cref{clifford_background} we recall how studying $2 \times 2$ matrix groups over orders in quaternion algebras all comes down to studying discrete subgroups $\SL_+\left(\Gamma_4(\Z)\right)$ of Vahlen groups, which are groups of $2 \times 2$ matrices with entries in a Clifford algebra. These groups were described in \cite{Vahlen}, where Vahlen generalised M\"obius transformations to higher-dimensional hyperbolic spaces. Discrete subgroups of the latter provide a natural generalisation of the modular group to higher dimensions. These generalised groups are denoted by $\SL_+(\Gamma_n(\Z))$ and referred to as \emph{higher modular groups}. \Cref{main th large} will be a consequence of the following statement on higher modular groups. In fact, the following was our motivation to consider largeness of congruence subgroups of $\SL_2(\LL_{u,v})$.

\begin{maintheorem}\label{cong higher modular}
    Let $u,v \in \Z_{< 0}$ and suppose that the torsion-free principal congruence subgroup $\con_n$ (see \eqref{def con}) of the higher modular group $\SL_+\left(\Gamma_4^{u,v}(\Z)\right)$ is torsion-free. Then $\con_n$ is large.
\end{maintheorem}

\begin{remark*}
In \Cref{torsionfree,} it is shown that sufficiently deep (depending on $u$) principal congruence subgroups are torsion-free. Using the connection built in the proof of \Cref{infinite abelianization SL}, the latter also yields that there exists some $n$ such that $\con_n$ is torsion-free.
\end{remark*}

\begin{remark*}
Being large implies that the abelianisation is infinite. In particular, if a group $\Gamma$ has $\vQL$, then $\Gamma$ does not have property (FAb), i.e. not every finite index subgroup has finite abelianisation. This observation, along with \Cref{infinite abelianization SL} applied to $\qa{-2}{-5}{\Q}$, can be used to remove the condition $5 \nmid |G|$ in \cite[Theorem 6.10]{BJJKT2}.
\end{remark*}

\subsubsection*{A look at the congruence kernel}

Associated to groups $\Gamma$ such as $\SL_+(\Gamma_n(\Z))$ and $\SL_1(\O)$, there are two natural topologies. Firstly, the congruence topology, which has the congruence subgroups as a basis of neighbourhoods of the identity. A second is the profinite topology, which considers all finite index subgroups as a neighbourhood basis at the identity. Denote by $\widetilde{\Gamma}$ and $\widehat{\Gamma}$ the respective completions. The \emph{subgroup congruence problem} (\textbf{SCP}) asks whether all finite index subgroups contain a principal congruence subgroup. In other words, whether the natural epimorphism 
\[\pi_{\Gamma} \colon \widehat{\Gamma} \rightarrow \widetilde{\Gamma}\]
is injective. The kernel $C(\Gamma) := \ker(\pi_{\Gamma})$ is called the \emph{congruence kernel}. Serre proposed a quantitative version of (\textbf{SCP}) by asking to compute $C(\Gamma)$.

In \Cref{congruence kernel qa} we use \Cref{cong higher modular} (resp. \Cref{main th large}) to show that the congruence kernel of $\SL_+\left(\Gamma_4^{u,v}(\Z)\right)$ (resp. $\SL_1(\O)$) contains as a closed subgroup a copy of the free group of countable rank, see \Cref{congruence kernel over quaternion} and \Cref{congruence kernel higher modular}. In case of $\SL_2(\mc{I}_d)$ this was obtained in \cite[Theorem B]{Lubot82} and our proof uses the methods in \emph{loc. cit.}

\subsection{Applications to the unit group of a group ring}

The results obtained in \Cref{sectie vFQ} are applied in \Cref{subsection vFQ vsp} to the unit group of a group ring. 

As a first application we complement \Cref{main characterisation theorem} with a group theoretical characterisation of when a group algebra $FG$ has $\Mexc$. For this, we define the following class of groups:
\[\mc{G}_{\vQL} :=  \left\{ \Gamma \mid \Gamma \text{ is virtually indecomposable and has } \vQL  \right\} \sqcup \{ \text{ finite groups } \}.\]
Subsequently, consider the larger class  
\[\prod \mc{G}_{\vQL}  := \left\{ \prod_{i} G_i \mid G_i \text{ is finitely generated abelian or in } \mc{G}_{\vQL} \right\}.\]

\begin{maintheorem}[\Cref{Mexc versus vFQ}]
  Let $G$ be a finite group, $F$ a number field and $R$ its ring of integers. Then the following are equivalent:
\begin{enumerate}
    \item $\U(RG)$ is virtually-$\prod \mc{G}_{\vQL}$.
    \item $FG$ has $\Mexc$ and no exceptional division components.
\end{enumerate}
Moreover, $FG$ has $\Mexc$ if and only if $\SL_1(RGe)$ has $\vQL$ for all $e  \in  \PCI(FG)$ such that $FGe$ is not a division algebra.
\end{maintheorem}\smallskip

The second application concerns the congruence kernel $C_G$ of the unit group of an integral group ring $\U(\Z G)$. The latter has been investigated in \cite{CdR1, CdR2, CdRZ}. In \cite{CdR1, CdR2} several families of finite groups $G$ are described with the property that if $C_G$ is infinite, then $G$ belongs to one of these families.  Under the assumption that $\Q G$ has no exceptional division components, an upper-bound on the virtual cohomological dimension of $C_G$ is obtained in \cite[Theorem 5.1]{CdRZ}. We contribute towards bounding from below, by showing that whenever $\Q G$ has an exceptional matrix component, then the congruence kernel contains a profinite free group of countable rank $\widehat{F_\omega}$. This is a direct consequence of \cite[Theorem B]{Lubot82}, \Cref{congruence kernel over quaternion} and the well-behavedness of the congruence kernel with direct products.

\begin{maincorollary}[\Cref{congruence kernel with excep component}]
    Let $G$ be a finite group. Then the following hold:
    \begin{enumerate}
        \item If $\Q G$ has an exceptional matrix algebra, then $C_{G}$ contains $\widehat{F_{\omega}}$,
        \item If $\Q G$ has $\Mexc$, then $C_{G}$ contains $\prod_{e \in \PCI_{\neq 1}} \widehat{F_{\omega}}$,
    \end{enumerate}
where $\PCI_{\neq 1} := \{ e \in \PCI(\Q G) \mid \Q Ge \text{ is not a division algebra}\}$.
\end{maincorollary}

\subsection{The blockwise subgroup isomorphism and Zassenhaus property}

Recall that $\U (\Z G) = \pm 1 \cdot V(\Z G)$ where $V (\Z G)$ is the set of units in $\U(\Z G)$ of augmentation $1$, called the \emph{normalised unit group}. Three conjectures concerning finite subgroups $H$ of $V(\Z G)$ were attributed to Zassenhaus \cite{Zass,SehAntwerp}. Namely, consider the following sets: 
    \begin{align*}
        \mc{G}_1 & := \left\{ H \leqslant V(\Z G) \mid H \text{ cyclic and finite}\right\}, \\
        \mc{G}_2 & := \left\{ H \leqslant V(\Z G)  \mid \abs{H} = \abs{G}\right\}, \\
        \mc{G}_3 & := \left\{ H \leqslant V(\Z G) \mid H \text{ finite} \right\}.
    \end{align*}
The \emph{$i$\textsuperscript{th} Zassenhaus conjecture} predicts that if $H \in \mc{G}_i$, then $H$ is $\Q G$-conjugated to a subgroup $K_{H}$ of $G$. Each of these three conjectures have been disproven \cite{Scott, Rogg, EisMar}, transforming them into the question for which $H$ and $G$ the above conjugation property holds. The above conjectures, but where conjugation over $\Q G$ is weakened to the existence of an abstract isomorphism between $H$ and some $K_H \leqslant G$, are called the \emph{subgroup isomorphism problem}, see \cite{Margolis,PdRV} and references therein. \smallskip

The aim of \Cref{block zass and iso} is to study a blockwise variant of the Zassenhaus conjectures, i.e. investigating whether the image of $H$ under any irreducible $\Q$-representation $\rho$ of $G$ is $\rho(\Q G)$-conjugated to a subgroup of $\rho(G)$. Similarly, we also consider a blockwise subgroup isomorphism problem. Note that these problems can be investigated more generally. Therefore in \Cref{Section what is block Zassenhaus} we define the Zassenhaus and subgroup isomorphism property for general finite-dimensional semisimple $F$-algebras. The subgroup isomorphism property for the set $\mc{G}_2$ and twisted group rings has recently also received considerable attention, see \cite{MaSc, MaSc2, MaSc3, HKSa, HKSi}.\smallskip

The rest of the section is devoted to the investigation of the Zassenhaus property for exceptional matrix algebras. As a first main result of independent interest, we obtain in \Cref{sectioin sbgrps excetpional} a precise description of the conjugacy classes of finite subgroups of the unit group of an exceptional matrix algebra $A$. If $\O$ is an order in $A$, the latter description allows to say which finite subgroups are $A$-conjugated to a subgroup of $\U(\O)$ which spans $A$ over $\Q$. We refer to \Cref{all in a spanning} for a precise statement.\smallskip

Another main result is the following, stating that for certain irreducible $\Q$-representations $\rho$ of exceptional type, one is always able to $\rho(\Q G)$-conjugate $\rho(H)$ into $\rho(G)$.

\begin{maintheorem}[\Cref{Zass prop for some excep components}]\label{zass prop for exc}
Let $G$ be a finite and $e \in \PCI(\Q G)$ such that $\Q Ge$ is either a field, a quaternion algebra or isomorphic to $\Ma_2(\Q(\sqrt{-d}))$ with $d\neq 3$. If $H$ is a finite subgroup of $V(\Z G)$, then $(He)^{\alpha} \leq Ge$ for some $\alpha \in \Q Ge$.
\end{maintheorem}

We expect that combining the results from \Cref{sectioin sbgrps excetpional} and the methods from \cite{PdRV}, one could obtain \Cref{zass prop for exc} for any exceptional algebra.

We would also like to mention that along the way we formulate several questions which arise naturally:  \Cref{no first iso but block yes}, \Cref{block ISO}, \Cref{iso for faith irr} and \Cref{block spectrum problem}.

\subsection{The Virtual structure problem and rational isomorphism problem for exceptional components} 

In the sequel to this paper, \cite{CJT}, we will investigate the virtual structure problem and rational isomorphism problem for group algebras $FG$ having $\Mexc$.

Recall that in \cite{KleinertSurvey}, Kleinert formulated formally the problem to obtain unit theorems for the unit group of orders $\O$ in finite-dimensional semisimple algebras $A$. More precisely, the problem concerns finding classes of groups $\mc{G}(A)$ such that most of the generic unit groups of $A$ are contained in $\mc{G}(A)$. Formulated a bit differently, one could consider a class of groups $\mc{G}$ and classify all semisimple algebras $A$ such that $A$ has a unit theorem for $\mc{G}$. Now if we restrict ourselves to group algebras, then the simple factors are related to each other via the representation theory of the group basis. As such the following problem was distilled in \cite{JesRioReine}.

\begin{problem*}[Virtual Structure Problem]
Let $\mathcal{G}$ be a class of groups. Classify the finite groups $G$ such that $\U (\Z G)$ has a subgroup of finite index lying in $\mathcal{G}$.
\end{problem*}

For a historical overview on the virtual structure problem we refer the reader to the introduction of \cite{J1}. 

A first main result resolves the virtual structure problem for the class of $\Mexc$ groups. This result completes a line of research started $30$ years ago with among others the papers \cite{JesLealRio, JesRioReine, PitaRioRuiz, FreebyFree}. The most general result until now concerned the finite groups $G$ such that $\Q G$ has the stronger property that
\begin{enumerate}[(i)]
    \item  all matrix components are of the form $\Ma_2(\Q(\sqrt{-d}))$ with $d \in \N$ ,
    \item and $\Q G$ has no exceptional division components,
\end{enumerate}
which were classified in \cite{JesRioReine, FreebyFree}. Considering property $\Mexc$ generalises their considerations in two ways. Anecdotally, the \textsf{MathSciNet} review  of \cite{JesRioReine} states that they obtain ``the most far-reaching result possible'', which indeed looked like to be so for now almost 20 years. We will prove the following.

\begin{maintheorem}\label{charc Mexc grps}
Let $G$ be a finite group and $F$ a number field such that $FG$ has $\Mexc$. Then the following holds:
\begin{enumerate}
    \item $G \cong C_3^k \rtimes Q$ with $Q$ a $2$-group of nilpotency class at most $3$,
    \item $G$ has an abelian normal subgroup $B$ such that $G/B\cong C_2$ or $C_2 \times C_2$,
    \item $[G:\Fit(G)] \leq 2$,
    \item if $Q$ is non-abelian, then $FG$ has no exceptional division components.
\end{enumerate}
Moreover, a precise description of the action of $Q$ on $C_3^k$ and the isomorphism class of $Q$ is obtained. 
\end{maintheorem}

Finally, we study the rational isomorphism problem, i.e. the problem of distinguishing a finite group among all finite groups through its rational group algebra. Positive instances of the latter problem are scarce as the authors are only aware of the case of abelian groups \cite{PW} and metacyclic groups \cite{GBdR1,GBdR2}. Note that Dade's examples \cite{Dade} are metabelian, indicating the difficulty to find new classes. Although finite groups satisfying $\Mexc$ are metabelian by \Cref{charc Mexc grps}, we will show that they are distinguishable.

\begin{maintheorem}
Let $G$ and $H$ be finite groups such that $\Q G \cong \Q H$ as $\Q$-algebras and $\Q G$ has $\Mexc$. Then $G \cong H$.
\end{maintheorem}

\vspace{0,2cm}
\noindent \textbf{Acknowledgements.} 
This research was supported through the program “Research in Pairs” by the Mathematisches Forschungsinstitut Oberwolfach in $2021$, we warmly thank the institute for its ideal working conditions. We are grateful to Angel del R\'{\i}o for sharing with us a proof of part (i) of \Cref{blockwise Cohn-Livingstone}. The second author would also like to warmly thank Leo Margolis for his encyclopedic knowledge on the Zassenhaus conjectures, his feedback on an early version and for sharing a proof of \Cref{folklore remark}.

\vspace{0,3cm}
\noindent {\it Conventions and notations.} 
Throughout the full article all groups denoted by a Latin letter will be a finite group. All orders will be understood to be $\Z$-orders. We also use the following notations:
\begin{itemize}
\item $\PCI(FG)$ for the set of primitive central idempotents of $FG$.
\item $\pi_e\colon \mathcal{U}(FG) \twoheadrightarrow FGe$ the projection onto a simple component.
\item $\mc{C}(FG) = \{ FGe \mid e \in \PCI(FG) \}$ for the set of isomorphism types of the simple components of $FG$.
\item The degree of a central simple algebra $A$ is $\sqrt{ \dim_{\ZZ(A)} A}$.
\item By $\phi(n)$ we denote Euler's phi function.
\item The notation $g^t$ denotes the conjugation $t^{-1}gt$.
\item For any even $n \geq 2$, the notation $D_{n}$ denotes the dihedral group of order $n$.
\item For a group element $g \in G$, we denote its order by $o(g)$. For an integer $n \geq 1$, we denote its multiplicative order modulo $m$ by $o_m(n)$.
\item We denote the non-negative integers by $\N$, and the positive integers by $\N_0$.
\end{itemize}

\section{The block Virtual Structure Problem}\label{sectie block VSP}
The overarching spirit of this paper is to determine properties of a finite group which are determined by its irreducible representations over a number field. For example, one could wonder which groups are fully determined by certain interesting predescribed conditions on them (such as property $\Mexc$):

\begin{problem*}[Block Virtual Structure problem]\label{block VSP}
    Let $\mc{P}$ be a property. Classify the group algebras $FG$ such that for a set of primitive central idempotents $e \in \PCI (FG)$ (depending on $\mc{P}$), $FGe$ has property $\mc{P}$.
\end{problem*}

In the case that $\mc{P}$ is the property of belonging to a class of groups $\mc{G}$ which behaves well with direct products and is constant on commensurability classes, the Block Virtual Structure problem is equivalent to the classical Virtual Structure Problem, as shown by \Cref{virutally G-am imply special components} below. \smallskip

We now first recall the concept of the reduced norm. First, let $A$ be a finite-dimensional central simple algebra over a field $K$ of characteristic $0$ and let $E$ be a splitting field of $A$ (i.e. $A \otimes_K E \cong \Ma_n(E)$ for some $n$). Then the \emph{reduced norm} of $a \in A$ is defined as \[\operatorname{RNr}_{A/K}(a) = \det(1_E \otimes_K a). \] 
Note that $\operatorname{RNr}_{A/K}(\cdot)$ is a multiplicative map, $\operatorname{RNr}_{A/K}(A) \subseteq K$ and $\operatorname{RNr}_{A/K}(a)$ only depends on $K$ and $a \in A$ (and not on the chosen splitting field $E$ or the isomorphism $E \otimes_K A \cong \Ma_{n}(E)$), see \cite[page 51]{EricAngel1}. For a subring $R$ of $A$, define
 \begin{equation}\label{definition SL_1}
  \SL_1(R) = \{\ a \in \U(R)\ |\ \operatorname{RNr}_{A/K}(a) = 1 \ \},
  \end{equation}
 which is a (multiplicative) group.
If $A = \Ma_n(A')$ and $R= \Ma_n(R')$ with $A'$ a finite-dimensional central simple algebra over $K$ and $R'$ a subring of $A'$, then we also write $\SL_1(A)= \SL_n(A')$ and $\SL_1(R) = \SL_n(R')$. Next, if $A = \prod_i \Ma_{n_i}(D_i)$ is semisimple and $h_i$ is the projection onto the $i$\textsuperscript{th} component, then 
\[\SL_1(R):=\{\ a \in R\ \mid \ \forall\ i \colon \operatorname{RNr}_{\Ma_{n_i}(D_i)/\ZZ(D_i)}(h_i( a)) = 1\ \}.\]

 We will use the notation
 \begin{equation}\label{prod G}
 \prod\mathcal{G} := \left\{ \prod_{i} G_i \mid G_i \text{ is finitely generated abelian or in } \mc{G} \right\},
 \end{equation}
 for the class of groups which are formed out of direct products of groups in $\mathcal{G}$ {\it and} abelian groups. Using \cite{KleRio}, together with classical results on unit groups of orders, we obtain the following. We will say that a group is \emph{virtually-$\mc{G}$} if it has a finite index subgroup which belongs to the class $\mc{G}$.

\begin{lemma}\label{virutally G-am imply special components}
Let $A$ be a finite-dimensional semisimple $F$-algebra with $F$ a number field and $\O$ an order in $A$. Let $\mc{G}$ be a class of groups such that being virtually-$\prod \mc{G}$ is closed under commensurability and direct factors\footnote{In other words if $G_1 \times G_2$ is virtually-$\prod \mc{G}$, then both $G_1$ and $G_2$ are virtually-$\prod \mc{G}$.}. 
Then the following hold.
\begin{enumerate}
\item \label{prod G to SL} $\mathcal{U}(\O)$ is virtually-$\prod\mathcal{G}$ if and only if for every $e \in \PCI(A)$, the group $\SL_1(\O e)$ is either virtually-cyclic or virtually-$\mathcal{G}$.
\item\label{virtually Z} $\SL_1(\O e)$ is virtually cyclic if and only if $\SL_1(\O e)$ is finite.
\item\label{not FA then deg} If moreover being in $\mc{G}$ implies not having property\footnote{A group $\Gamma$ does not have Serre's property (FA) if it acts on a tree without fixed vertices. This for instance is the case if $\Gamma$ has a non-trivial decomposition as an amalgamated product or as an HNN extension.} (FA), then the degree of $A e$ is at most $4$.
\end{enumerate}
\end{lemma}

It was proven by Kleinert that $\SL_1(\O)$ (for $\O$ an order in a finite-dimensional simple $\Q$-algebra $A$) is finite if and only if $A$ is a field or a totally definite quaternion algebra, see \cite{Kleinert} or \cite[Proposition 5.5.6]{EricAngel1}. Moreover, a classification of the rings of integers $R$ and finite groups $G$ such that $\U (R G)$ is virtually cyclic was obtained in \cite[Proposition 1.4]{J1}.

\begin{remark}\label{remark on vsp lemma}
    If one defines the set $\prod \mc{G}$ as simply being the products $\prod_i G_i$ with $G_i \in \mc{G}$, then \Cref{virutally G-am imply special components} still holds. In fact in that case one has in \eqref{prod G to SL} that $\SL_1(\O e)$ is necessarily virtually-$\mc{G}$.

    The proof of \Cref{virutally G-am imply special components} will show that the same conclusion holds if the class of groups that are virtually-$\prod \mc{G}$ has the following weaker direct factor property: if $\prod_{e \in \PCI(A)} \SL_1(\O e) \times \U(\ZZ(\O))$ is virtually-$\prod \mc{G}$, then each $\SL_1(\O e)$ is virtually-$\prod \mc{G}$. Or more abstractly, thanks to \cite{Kleinert}, it is enough to obtain the property for direct products of the form $\prod_{i=1}^m \Gamma_i \times \Gamma$ with $\Gamma_i$ non-abelian virtually indecomposable and $\Gamma$ a finitely generated abelian group.
\end{remark}

\begin{example}
Let $\mc{G}$ be the class consisting of groups $\Gamma$ having a (potentially trivial) normal subgroup $N$ which is free such that $\Gamma/N$ is also free. Then it was proven in \cite[Theorem 2.1]{FreebyFree} that $\mc{G}$ satisfies the hypothesis from \Cref{virutally G-am imply special components}. 

Another example is the class of groups having infinitely many ends. By Stalling's theorem this is equivalent to having an amalgamated or HNN splitting over a finite group. In \cite[Lemma 2.1]{J1} it is proven that this class also satisfies the hypothesis from \Cref{virutally G-am imply special components}. 
\end{example}

In the remainder of the paper we will sometimes use the following fact without further notice. This result is somehow folklore, but for convenience of the reader we sketch the proof.
\begin{lemma}\label{SL order commensurable}
Let $A$ be a finite-dimensional semisimple $F$-algebra with $F$ a number field. If $\O_1$ and $\O_2$ are orders in $A$, then $\SL_1(\O_1)$ and $\SL_1(\O_2)$ are commensurable. 
\end{lemma}
\begin{proof}
We know that $\U(\O_1)$ and $\U(\O_2)$ are commensurable. Therefore \cite[Proposition 5.5.1]{EricAngel1} implies that the subgroups $\langle \SL_1(\O_j), \U(\ZZ(\O_j)) \rangle$ for $j=1,2$ are commensurable. Furthermore, \cite[Corollary 5.5.3]{EricAngel1} implies that $\langle \SL_1(\O_j), \U(\ZZ(\O_j)) \rangle$ is commensurable with $\SL_1(\O_j) \times \U(\ZZ(\O_j))$. From this we conclude that the $\SL_1(\O_j)$, $j = 1,2$ are necessarily commensurable.
\end{proof}

\begin{proof}[Proof of \Cref{virutally G-am imply special components}]
The sufficiency in \eqref{prod G to SL} is clear. Hence assume that $\U(\O)$ is virtually-$\prod \mc{G}$. Recall \cite[Proposition 5.5.1]{EricAngel1}, which says that $\langle \SL_1(\O), \U (\ZZ (\O)) \rangle$ is a subgroup of finite index in $\U(\O)$ and moreover that $\SL_1(\O) \cap  \U (\ZZ (\O))$ is finite. Therefore, $\SL_1(\O) \times \U (\ZZ (\O))$ is commensurable with $\U (\O)$. Furthermore, $\SL_1(\O)$ is commensurable with $\prod_{e \in \PCI(A)} \SL_1(\O e)$  (see also \cite[Corollary 5.5.3]{EricAngel1}). Hence by the assumptions on $\mc{G}$, we have that $\prod_{e\in \PCI(A)}\SL_1(\O e) \times \U (\ZZ (\O))$ and subsequently $\SL_1(\O e)$ is virtually-$\prod \mc{G}$ for each $e \in \PCI(A)$. Now if $Ae$ is a field or a totally definite quaternion algebra, then $\SL_1(\O e)$ is finite by \cite{Kleinert}. For other simple algebras $Ae$, \cite[Theorem 1]{KleRio} says that $\SL_1(\O e)$ is virtually indecomposable, meaning that if $\SL_1(\O e)$ is commensurable with a direct product $\Gamma_1 \times \Gamma_2$, then $\Gamma_1$ or $\Gamma_2$ is finite. Therefore, $\SL_1(\O e)$ being virtually-$\prod \mc{G}$ implies that it is either virtually cyclic or virtually-$\mc{G}$, as desired.

For part \eqref{virtually Z} we may assume without loss of generality that $A$ is simple. It follows from \cite[Theorem 4.1]{JanJesTem} that $\U(\O)$ contains a free group if and only if $A$ is not a totally definite quaternion algebra. In particular, combined with \cite[Proposition 5.5.6]{EricAngel1}, \eqref{virtually Z} follows.

Now we consider \eqref{not FA then deg}. By the first part, $\SL_1(\O e)$ is either virtually cyclic or virtually-$\mc{G}$. Part \eqref{virtually Z} and \cite[Proposition 5.5.6]{EricAngel1} moreover imply that the former case only happens if $Ae$ is a field or a totally definite quaternion algebra. In particular the degree of $Ae$ is at most $2$. Hence suppose that $\SL_1(\O e)$ is virtually-$\mc{G}$. In that case, the desired conclusion follows from \cite[Theorem 3]{KleRio}.
\end{proof}

\section{Finite groups with only exceptional higher simple components}\label{section description components}

In this section we first show in \Cref{remark on larger field} how property $\Mexc$ behaves over different coefficient fields. Afterwards, in \Cref{broad form finite grps}, we describe which exceptional algebras can be components of a group algebra with $\Mexc$. The possibilities turn out to be quite limited.

\subsection{The role of the coefficient field}\addtocontents{toc}{\protect\setcounter{tocdepth}{2}}

Considering the role of the coefficient field, we immediately note the following.

\begin{theorem}\label{remark on larger field}
        Let $G$ be a finite group and $F\subseteq L$ fields of characteristic $0$. If $LG$ has $\Mexc$, then $F G$ also has $\Mexc$. If moreover $\Q \subsetneq L$ and $LG$ has a non-division algebra component, then 
        \[LG \text{ has } \Mexc \Leftrightarrow 
\left\lbrace
\begin{array}{l}
\Q G \text{ has } \Mexc, \\
L = \Q(\sqrt{-d}) \text{ for some square-free } d\in \N,\\
\mc{Z}(\Q Ge) \subseteq L \text{ for all } e\in \PCI(\Q G)  \text{ s.t. } \Q Ge \text{ is not a division algebra. }
\\
\end{array}
\right.\]
\end{theorem}

This shows the importance of first understanding the case of rational group algebras. In the companion paper \cite{CJT}, where the finite groups such that $FG$ has $\Mexc$ will be classified, the bulk of the work will be dedicated to the case $F =\Q$. 
\begin{remark}\label{only division alg comp remark}
Finite groups $G$ for which $\Q G$ has only division algebra components have been classified in \cite[Theorem 3.5]{SehNilpotent}. Concretely, $\Q G$ is a direct product of division algebras if and only if $G$ is abelian or $G$ is of the form $A \times C_2^{n} \times Q_8$ with $A$ an abelian group of odd order such that the multiplicative order of $2 \in \Q$ is odd modulo $|A|$.

If $F$ is any field of characteristic $0$, one can show, using the aforementioned classification, that $FG$ is a direct sum of division algebras exactly when $\Q G$ is and $F$ is not a splitting field of $\qa{-1}{-1}{\Q}$. Therefore, we may in the sequel of the paper suppose that $FG$ has a matrix component. 
\end{remark}

\begin{proof}[Proof of \Cref{remark on larger field}]
Recall that there is a bijection between the simple components of $LG$ and the absolutely irreducible characters $\Irr (G)$ of $G$. For $\chi \in \Irr(G)$ denote the associated component of $LG$ by $A_L(\chi)$. Now, following \cite[Theorem 3.3.1]{EricAngel1}, there is an isomorphism of $F(\chi)$-algebras
\begin{equation}\label{change coeff in component}
A_L(\chi) \cong L(\chi)\ot_{F(\chi)}A_F(\chi),
\end{equation}
where $L(\chi)$ denotes the character field associated to $\chi$.

If $L=F$ there is nothing that needs to be proven, so let $F \subsetneq L$. Note that the result from \emph{loc. cit.} implies that the reduced degree of $A_F(\chi)$ is less than or equal to that of $A_L(\chi)$. Hence if $LG$ has only division algebra components, then so does $FG$. Furthermore, if $LG$ has $\Mexc$, then all matrix components of $FG$ also have reduced degree two.

Next suppose that $A_F(\chi)$ is a matrix component and hence so is $A_L(\chi)$. Then from \eqref{change coeff in component} it follows that $L(\chi) = \mc{Z}(A_L(\chi))$ contains $F(\chi)$ in its centre. As $LG$ has $\Mexc$, this means that \[[F(\chi):\Q] \mid [L(\chi):\Q] \mid 2.\] Consequently, either $F(\chi) = \Q$ or $F(\chi) = L(\chi)$. From all this we now conclude that $FG$ must also have $\Mexc$, which proves the first statement. 

Now suppose that $LG$ has $\Mexc$ and a non-division algebra component. By the first part $\Q G$ must also have $\Mexc$. Moreover, since $L \subseteq \mathcal{Z}(A_{L}(\chi))$ for all $\chi \in \Irr(G)$, the $\Mexc$ assumption implies that $L = \Q(\sqrt{-d})$ for some $d \in \N$. Since we assume $F \neq \Q$, the number $d$ can be taken square-free. Now, as $LG$ has $\Mexc$, we have that for each $e \in \PCI(LG)$ such that $LG e$ is not a division algebra, $\ZZ(LGe)$ is $\Q$ or an imaginary quadratic extension of $\Q$. Consequently, $L(\chi, \sqrt{-d}) = L(\sqrt{-d})$ for all $\chi \in \Irr(G)$, yielding simple components of reduced degree at least  $2$. In other words, $L(\chi) \subseteq \Q(\sqrt{-d})$ for all such $\chi$, showing the necessary conditions. The converse clearly also holds, using \eqref{change coeff in component}.
\end{proof}

\subsection{Division algebra factors in case of property \texorpdfstring{$\Mexc$}{(Mexc)}}

In this section we determine which simple algebras can occur as a simple component of a group algebra $FG$ with property $\Mexc$. The main result will be that the possible division algebras that can arise are restricted. Recall that associated to a central idempotent $e \in \PCI(FG)$, there is an epimorphism from $FG$ to the simple component $FGe$ which restricts to
\begin{equation}\label{map associated to central id}
\varphi_e \colon G \rightarrow Ge, \ \ g \mapsto ge.
\end{equation}
Thus $Ge$ denotes a finite subgroup of $\mathcal{U}(F G)e$. It is classical, see \Cref{Prop components preserved for quotient}, that 
\begin{equation}\label{cover via Ge}
\mathcal{C}(F G) =  \displaystyle\bigcup\limits_{e \in \PCI(F G)} \mathcal{C}(F[Ge]).
\end{equation}

This motivates the following definition.

\begin{definition}
   Let $G$ be a finite group and $F$ a field with $\Char(F) \nmid |G|$. We say that $FG$ is \emph{covered by $H_1, \ldots, H_{\ell}$} if $\mc{C}(FG) = \bigcup_{i=1}^{\ell} \mc{C}(FH_i)$.
\end{definition}
Thus \eqref{cover via Ge} can be reformulated as saying that $FG$ is covered by $\{Ge \mid e \in \PCI(FG) \}$. To obtain the restrictions on the simple components of $FG$, we will describe which finite groups can be isomorphic to $Ge$ with $e \in \PCI( F G)$ when $FG$ has $\Mexc$. In \cite[Proposition 4.1]{CdRZ} such a study was essentially already done. The proof in \textit{loc. cit.} however contains some flaws, which we correct. 
Note that if $FG$ has $\Mexc$, then each $Ge$ is a subgroup of either a division algebra or of an exceptional matrix algebra. We will call finite subgroups of such simple algebras \emph{low rank spanning groups}.

For the upcoming study in \Cref{section kleinian groups} of components whose $\SL_1$ are discrete in $\SL_4(\C)$, we will need to understand the simple components of $FG$ in case it satisfies following slightly weaker version of $\Mexc$: 
\begin{equation}
\tag{$\mathrm{wM_{exc}}$} 
\begin{array}{c}
\text{ All } F Ge \cong \Ma_n(D), \text{ with } n \geq 2, \text{ are exceptional or} \Ma_4(\Q)  \\[0.1cm]
\text{ and } FG \text{ is covered by low rank spanning groups.}
\end{array}
\end{equation}

If $FG$ has no $\Ma_4(\Q)$ as component, then it will follow from \Cref{Prop components preserved for quotient} that $FG$ has $\Mexc$ if and only if it has $\wMexc$.

As mentioned in \Cref{only division alg comp remark}, it is known when $FG$ is a direct product of division algebras. Therefore the assumption that $\Q G$ has at least one matrix component in the next result, does not harm the generality. In the following statement, the notation $\cd(G)$ denotes the set of the degrees of the irreducible characters of $G$.

\begin{theorem}\label{broad form finite grps}
Let $G$ be a finite group and suppose that $\Q G$ has a matrix component. If $\Q G$ has $\wMexc$, then the following hold.
\begin{enumerate}
\item The $1 \times 1$ components of $\Q G$ are either fields or quaternion algebras. More precisely, the non-commutative possibilities are: 
\[\left\{\qa{\zeta_{2^t}}{-3}{\Q(\zeta_{2^t})}, \qa{-1}{-1}{\Q}, \qa{-1}{-3}{\Q}, \qa{-1}{-1}{\Q(\sqrt{2})}, \qa{-1}{-1}{\Q(\sqrt{3})} \mid t\in \N_{\geq 3}\right\}.\]
\item\label{item: Mexc versus wMexc} $\Q G$ has $\Mexc$ if and only if the Fitting subgroup $\Fit(Ge)$ has index at most $2$ and nilpotency class at most $3$ for each $e \in \PCI( \Q Ge)$ such that $\Q Ge$ is not a division algebra.
\item If $\Q G$ has $\Mexc$, then the only possible exceptional matrix algebras\footnote{Note that by \cite[Theorem 3.5]{EKVG} the only exceptional matrix algebra not appearing is $\Ma_2(\qa{-2}{-5}{\Q})$.} are
\[\left\{\Ma_2(\Q(\sqrt{-d})), \Ma_2(\qa{-1}{-1}{\Q}), \Ma_2(\qa{-1}{-3}{\Q}) \mid d = 0, 1, 2, 3\right\}.\]
\end{enumerate}
In particular, $\cd (G) \subseteq \{ 1,2,4\}$ and $\deg (\Q G e) \mid 4$ for every $e \in \PCI(\Q G)$.
\end{theorem}

\begin{remark}
In the proof we will obtain that for $C_3 \rtimes C_{2^n}$, with $C_{2^n}$ acting by inversion, the rational group algebra has all non-division simple components of the form $\Ma_2(\Q)$ and $\Ma_2(\Q(i))$. Furthermore, it has a division component of the form $\qa{\zeta_{2^{n-1}}}{-3}{\Q(\zeta_{2^{n-1}})}$. Hence it gives an example of a group that has all matrix components of the form $\Ma_2(\Q(\sqrt{-d}))$ with $d \in \N$ but which does not satisfy the stronger property considered in \cite{FreebyFree}, where they require that $\Q G$ has no exceptional division components.

One interesting class of groups for which $\Q G$ has no exceptional division components is the class of {\it cut groups}, see \cite[Proposition 6.12]{BJJKT}. Recall that a group is called \emph{cut} if $\ZZ (\U (\Z G))$ is finite.
\end{remark}

For a general coefficient field, we will deduce the following.

\begin{corollary}\label{arbitrary coefficients}
Let $G$ be a finite group and $F$ a field of characteristic $0$ different from $\Q$, such that $FG$ has $\Mexc$ and at least one non-division component. Then the isomorphism types of the non-commutative components of $FG$ are amongst the following:
\[
\left\{ \qa{\zeta_{2^t}}{-3}{\Q(\zeta_{2^t})},\qa{-1}{-3}{\Q(\sqrt{-2})}, \Ma_2(F) \mid t \in \N_{\geq 3}\right \}.\]
Moreover, 
\begin{itemize}
    \item $\qa{-1}{-3}{\Q(\sqrt{-2})} \in \mc{C}(FG)$ if and only if $F = \Q(\sqrt{-2})$ and $G$ maps onto $C_3 \rtimes C_4$,
    \item  $ \qa{\zeta_{2^t}}{-3}{\Q(\zeta_{2^t})}\in \mc{C}(FG)$ if and only if $F = \Q(\sqrt{-1})$ and $G$ maps onto $C_3 \rtimes C_{2^n}$ with action by inversion and $n \geq t$.
\end{itemize}
\end{corollary}

\subsubsection{Background on strong Shoda pairs and simple components}

We need to recall some methods to construct primitive central idempotents of $\Q G$. These methods were introduced by Olivieri--del R\'io--Sim\'on \cite{OlRiSi}, see \cite[Chapter 3]{EricAngel1} for a good introduction. 

A tuple $(H,K)$ is called a \emph{strong Shoda pair}  (\emph{SSP} in short) when \[K \leq H \unlhd N_G(K),\] such that $H/K$ is cyclic and a maximal abelian subgroup of $N_G(K)/K$, and such that the $G$-conjugates of 
\begin{equation}\label{Def epsilon idempotents}
\epsilon (H,K) :=  \prod\limits_{{M}/{K} \in \mathcal{M}(H/K)} (\wh{K} - \wh{M}),
\end{equation}
(where $ \mathcal{M}(H/K)$ denotes the set of the non-trivial minimal normal subgroups of $H/K$) are orthogonal. Associated to such a tuple is a primitive central idempotent of $\Q G$, given by
\begin{equation}\label{Def e(G,H,K)}
e(G,H,K) = \sum_{t \in \mathcal{T}} \epsilon(H,K)^{t},
\end{equation}
where $\mathcal{T}$ is a right transversal of $\Cen_G(\epsilon(H,K))$ in $G$. 

The following is a combination of \cite[Proposition 3.4.1, Theorems 3.4.2 \& 3.5.5 and Problem 3.5.1]{EricAngel1} and \cite[Lemma 3.4]{JM}.

\begin{theorem}[\cite{OlRiSi}]\label{form of SSP idempotent result}
With notations as above, $e :=e(G,H,K)$ is a primitive central idempotent of $\Q G$ if $(H,K)$ is a strong Shoda pair. Moreover, in that case, $\Cen_G(\epsilon(H,K)) \cong N_G(K)$, $\ker(\varphi_{e}) = \text{core}_G(K) = \bigcap_{g \in G} K^g$ and 
\[\dim_{\Q} \Q G e =  [G:H] [G:N_G(K)] \phi([H:K]). \]
Furthermore,
$\Q G e \cong \Ma_{[G:N_G(K)]}\left(\Q(\zeta_{[H:K]}) * N_G(K)/H \right)$ for some explicit crossing. In particular, $\deg (\Q Ge) = [G:H]$.
\end{theorem}

The notation $\deg$ denotes {\it the degree} of the central simple algebra $\Q G e(G,H,K))$, i.e. $\deg (A) = \sqrt{n}$ if $A \ot_{\mc{Z}(A)} \C \cong \Ma_n(\C)$.

Finally, we will mainly use the theory of strong Shoda pairs for $G$ metabelian, in which case the SSPs are easy to describe.

\begin{theorem}[{\cite[Theorem 3.5.12]{EricAngel1}}]
\label{SSP in metabelian}
Let $G$ be a finite metabelian group and let $B$ be a maximal abelian subgroup of $G$ containing $G'$. Let $K \leq G$ be such that $C' \leq K \leq C$ for some subgroup $C$ of $G$ with $B \leq C \leq G$. Then $(H,K)$ is a strong Shoda pair if and only if the following hold:
\begin{enumerate}
    \item $H/K$ is cyclic,
    \item $H$ is maximal in the set $\{ C \leq G \mid B \leq C \text{ and } C' \leq K \leq C\}$.
\end{enumerate}
\end{theorem}

\subsubsection{Proof of the main result}

The proof of \Cref{broad form finite grps} requires a finer understanding of the quotient groups $Ge$, with $e\in \PCI(\Q G)$.

\begin{lemma}\label{Prop components preserved for quotient}
    For any finite group $G$, a normal subgroup $N \leqslant G$, a field $F$ with $\Char(F) \nmid |G|$ and $e \in \PCI(F G)$, the following holds:
    \[\mathcal{C}(F[G/N]) \subseteq \mathcal{C}(F G)\,  \text{ and } \, \mathcal{C}(F[Ge]) \subseteq \mathcal{C}(F G). \] 
    Consequently, 
    \[\mathcal{C}(F G) =  \displaystyle\bigcup\limits_{e \in \PCI(F G)} \mathcal{C}(F[Ge]).\]
    In particular, properties $\Mexc$ and $\wMexc$ are inherited by quotients.
\end{lemma}
\begin{proof}
The first claim follows from the fact that $F[G/N]$ is a semisimple subalgebra of $FG$. Indeed, it is immediately semisimple since it is a group algebra, and a straightforward calculation shows that $F[G/N] \cong F G \widetilde{N} \leq F G$ with $\widetilde{N}$ the central idempotent $\frac{1}{|N|} \displaystyle\sum\limits_{n\in N} n$. The second inclusion follows from the first since the group $Ge$ is an epimorphic image of $G$. The last statement is now also immediate, since every simple component of $F G$ corresponds to a primitive central idempotent $e \in \PCI(F G)$.
\end{proof}

Using \Cref{Prop components preserved for quotient}, the proof of \Cref{broad form finite grps} reduces to a study of the subgroups of division algebras classified by Amitsur \cite{Amitsur} and the groups spanning an exceptional matrix algebra classified in \cite{EKVG}.

\begin{proof}[Proof of \Cref{broad form finite grps}.]
For a group $G$ having $\wMexc$, the set $\PCI(\Q G)$ naturally decomposes into $\PCI_1 := \{ e \mid \Q Ge \cong D \}$ and $\PCI_{2} := \{ e \mid \Q Ge \cong \Ma_2(D) \}$, where $D$ always signifies a rational division algebra, and $\PCI_{4} := \{ e \mid \Q Ge \cong \Ma_4(\Q) \}$. Hence, with \Cref{Prop components preserved for quotient} in mind, for the first statement it suffices to analyse the components possibly appearing in $\Q [Ge]$ for $e \in \PCI_i$ with $i=1,2,4$. 

We start with $\PCI_{2}$. The finite subgroups $\mathcal{G}$ of $\GL_2(\Q(\sqrt{-d}))$ or $\GL_2(\qa{-a}{-b}{\Q})$ with $a,b,d \in \N$ with the property that $\Span_{\Q}\{\mc{G}\}$ is the respective $\Ma_2(\cdot)$
have been classified\footnote{In \cite[Table 2]{EKVG} a group was missing, see \cite[Appendix A]{BJJKT} for a complete list.} in \cite[Theorem 3.7]{EKVG}. This classification consists of $55$ groups and in particular the groups $Ge$ for $e \in \PCI_2$ are among them. The simple components of $\Q \mc{G}$ may be computed using e.g. the Wedderga package in GAP, see \Cref{tableB} in \Cref{appendix section} for the result. This moreover shows that for groups $Ge$ with $e\in \PCI_2$, the property $\Mexc$ is equivalent to the weaker property that each non-division component has reduced degree $2$. A case-by-case verification also shows that each of these groups $Ge$ contains an abelian normal subgroup of index a divisor of $4$. Furthermore, as can be seen in the table, the only $1 \times 1$ components appearing are 
\[\Q, \Q(\zeta_3), \Q(i), \Q(\zeta_8), \Q(\zeta_{12}), \qa{-1}{-1}{\Q}, \qa{-1}{-3}{\Q}, \qa{-1}{-1}{\Q(\sqrt{2})}, \qa{-1}{-1}{\Q(\sqrt{3})}.\] Inspection of the table also shows that the remaining statements hold for such groups. 

Next we consider the case $e\in \PCI_1$. A similar reasoning applies. Indeed, the finite subgroups (such as $Ge$ for $e \in \PCI_1$) of rational division algebras have been classified by Amitsur in \cite{Amitsur}. We will use the rephrasing from \cite[Theorem $2.1.4$]{ShiWeh} which asserts that they are: 
\begin{enumerate}[(a)]
\item a $\textbf{Z}$-group, i.e. having cyclic Sylow-subgroups, with restrictions listed later on in the proof.
\item \begin{enumerate}[(i)]
\item the binary octahedral group $O^{\ast} \cong \SU_2(\F_3)$ of order $48$: \[\left\{ \pm 1, \pm i, \pm j, \pm ij, \frac{\pm 1 \pm i \pm j \pm ij}{2}\right\} \cup \left\{ \frac{\pm a \pm b}{\sqrt{2}} \mid a,b \in \{ 1,i,j,ij\}\right\}.\]
\item $C_m \rtimes Q$, where $m$ is odd, $Q$ is quaternion of order $2^t$ for some $t \geq 3$, an element of order $2^{t-1}$ centralises $C_m$ and an element of order $4$ inverts $C_m$.
\item $M \times Q_8$, with $M$ a $\textbf{Z}$-group of odd order $m$ and the (multiplicative) order of $2$ mod $m$ is odd.
\item $M \times \SL_2(\mathbb{F}_3)$, where $M$ is a $\textbf{Z}$-group of order $m$ coprime to $6$ and the (multiplicative) order of $2$ mod $m$ is odd.
\end{enumerate}
\item $\SL_2(\mathbb{F}_5)$.
\end{enumerate}

None of the groups $\SU_2(\F_3) = O^{\ast}$, $\SL_2(\mathbb{F}_3)$ and $\SL_2(\mathbb{F}_5)$ have $\wMexc$. Indeed, using\footnote{The \textsc{SmallGroupID's} of these three groups are respectively \texttt{[48,28]}, \texttt{[24,3]} and \texttt{[120,5]}.} the Wedderga package in GAP one verifies that $\Ma_3(\Q)$ is a simple component over $\Q$ of $O^{\ast} \cong \SU_2(\mathbb{F}_3)$ and $\SL_2(\mathbb{F}_3)$, and moreover that $\Ma_5(\Q)$ is a simple component over $\Q$ for $\SL_2(\mathbb{F}_5)$. Consequently, they cannot be epimorphic images of the group $G$. Hence:\vspace{-0,1cm}
\begin{center}
The cases (b)(i), (b)(iv) and (c) {\it do not} appear as groups $Ge$ for $e \in \PCI_1$ with $\Q[G]$ $\Mexc$.
\end{center}\vspace{-0,1cm}

The groups in (b)(ii) are dicyclic groups of order $2^tm$, i.e. $Dic_{4n}$ with $n = 2^{t-2}m$ and $t\geq 3$. When $n$ odd, the dicylic groups coincide with case (B) in the family of $\textbf{Z}$-groups (see below). Therefore, consider a general dicyclic group:
\begin{equation}
\label{dicyclic groups}
    Dic_{4n} = \langle a,b \mid a^{2n}=1 , b^2 = a^n, b^{-1}ab = a^{-1} \rangle.
\end{equation}
This is a metabelian group and hence its strong Shoda pairs are described by \Cref{SSP in metabelian}, which we apply now. Note that the derived subgroup $Dic_{4n}'$ is given by $\langle a^2 \rangle$, and $\langle a \rangle$ is the maximal abelian subgroup of $Dic_{4n}$ containing it. For any SSP $(H,K)$, one has that $\langle a \rangle \subseteq H$. In other words, $H = \langle a \rangle $ or $H = Dic_{4n}$. If $H= Dic_{4n}$ then the simple component associated to $(H,K)$ is a field, by \cite[Lemma 2.4]{JespersSun}. Via \Cref{SSP in metabelian}, it is a direct verification that for $d \mid 2n$, the tuple $(\langle a \rangle, \langle a^d \rangle)$ is a SSP if and only if $d \neq 1,2$. Note that $K$ is normal in $Dic_{4n}$. Hence the associated primitive central idempotent is $\epsilon (\langle a \rangle, \langle a^d \rangle)$. Now, in \cite[Example $3.5.7$]{EricAngel1}, it is noted that (for $n$ not necessarily even)
\begin{equation}\label{bad component dicyclic}
\Q Dic_{4n} \epsilon (\langle a \rangle, \langle a^d \rangle) \cong 
\Ma_2(\Q (\zeta_d + \zeta_d^{-1})) \quad \text{ if } d \mid n \text{ and } d\nmid 2,
\end{equation}
where $\zeta_d$ denotes a complex primitive $d$-th root of unity. Since $\Q (\zeta_d + \zeta_d^{-1})= \Q(\Re(\zeta_d))$ and $\Re(\zeta_d) = \cos \frac{2\pi}{d} \in \R \setminus \Q$, Niven's theorem implies that $\Ma_2(\Q (\zeta_d + \zeta_d^{-1}))$ is exceptional\footnote{Recall that $[\Q (\Re(\zeta_d)) : \Q]  = \phi(d) /2$. In particular, in contrast to the case $e \in \PCI_2$, it can happen that all non-division components are exceptional without the group having $\Mexc$. Moreover, it can happen that all non-division components are of the form $\Ma_2(F)$ with $F$ a quadratic extension of $\Q$. As shown by \eqref{bad component dicyclic}, this  holds for $Dic_{4n}$ with $n = 5,10,8,12$.} if and only if $\phi(d) \leq 2$. The latter is equivalent to $d \in \{1,2,3,4,6\}$. In conclusion, if $Dic_{4n}$ is non-abelian and has $\wMexc$, then it must be isomorphic to one of the groups 
\[Dic_{4.2} = Q_8,\quad  Dic_{4.4} = Q_{16}, \quad Dic_{4.6}= C_3 \rtimes Q_8 \quad \text{ or } \quad Dic_{4.3}= C_3 \rtimes C_4.\] The first three groups are in the family (b)(ii). Moreover, these three groups each have $\Mexc$, since their Wedderburn--Artin components are\footnote{Calculated using the Wedderga package in GAP (the groups have \textsc{SmallGroupID's} respectively \texttt{[8,4]}, \texttt{[16,9]} and \texttt{[24,4]}).}: 
\begin{align*}
\mc{C}(\Q Q_8)   = \{ \Q , \qa{-1}{-1}{\Q} \},  \qquad   & \mc{C}(\Q Q_{16})  = \{\Q, \qa{-1}{-1}{\Q(\sqrt{2})}, \: \Ma_2(\Q)\}, \quad \text{ and } \\
\mc{C}(\Q[C_3 \rtimes Q_8])  = \{\Q , & \qa{-1}{-1}{\Q},  \qa{-1}{-1}{\Q(\sqrt{3})},\: \Ma_2(\Q)\}.
\end{align*}

Before we consider case (b)(iii) from above, we will discuss (a), the $\textbf{Z}$-groups (and in particular $Dic_{4.3}$). They have also been classified, see \cite[Theorem $2.1.5$]{ShiWeh}. They are the following: 
\begin{enumerate}[(A)]
\item Finite cyclic groups,
\item $C_m \rtimes C_4$, where $m$ is odd and $C_4$ acts by inversion,
\item $G_0 \times G_1 \times \ldots \times G_s$, with $s \geq 1$, $\gcd(|G_i|,|G_j|)=1$ for all $0 \leq i \neq j \leq s$ and $G_0$ is the only cyclic subgroup amongst them. Furthermore each of the $G_i$, for $i \neq 0$, is of the form \[C_{p^a} \rtimes \left(C_{q_1^{b_1}} \times \ldots \times C_{q_r^{b_r}}\right), \]for $p, q_1, \ldots,q_r$ distinct primes. Moreover, each of the groups $C_{p^a} \rtimes C_{q_j^{b_j}}$ is non-cyclic (i.e. if $C_{q_j^{\alpha_j}}$ denotes the kernel of the action of $C_{q_j^{b_j}}$ on $C_{p^a}$, then $\alpha_j \neq b_j$) and satisfies the following properties: \begin{enumerate}[(i)]
\item $q_jo_{q_j^{\alpha_j}}(p) \nmid o_{\frac{|G|}{|G_i|}}(p)$.
\item one of the following is true: 
\begin{center}
    \begin{itemize}
        \item $q_j=2$,  $p \equiv -1 \mod 4$,  and  $\alpha_j = 1$,
        \item $q_j=2$,  $p\equiv -1 \mod 4$, and $2^{\alpha_j+1} \nmid p^2-1$,
        \item $q_j=2$,  $p\equiv 1 \mod 4$, \hspace{4.5pt} and  $2^{\alpha_j + 1} \nmid p-1$, 
        \item $q_j>2$, \hspace{2.38cm} and $ q_j^{\alpha_j+1} \nmid p-1$.
    \end{itemize}
\end{center}
\end{enumerate}
Here $o_m(q)$ denotes the order of $q$ modulo $m$.
\end{enumerate}

It is clear that the cyclic groups have $\Mexc$ since $\Q C_n$ is abelian. Moreover, by the theorem of Perlis--Walker \cite[Theorem 3.3.6]{EricAngel1}, $\mathcal{C}(\Q C_n) = \{\Q(\zeta_d) \mid d \text{ divides } n \}$.

Case (B), i.e $Dic_{4n}$ with $n$ odd, is treated in (\ref{bad component dicyclic}). The conclusion was that the only possible (non-abelian) such group having $\wMexc$ is $C_3 \rtimes C_4$. In this case it has $\Mexc$, since
\begin{equation}\label{QG of order 12}
\mc{C}(\Q[C_3 \rtimes C_4])) = \{ \Q, \Q(i), \qa{-1}{-3}{\Q}, \Ma_2(\Q) \}.
\end{equation}

Next we consider case (C). We first show that $\wMexc$ implies $2 \mid |G_i|$, for $1 \leq i \leq s$ and hence $s=1$ by the coprimeness condition. For this, we consider $A_i = \prod_{j=1}^s C_{q^{\alpha_j}}$, the kernel of the action in the decomposition of $G_i$ into a semidirect product. Then \[B := G_i / A_i \cong C_{p^{a}} \rtimes C_{q_1^{k_1}\cdot \ldots \cdot q_s^{k_s}},\] where $k_j := b_j - \alpha_j > 0$ and the action is non-trivial and faithful. Denote by $x$ and $y$ the respective generators of the factors of $B$, i.e. $B = \langle x \rangle \rtimes \langle y \rangle$. By \Cref{Prop components preserved for quotient}, the group $B$ also has $\Mexc$. Note that $C_{p^{a}}=\langle x \rangle$ is a maximal abelian subgroup of $B$ containing $B'$. Now using \Cref{SSP in metabelian}, it is a direct verification that $(H,K)= (\langle x \rangle, 1)$ is a SSP of $B$. Moreover, $\Q B e(G,\langle x \rangle, 1) \cong \Q(\zeta_{o(x)}) * \langle y \rangle$ for some explicit crossing (see \cite[Remark 3.5.6]{EricAngel1}) that would show that the component is not a division algebra. Hence, as $B$ has $\wMexc$, it is an exceptional matrix algebra or $\Ma_4(\Q)$ and thus $\dim_{\Q} \Q B e(G,\langle x \rangle, 1) \mid 16$.  Now using \cite[Lemma 3.4]{JM}, we compute that 
\[\dim_{\Q} \Q B e(G,\langle x \rangle, 1) =[G:\langle x \rangle]\: \phi (o(x)) = q_1^{k_1} \cdot \ldots \cdot  q_s^{k_s} \, p^{a-1}(p-1).\]
Combining both facts with the fact that $p$ and $q_i$ are pairwise distinct primes, we obtain that $s=1$, $q_1 = 2$ and $p^{a} = 3$. Thus $B \cong C_3 \rtimes C_{2^{k_1}}$. Furthermore, since the action is faithful, we obtain that $k_1 = 1$, i.e. $B \cong C_3 \rtimes C_2$ with action by inversion. Consequently, $G_1 \cong C_3 \rtimes C_{2^{b_1}}$ with the action by inversion (as $\alpha_1 = b_1 -1$). For such groups, using \Cref{SSP in metabelian}, one verifies that, for $b_1 \geq 4$,
\begin{equation}\label{decomp grp with exp comp}
\begin{array}{rcr}
\mc{C}(\Q[C_3 \rtimes C_{2^n}]) = \Big\{ \Q(\zeta_{2^{\ell}}), & \qa{-1}{-3}{\Q}, \qa{\zeta_{2^t}}{-3}{\Q(\zeta_{2^t})}, \:  \Ma_2(\Q), \: \Ma_2(\Q(i)) \\[0,2cm] & \mid 1 \leq \ell \leq n, \, 3 \leq t \leq n-1 \Big\},
\end{array}
\end{equation}
and, for $b_1 \leq 3$:
\begin{align*}
\mc{C}(\Q[C_3 \rtimes C_{2}]) & = \{\Q, \Ma_2(\Q) \}, \qquad \mc{C}(\Q[C_3 \rtimes C_{4}]) = \{\Q, \Q(\sqrt{-1}), \qa{-1}{-3}{\Q}, \: \Ma_2(\Q)\},\\
\mc{C}(\Q[C_3 & \rtimes C_{8}]) = \{\Q, \Q(\sqrt{-1}), \Q(\zeta_8), \qa{-1}{-3}{\Q}, \: \Ma_2(\Q), \: \Ma_2(\Q(\sqrt{-1}))\}.
\end{align*}
Hence we see that such a $G_1 \cong C_3 \rtimes C_{2^{b_1}}$ even have $\Mexc$.

It remains to consider $G_0 \times G_1$ with $G_0 \cong C_m$ cyclic. From the Wedderburn--Artin decomposition of $\Q[G_1]$ above, we see that $\Q[G_0 \times G_1]$ contains as a simple component 
\[\Q(\zeta_m) \otimes_{\Q} \Ma_2(\Q) \cong \Ma_2(\Q(\zeta_m)) .\]
Since $G_0 \times G_1$ is assumed to have $\wMexc$, this implies that $[\Q(\zeta_m): \Q] \leq 2$. The latter happens exactly when $m \in \{1,2,3,4,6 \}$. Now recall that $|G_0|=m$ and $|G_1|=3\cdot 2^{b_1}$ are relatively prime, yielding that $m=1$, since $b_1 \geq 1$. Hence in conclusion,
\begin{equation}
\begin{array}{c}
\text{ {\bf Z}-groups of type  (C) with } \wMexc \text{ are } C_3 \rtimes C_{2^n} \text{ with action by inversion}. \\
\text{ Moreover, they have } \Mexc.
\end{array}
\end{equation}

The last case to handle is (b)(iii), i.e. $M \times Q_8$ with $M$ a ${\bf Z}$-group of odd order $m$ and  the multiplicative order of $2$ modulo $m$ odd. By \Cref{Prop components preserved for quotient}, the group $M$ must have $\wMexc$. Looking at the possible $\mathbf{Z}$-groups of odd order, we see that $M$ must be cyclic. Now recall that $\mc{C}(\Q Q_8) = \{ \Q, \qa{-1}{-1}{\Q} \}$. If $M$ is cyclic, then \[\mc{C}(\Q[M \times Q_8]) = \{ \Q(\zeta_d), \qa{-1}{-1}{\Q(\zeta_d)} \mid d \text{ divides } m\}.\] As $m$ and $o_m(2)$ are odd, all the components are division algebras -- a conclusion that also directly would have followed from \cite{SehNilpotent}. In conclusion\footnote{This conclusion only requires that every non-division component is of the form $\Ma_2(D)$ with $[\ZZ(D):\Q] \leq 2$, not the full strength of $\Mexc$.},
\begin{center}
Groups of type (b)(iii) with $\wMexc$ are given by $C_m \times Q_8$ with $m$ and $o_m(2)$ odd. They all have $\Mexc$, since all components are division algebras.
\end{center}

In summary, with the analysis above we have shown part (1) and (2) from the statement by describing $\prod_{e \in \PCI(\Q G)} Ge$. Note that all allowed groups $Ge$ have been highlighted in the proof. One verifies case-by-case  that the Fitting subgroup $\Fit(Ge)$ has index at most $2$ and class at most $3$ if and only if $Ge$ has $\Mexc$.

Additionally, for the simple algebras $\Ma_n(D)$ allowed by $\wMexc$, we see that $\Ma_n(D) \ot_{\mc{Z}(D)} \C$ is isomorphic to $\Ma_2(\C)$ or $\Ma_4(\C)$. By the first part of the proof, if $\Q Ge \cong D$, then $D \ot_{\Z(D)} \C$ is either $\C$ or $\Ma_2(\C)$. Hence indeed $\cd(G) \subseteq \{ 1,2,4 \}$ and $\deg (\Q Ge) \mid 4$ for all $e \in \PCI(\Q G)$.
\end{proof}

To finish this section we consider coefficient rings larger than $\Q$.

\begin{proof}[Proof of \Cref{arbitrary coefficients}]
By \Cref{remark on larger field} we know that $F = \Q(\sqrt{-d})$ for some square-free $d \in \N$. Moreover, $F \subseteq \ZZ(FGe)$ for each $e \in \PCI(FG)$. In particular if $\Ma_2(D)$ is an exceptional matrix component of $FG$, then $\ZZ(D)=F$. Hence $D$ cannot be a quaternion algebra $\qa{-a}{-b}{\Q}$, since $F \neq \Q$. In conclusion, every non-division component of $FG$ is of the form $\Ma_2(F)$.

\Cref{remark on larger field} also yields that $\Q G$ has $\Mexc$, and hence the possible isomorphism types of components of $\Q G$ are described in \Cref{broad form finite grps}. As $FG \cong F \ot_{\Q} \Q G$, it remains to describe $F \ot_{\Q} D$ for $D$ a non-commutative division algebra component of $\Q G$, say $\Q Ge =D$. Here one needs to be careful to also remember for which low rank spanning groups such $D$'s appear, as it can happen that $FG$ no longer has $\Mexc$, for a low rank spanning group $G$ such that $\Q G$ has $\Mexc$.

It is well-known that $\Q(\sqrt{-d})$ is a splitting field of $\qa{-1}{-1}{\Q}$ for $d = 1, 2,3$ and of $\qa{-1}{-3}{\Q}$ for $d=1,3$. Hence if $D$ is one of these, then $F \ot_{\Q} D$ can only be a division algebra in case $F = \Q(\sqrt{-2})$ and $D = \qa{-1}{-3}{\Q}$. Recall that $D = \Q Ge$ for some $e \in \PCI(\Q G)$. If $D \in \mc{C}(\Q[Ge])$ with $Ge$ a group from \Cref{tableB} in \Cref{appendix section}, then $Ge$ has \textsc{SmallGroupID} \texttt{[24,1]} or \texttt{[36,6]}.

However, over $\Q(\sqrt{-2})$ both groups no longer have $\Mexc$. If $D \in \mc{C}(\Q[Ge])$ with $Ge$ a group classified by Amitsur, then inspecting the options in the proof of \Cref{broad form finite grps}, we see that $Ge \cong C_3 \rtimes C_4$, see \eqref{QG of order 12}. Furthermore, $\Q(\sqrt{-2})[C_3 \rtimes C_4]$ has $\Mexc$.

Next, we consider the division algebras $\qa{-1}{-1}{\Q(\sqrt{c})}$ with $c = 2$ or $3$. Since $\qa{-1}{-1}{\Q}$ splits over $F = \Q(\sqrt{-d})$ with $d=1,2,3$, the same holds for $\qa{-1}{-1}{\Q(\sqrt{c})}$. Therefore,
\[F \ot_{\Q}\qa{-1}{-1}{\Q(\sqrt{c})} \cong \Ma_2(F) \oplus \Ma_2(F(\sqrt{c})).\] Since $\Ma_2(F(\sqrt{c}))$ is not an exceptional matrix algebra, we conclude that such quaternion algebras $\qa{-1}{-1}{\Q(\sqrt{c})}$ cannot be a component of $FG$. 

It remains to consider the division algebra $\qa{\zeta_{2^t}}{-3}{\Q(\zeta_{2^t})}$ with $t \in \N_{\geq 3}$. By \eqref{decomp grp with exp comp} this is a component if and only if $G$ maps onto $C_3 \rtimes C_{2^n}$, where the action is by inversion, and $3 \leq t \leq n-1$. Moreover, $\Ma_2(\Q(i)) \in \mc{C}(\Q[C_3 \rtimes C_{2^n}]))$. Hence it follows that $F = \Q(i)$. Now, it is easy to see that \[\Q(i) \ot_{\Q} \qa{\zeta_{2^t}}{-3}{\Q(\zeta_{2^t})} \cong \qa{\zeta_{2^t}}{-3}{\Q(\zeta_{2^t})} \op \qa{\zeta_{2^t}}{-3}{\Q(\zeta_{2^t},\sqrt{-1}))},\] and that $i \in \Q(\zeta_{2^t})$. In conclusion, $\Q(i)[C_3\rtimes C_{2^n}]$ has $\Mexc$ for all $n \geq 4$, finishing the proof of the statement.
\end{proof}

\section{Homological and topological characterisation of exceptional components}\label{section hom and topo char}

In this section we consider various homological and topological characterisations of exceptional matrix algebras. These characterisations, along with the results in \Cref{section description components}, will be used to give various characterisations of property $\Mexc$.

\subsection{Groups of virtual cohomological dimension 4}\label{subsectie vcd}
Let $\Gamma$ be a discrete group. Then the {\it  cohomological dimension} of $\Gamma$ over the ring $R$ is
\[\hdim_R \Gamma := \min \{ n \mid H^k(\Gamma,M)= 0 \text{ for all } k > n \text{ and } M \in \Mod (R \Gamma) \}.\]
If no such $n$ exists, one says that $\hdim_R \Gamma = \infty$. A usual obstruction to having finite cohomological dimension is torsion in $\Gamma$. However, if $\Gamma$ has a torsion-free subgroup of finite index (e.g. when $\Gamma$ is linear), then each of such finite index subgroups has the same cohomological dimension. Hence, one defines the \emph{virtual cohomological dimension} as:
\[\vcd (\Gamma) := \{ \hdim_{\Z} \Lambda \mid [\Gamma: \Lambda] < \infty \text{ and } \Lambda \text{ torsion-free}\}.\]

The main result of this section is the following characterisation of $\Mexc$.

\begin{theorem}\label{block VSP main theorem}
Let $G$ be a finite group, $F$ a number field and $R$ its ring of integers. Suppose that $FG$ has a matrix component. Then the following are equivalent:
    \begin{enumerate}
        \item $FG$ has $\Mexc$,
        \item $[F:\Q]\leq 2$ and $\vcd (\SL_1(R Ge)) \mid 4$ for every $e \in \PCI(F G)$ such that $F Ge$ is not a division algebra.
    \end{enumerate}
Furthermore, if $FG$ has $\Mexc$, then $\vcd (\SL_1(RGe)) > 4$  whenever $FGe$ is an exceptional division algebra.
\end{theorem}

\begin{remark}
    The implication on $\vcd$ from (1) to (2) in \Cref{block VSP main theorem} is true for any field $F$. However, for the converse the condition that $F$ is a quadratic number field is required. For example for any field $F \supseteq \Q$ one has that $FD_8 \cong F \times F \times \Ma_2(F)$. If $F$ is a cubic number field with one real embedding and one pair of complex embeddings (e.g. $F = \Q(\sqrt[3]{d})$ for $d \in \N_0$), then $\vcd(\SL_2(R)) = 4$ by \Cref{simple with vcd | 4} below, but $FD_8$ does not have $\Mexc$.
\end{remark}

By definition the virtual cohomological dimensions of commensurable groups are equal. In particular, the unit groups of two orders $\mc{O}_1$ and $\mc{O}_2$ in a simple algebra $A$ have equal $\vcd$. To prove \Cref{block VSP main theorem}, we will classify the simple algebras $A$ having an order $\O$ with $\vcd(\SL_1(\O)) \mid 4$, a result of independent interest.

Recall that in \cite[Proposition 3.3]{FreebyFree} the finite-dimensional simple $F$-algebras $A$, over a number field $F$, such that $\vcd(\SL_1(\mc{O})) \leq 2$ for some (and hence each) order $\mc{O}$ in $A$ have been classified. We extend this result to virtual cohomological dimension $4$. See \Cref{The other vcd} below for $\vcd(\SL_1(\mc{O})) =3$.
\begin{proposition}\label{simple with vcd | 4}
Let $A$ be a finite-dimensional simple $F$-algebra with $F$ a number field. If $\vcd(\SL_1(\mc{O})) = 4$ for an order $\mc{O}$ in $A$, then $A$ is isomorphic as an $F$-algebra to:
\begin{enumerate}
\item $\Ma_2(\qa{-a}{-b}{\Q})$ with $a,b \in \N_0$,
\item $\Ma_2(F)$, with $F$ a cubic field with precisely one real embedding and one pair of complex embeddings,
\item or to $\qa{-a}{-b}{F}$ such that it is non-ramified at exactly two real places and $F$ is totally real.
\end{enumerate}
\end{proposition}
\begin{proof}
Let $D$ a division ring of degree $d$, for an integer $n \geq 1$,  and let $A = M_n(D)$, and $F = \ZZ(D)$. Let $\O$ be an order in $A$. We make use of the following formula, as stated in \cite[Eq. (1)]{FreebyFree}. 
\begin{equation}
    \label{vcd formule}
    \vcd(\SL_1(\O)) = r_1 \frac{(nd-2)(nd+1)}{2} + r_2 \frac{(nd+2)(nd-1)}{2}+ s(n^2 d^2-1)-n+1,
\end{equation}
where $s$ is the number of pairs of non-real complex embeddings of $F$, $r_1$ is the number of real embeddings of $F$ at which $A$ is ramified, and $r_2$ the number of real embeddings of $F$ at which $A$ is not ramified. We may assume that $nd >1$, since when $nd = 1$, $A$ is a field, which implies that $\vcd(\SL_1(\O)) = 0$ by \cite[Proposition 3.3]{FreebyFree}. Note that for any choice of $nd \geq 2$, the first two terms of \cref{vcd formule} are non-negative. Additionally, when $d$ is odd, it is well-known that $r_1 = 0$.

Suppose $s \geq 2$. Then for any choice of $n$ or $d$ such that $nd = 2$, \[\vcd(\SL_1(\O)) \geq s(n^2 d^2-1)-n+1 > 4,\]
and this expression is strictly increasing in both $n$ and $d$. Hence $\vcd(\SL_1(\O)) > 4$ for any $s \geq 2$, and in particular we conclude that $F$ has at most one pair of non-real complex embeddings. 

Suppose then $s = 1$. Then \[\vcd(\SL_1(\O)) \geq s(n^2d^2 -1) -n+1 = n^2d^2 -n > 4,   \text{ when } n \geq 3.\]
 Hence $n$ is at most $2$ in this case. Suppose first $n = 1$. Then \[\vcd(\SL_1(\O)) \geq d^2 - 1 > 4, \text{ when } d > 2.\]
Thus $d = 2$ because $nd \geq 2$ by assumption. Then we find \[\vcd(\SL_1(\O)) = r_2 \frac{4 \cdot 1}{2}+4-1 = 4 \iff r_2 = - \frac{1}{2},\]
which is a contradiction since $r_2 \in \N$. Hence we cannot have $n = 1$. The only other option is $n = 2$. Then $n^2d^2 -n = 4d^2 - 2 \leq 4$ if and only if $d = 1$, in which case 
\[\vcd(\SL_1(\O)) = r_2 \frac{4\cdot 1}{2} + 2 = 4 \iff r_2 = 1.\]
In this case we find that $A = M_2(F)$ with $F$ a cubic number field with precisely one real embedding and one pair of complex embeddings.

Suppose now $s = 0$. We examine the case $r_2 = 0$ first. Then $r_1 \geq 1$ and 
\[\vcd(\SL_1(\O)) = r_1 \frac{(nd-2)(nd+1)}{2}-n+1,\]
and hence $\vcd(\SL_1(\O)) = 4$ necessarily implies that $nd \geq 3$. Suppose first that $n = 1$. Then \[\vcd(\SL_1(\O)) = r_1\frac{(d-2)(d+1)}{2} = 4 \iff r_1 \leq 2, \text{ since } d \geq 3.\]
If $r_1 = 2$, then $\vcd(\SL_1(\O)) = (d-2)(d+1) = 4$ implies that $d = 3$, but we have already remarked that $r_1 = 0$ when $d$ is odd, a contradiction. If $r_1 = 1$, then $(d-2)(d+1)=8$, which implies that $d = \frac{2 \pm \sqrt{32}}{2}$, which is not an integer, a contradiction. We conclude that if $s = 0$ and $r_2 = 0$, then $n >1$. Suppose $n = 2$. Then necessarily $d \geq 2$ since $nd \geq 3$. It follows that \[\vcd(\SL_1(\O)) = \frac{r_1}{2}(2d-2)(2d+1) -1 \geq 4,\]
with equality if and only if $r_1 = 1$ and $d = 2$. We find that in this case, $A = M_2(\left(\frac{-a,-b}{\Q}\right))$ for some positive integers $a,b$. Now the expression $\frac{r_1}{2}(nd-2)(nd+1)-n+1$ is strictly increasing in $n$ if and only if $n > \frac{d-1}{r_1d^2}$, which in the case at hand is satisfied since $d \geq 1$ and $r_1 \geq 1$ by assumption. It follows that $\vcd(\SL_1(\O)) > 4$ whenever $n \geq 3$. 

Still under the assumption that $s = 0$, we now turn our attention to the case  $r_2 \neq 0$, meaning $r_2 \geq 1$. Then 
\[\vcd(\SL_1(\O)) = \frac{r_2}{2}(nd+2)(nd-1) + \frac{r_1}{2}(nd-2)(nd+1) -n+1.\]
If $r_1 \geq 1$, then $d$ is even, meaning that $nd \geq 4$ when $n \geq 2$. But when $n = 2$, then 
\[\frac{r_2}{2} (2d+2)(2d-1) + \frac{r_1}{2}(2d-2)(2d+1) - 1 \geq 9 r_2 + 5 r_1 -1  > 4, \]
and again the expression above for $\vcd(\SL_1(\O))$ is strictly increasing in $n$. Hence $r_1 = 0$ if $r_2 \geq 1$ and $n \geq 2$. If $r_1 = 0$, then $\vcd(\SL_1(\O)) = 4$ if and only if $\frac{r_2}{2}(nd+2)(nd-1)-n = 3$. This expression is strictly increasing in $n$ if and only if $n > \frac{1-d}{r_2 d^2}$, which is always satisfied by assumption on $r_2$ and $d$. But when $n \geq 2$ and $d \geq 2$, meaning in particular that $nd \geq 4$, one finds that \[\frac{r_2}{2}(nd+2)(nd-1) -n \geq 7,\]
implying that only the cases $n = 1$, and $(n,d) = (2,1)$ need to be investigated. If $n = 2$ and $d = 1$, then 
\[\vcd(\SL_1(\O)) = \frac{r_2}{2} (nd+2)(nd-1)-n+1 = 2r_2-1,  \]
which equals $4$ if and only if $r_2 = \frac{5}{2}$, a contradiction.

In particular, for $s = 0$, $r_2 \geq 1$ and any value of $r_1$, only the case $n = 1$ remains. Assuming $n = 1$, we have in particular that $d \geq 2$, and
\[\vcd(\SL_1(\O)) = \frac{r_2}{2}(d+2)(d-1) + \frac{r_1}{2}(d-2)(d+1) \geq \frac{r_2}{2}(d+2)(d-1) \geq 4,\]
where the last inequality becomes an equality if and only if $d = 2 = r_2$. In that case one also finds that $\frac{r_1}{2}(d-2)(d+1) = 0$, and hence $\vcd(\SL_1(\O)) = 4$ if $s = 0$, $r_2 = d = 2$ and $r_1$ takes any arbitrary integer value. We obtain that $A = \left(\frac{-a,-b}{F}\right)$ with $F$ a totally real number field and $A$ is non-ramified at precisely two places. If $d \geq 3$, it now immediately follows that $\vcd(\SL_1(\O)) \geq 9r_2 + 4r_1> 4$, and this concludes our analysis. 
\end{proof}

\begin{remark}\label{The other vcd}
   With a similar proof one can verify that $\vcd(\SL_1(\O))= 3$ if and only if  $A$ is isomorphic to one of the following simple algebras:
   \begin{itemize}
       \item $\Ma_3(\Q)$,
       \item $\Ma_2(\Q(\sqrt{d}))$ with $d\in \N_0$ square-free,
       \item $\qa{-a}{-b}{F}$ such that $F$ has one pair of non-real complex embeddings and is ramified at all real places.
   \end{itemize}
Concerning $\vcd(\SL_1(\O)) \leq 2$, \cite[Proposition 3.3]{FreebyFree} shows that
\begin{itemize}
    \item $\vcd(\SL_1(\O)) =0$ if and only if $A$ is a field or a totally definite quaternion algebra,
    \item $\vcd(\SL_1(\O)) =1$ if and only if $A \cong \Ma_2(\Q)$
    \item $\vcd(\SL_1(\O)) =2$ if and only if $A\cong \Ma_2(\Q(\sqrt{-d}))$ or a quaternion algebra with a totally real centre and which is non-ramified at exactly one infinite place.
\end{itemize}
\end{remark}

The other main ingredient of the proof of \Cref{block VSP main theorem} is the following result of independent interest, which describes which simple algebras of the form  $\Ma_2(F)$ with $F$ a cubic number field can occur as a component of a rational group algebra.

\begin{lemma}\label{cubic root are not}
Let $G$ be a finite group. Let $F$ be a cubic number field. Suppose that the simple algebra $\Ma_2(F)$ is a quotient of $\Q G$, say $\Ma_2(F) \cong \Q G e$ with $e \in \PCI(\Q G)$. Then the prime divisors $\pi(Ge) \subseteq \{2, 3, 7 \}$, and $F = \Q(\zeta_7 + \zeta_7^{-1})$ or $F = \Q(\zeta_9 + \zeta_9^{-1})$.
\end{lemma}
\begin{example}
    Recall that from \eqref{bad component dicyclic} it follows that when $G = Dic_{4n}$ with $7 \mid n$ (respectively $9 \mid n$), then $\Q Dic_{4n}$ has a component which is isomorphic to $\Ma_2(F)$, with $F = \Q(\zeta_7 + \zeta_7^{-1})$ (respectively $F = \Q(\zeta_9 + \zeta_9^{-1})$) is a cubic extension of $\Q$. Thus the statement of \Cref{cubic root are not} is sharp.
\end{example}
\begin{proof}
Let $\lambda \in F$ be a torsion element, say of order $n$. Then it is a primitive $n$\textsuperscript{th} root of unity, denoted $\zeta_n$, and 
\begin{equation}
\label{torsion in F}
3 = \abs{F : \Q} = \abs{F : \Q(\zeta_{n})} \abs{\Q(\zeta_n):\Q} \geq \phi(n).
\end{equation}
It follows that $n \in \{ 1,2,3,4,6\}$.

Let $g \in Ge$. By fixing a $\Q$-basis of $F$ as a $3$-dimensional $\Q$-space, one can realise $g$ as a $6$-by-$6$ matrix over $\Q$. We denote by $\chi_{F,g}$ and $\chi_{\Q,g}$ respectively the characteristic polynomials of $g$ over $F$ and over $\Q$. Similarly, we write $\mu_{F,g}$ and $\mu_{\Q,g}$ for the minimal polynomials of $g$ over respectively $F$ and $\Q$. By definition, any minimal polynomial $\mu_{\Q,g}$ has degree at most $6$, and any $\mu_{F,g}$ has degree at most $2$. Remark that $\mu_{F,g}$ has degree 1 if and only if $g$ is a scalar matrix over $F$. \par 
From \cite[Page 147]{Koo}, it follows that for any $g \in Ge$ of prime power order $p^k$, the $p^k$\textsuperscript{th} cyclotomic polynomial $\Phi_{p^k}$ equals $\mu_{\Q,g}$. In particular, $\Q(\zeta_{p^k}) \cong \frac{\Q[X]}{(\mu_{\Q,g})}$. The latter also holds over $F$, when $p^k > 4$. Indeed, $\mu_{F,g}$ is given as the unique monic polynomial generating the ideal $I_{g} := \{P \in F[X] \mid P(g) = 0\}$, and the minimal polynomial of $\zeta_{p^k}$ over $F$ is given by the unique monic polynomial generating the ideal $I_{\zeta_{p^k}} := \{Q \in F[X] \mid Q(\zeta_{p^k}) = 0\}$. We claim that $I_g = I_{\zeta_{p^k}}$. Indeed, by \cite[Page 146]{Koo}, there is some matrix $B \in \GL_2(\C)$ such that $BgB^{-1} = \diag(\lambda_1,\lambda_2)$, with each $\lambda_i$ a $p^k$\textsuperscript{th} root of unity, amongst which at least one primitive (otherwise $BgB^{-1}$ and hence $g$ would have order strictly smaller than $p^k$). Without loss of generality, assume $\lambda_1 = \zeta_{p^k}$. Remark that since conjugation by an invertible matrix induces an algebra automorphism of $\Ma_2(\C)$, it follows that $P \in I_g$ if and only if $P(BgB^{-1}) = 0$. In particular, $P \in I_g$ if and only if $P(\zeta_{p^k}) = 0 = P(\lambda_2)$. We conclude that $I_g \subseteq I_{\zeta_{p^k}}$. But since $\deg(\mu_{F,g}) = 2$, and  $p^k > 4$ (meaning that $\zeta_{p^k} \not\in F$), it follows that $I_g$ is a maximal ideal of $F[X]$, and hence $I_g = I_{\zeta_{p^k}}$. In particular, $F(\zeta_{p^k}) \cong \frac{F[X]}{(\mu_{F,g})}$. \par 
 Since $\Phi_p$ divides $\mu_{\Q,g}$ when $g \in Ge$ is an element of prime order $p$, and the degree of $\mu_{\Q,g}$ is at most $6$ as remarked earlier, it follows that $p \in \{2,3,5,7\}$, and in particular $\pi(Ge) \subseteq \{2, 3 , 5, 7\}$. 
Suppose there is some element $g \in G$ of order $p \in \{5,7\}$. Remark that $\deg(\chi_{F,g}) = 2$, since otherwise $g$ would be a scalar matrix with a $p$\textsuperscript{th} root of unity on the diagonal, which is a contradiction with the description of torsion elements in $F$ as given in \cref{torsion in F}. 
Over $F(\zeta_p)$, $\chi_{F,g} $ splits as $ (X-\zeta_p^{i})(X-\zeta_p^k)$ for some $0 \leq i\neq k\leq p-1$.  Now $\zeta_p^{i + k} \neq 1$ would imply that  $F$ has a $p$\textsuperscript{th} root of unity, a contradiction by \cref{torsion in F}. Thus, $k = -i$, and the degree 1 coefficient in $\chi_{F,g}$ is equal to $-(\zeta_p^{i}+ \zeta_p^{-i})$. In particular $\Q(\zeta_p^{i} + \zeta_p^{-i} ) \subseteq F$. If $p = 5$, then since $\abs{\Q(\zeta_5+\zeta_5^{-1}) : \Q} = 2$ does not divide $\abs{F : \Q} = 3$, we obtain a contradiction. When $p=7$, $\abs{\Q(\zeta_7+\zeta_7^{-1}) : \Q }= 3$. It follows that $F = \Q(\zeta_7+\zeta_7^{-1})$.   \par 
Let now $\pi(Ge) \subseteq \{2,3\}$. We bound the exponent of $Ge$. Let $g \in Ge$, say of order $2^{i}3^{j}$ for some non-negative integers $i$, $j$. If $i = 0$ or $j=0$, then $o(g) = p^n$ with $p \in \{2,3\}$. Since $\Phi_{p^n}$ divides $\mu_{\Q,g}$ for each $p \in \{2,3\}$ and $\deg(\mu_{\Q,g}) \leq 6$, it follows that $o(g) \in \{p, p^2\}$. If $i \neq 0 \neq j$, then considering a realisation of $g$ as an element in $\GL_6(\Q)$,  \cite[Page 147]{Koo} implies that there exist positive integers $m_1, \ldots, m_r$ such that $o(g) = \mathrm{lcm}\{m_1,\ldots,m_r\}$,  $\Phi_{m_i}$ divides $\mu_{\Q,g}$, and  $6 = \sum_{i = 1}^r d_i \phi(m_i)$, for some $d_i \geq 1$. In particular each $\phi(m_i) \leq 6$. Since $\pi(Ge) \subseteq \{2,3\}$, from a case-by-case analysis it follows that $m_i \in \{1,2,3,4,6,9,18\}$. In particular, $\exp(G) \mid 36$. \par 
Suppose that $Ge$ contains an element $g$ of order $9$. Then since $F(\zeta_9) \cong \frac{F[X]}{(\mu_{F,g})}$, and $\deg(\mu_{F,g}) = 2$, \[\abs{F(\zeta_9):\Q} = \abs{F(\zeta_9) : F}\abs{F: \Q} =6. \]  Moreover, since $\abs{\Q(\zeta_9) : \Q}= 6$, it follows that $F(\zeta_9) = \Q(\zeta_9)$, and in particular $F \subseteq \Q(\zeta_9)$. The only subfields contained in $\Q(\zeta_9)$ are $\Q$, $\Q(\zeta_3)$ and $\Q(\zeta_9+\zeta_9^{-1})$, and of these only $\Q(\zeta_9 + \zeta_9^{-1})$ is of degree $3$ over $\Q$. 

Suppose now that $\pi(Ge) \subseteq \{2,3\}$ and all elements of order $3^n$ necessarily have order $3$. Then $\exp(Ge) \mid 12$. Then by Brauer's splitting field theorem, $\Q(\zeta_{12})$ is a splitting field for $Ge$, since $Ge$ has exponent a divisor of 12. In particular, $F \subseteq \Q(\zeta_{12})$. However, $\mathrm{Gal}(\Q(\zeta_{12})/\Q) \cong \Z/4\Z$, which has order coprime to $3$, which together with the fundamental theorem of Galois theory implies a contradiction.
\end{proof}

We are now able to prove \Cref{block VSP main theorem}.

\begin{proof}[Proof of \Cref{block VSP main theorem}]
First suppose that $FG$ has $\Mexc$ and let $e \in \PCI (FG)$. If $FGe$ is an exceptional matrix component, then $\vcd (RGe) = 2 \text{ or } 4$ by \Cref{simple with vcd | 4} and \Cref{The other vcd}. Moreover, by \Cref{remark on larger field}, $F = \Q(\sqrt{-d})$ for some square-free $d \in \N$.

Next, if $FGe$ is a division algebra, then as explained in the proof of \Cref{remark on larger field}, one has that $FGe \cong F(\chi) \otimes_{\Q(\chi)}A_{\Q}(\chi)$ for some absolutely irreducible representation $\chi$ of $G$, and $A_{\Q}(\chi)$ must be a division algebra of $\Q G$ which also has $\Mexc$. Hence, by \Cref{broad form finite grps}, we know that $A_{\Q}(\chi)$ is either a field or one of the following quaternion algebras:
\[\{\qa{\zeta_{2^t}}{-3}{\Q(\zeta_{2^t})}, \qa{-1}{-3}{\Q}, \qa{-1}{-1}{\Q(\sqrt{2})}, \qa{-1}{-1}{\Q(\sqrt{3})} \mid t\in \N_{\geq 3}\}.\]

Of these only the first one is exceptional and will remain so after tensoring with $F(\chi)$. Some of the others might become exceptional after tensoring with $F(\chi)$. However, inspecting the possible quaternion algebras with virtual cohomological dimension less than $4$ given by \Cref{simple with vcd | 4} and \Cref{The other vcd}, we see that they are not part of the list as they must have totally real centre, and $F \subseteq \Q(\sqrt{-d})$ for some $d \in \N$.

It remains to prove the converse, so suppose that $[F : \Q] \leq 2$ and  $\vcd (\SL_1(R Ge)) \mid 4$ for every $e \in \PCI(F G)$ such that $F Ge$ is a non-division simple component. By the results referred to above, the only simple algebras not allowed by the property $\Mexc$ are those of the form $\Ma_2(K)$, with $K$ a cubic number field with one real embedding and one pair of complex embeddings. Now as $F \subseteq \mathcal{Z}(FGe) = K$ and $[F: \Q] \mid [K: \Q]= 3$, one has that $F = \Q$. However by \Cref{cubic root are not}, the algebra $\Ma_2(K)$ cannot be the simple component of $F G$, finishing the proof. 
\end{proof}

\subsection{Higher Kleinian groups: discrete subgroups of \texorpdfstring{$\SL_4(\C)$}{SL4(C)} }\label{section kleinian groups}

Another interesting property is a kind of higher Kleinian property:
\begin{definition}\label{def discrete in}
A group $\Gamma$ is said to \textit{have property $\Di_n$} if $n$ is the smallest number such that $\Gamma$ is a discrete subgroup of $\SL_n(\C)$.
\end{definition}

We will be interested in the case that $\Gamma$ has $\Di_n$ with $n$ a divisor of $4$. An alternative way to look at this is via the $5$-dimensional hyperbolic space, since one has the following isomorphism:
\[\operatorname{Iso}^+(\mathbb{H}^5) \cong \PGL_2(\qa{-1}{-1}{\Q}).\]
In particular, a group $\Gamma$ acts discontinuously on\footnote{Here one assumes that the action on $\mathbb{H}^5$ does not come from the embedding of an action on $\mathbb{H}^4$.} $\mathbb{H}^5$ if and only if $\Gamma$ has $\Di_4$.

The finite-dimensional simple algebras $A$ such that $\SL_1(\O)$ is Kleinian, i.e. is a discrete subgroup of $\SL_2(\C)$, for an order $\O$ in $A$, were classified in \cite[Proposition 3.2]{FreebyFree}. Note that if $\SL_1(\O)$ has property $\Di_n$, then so does $\SL_1(\O^{\prime})$ for any other order $\O^{\prime}$ in $A$ (as both groups are commensurable, see \Cref{SL order commensurable}). 

\begin{proposition}\label{discrete in SL4 description}
Let $A$ be a finite-dimensional simple $F$-algebra with $F$ a number field and $\mc{O}$ an order in $A$. Then $\SL_1(\O)$ has property $\Di_4$ if and only if $A$ has one of the following forms:
\begin{enumerate}
\item $\Ma_2(\qa{-a}{-b}{\Q})$ with $a,b \in \N_0$,
\item $\Ma_4(\Q)$,
\item $\qa{-a}{-b}{\Q(\sqrt{-d})}$ with $a,b \in \N_0$ and $d \in \N_{>1}$ square-free,
\item A division algebra of degree $4$ which is non-ramified at at most one infinite place.
\end{enumerate}
\end{proposition}
\begin{proof}
Write $A = \Ma_n(D)$ with $D$ a finite-dimensional division algebra and denote $K = \ZZ (D)$. As mentioned earlier, property $\Di_n$ does not depend on the chosen order. For convenience, we choose one of the form $\Ma_n(\O)$ with $\O$ an order in $D$. We will consider the set of infinite non-compact places: 
\[V^{nc} := V_{\infty}(K) \setminus \{ v \in V_{\infty}(K)\mid \SL_1(\Ma_n(D \otimes_K K_v)) \text{ is compact }\}.\]
The places in $V^{nc}$ are exactly those at which we can use strong approximation. Note that if $n \geq 2$, then $V^{nc} = V_{\infty}(K)$ is the set of all infinite places.

Now suppose that $\SL_1(\O)$ is discrete in $\SL_4(\C)$ and take $v_0 \in V_{\infty}(K)$ such that $\SL_{n}(\O) \leqslant \SL_n(D \otimes_K K_{v_0})$ embeds discretely in $\SL_4(\C)$. Recall that $\SL_1 (\Ma_n(D \otimes_K K_v))$ is compact if and only if $n = 1$ and $A$ is ramified at $v$. Thus $v_0$ can be chosen in $V^{nc}$.

Next consider another place $v_0 \neq v_1 \in V^{nc}$ and consider the diagonal embedding 
\[\Delta \colon \SL_1(\Ma_n(D)) \hookrightarrow \prod\limits_{v \in V^{nc} \setminus \{ v_1\} } \SL_1( \Ma_n(D \otimes_K K_v)).\]
By strong approximation, see \cite[Theorem 7.12, pg 418]{platRap}, the image $\Ima (\Delta)$ is dense. Therefore, $v_0 \in V^{nc}\setminus\{ v_1\}$ yields a factor in which the image of $\SL_1(\Ma_n(D))$ is both dense and discrete, a contradiction since $\SL_n(\O)$ is not finite (being discrete in $\SL_4(\C)$). Consequently, 
\begin{equation}\label{bound number nc places}
|V^{nc}| \leq 1.
\end{equation}

Now, notice that being discrete in $\SL_4(\C)$ implies that $\dim_{\C} \Ma_n(D \otimes_K \C) \leq 16$. Furthermore, if it is not discrete in some $\SL_m(\C)$ for $m \leq 4$, then $\dim_{\C} \Ma_n(D \otimes_K \C)= 16$. The latter implies that $A = \Ma_4(F)$, $\Ma_2(D)$ or  $D^{\prime}$ with $D$ a quaternion algebra and $D^{\prime}$ a division algebra of degree $4$. This combined with \eqref{bound number nc places} yields the stated possibilities. Indeed, simply recall that $\SL_1 (\Ma_n(D \otimes_K K_v))$ is compact if and only if $n = 1$ an $A$ is ramified at $v$. 

Conversely, it is easily verified that the $\SL_1$ of all algebras mentioned indeed have a discrete embedding in $\SL_4(\C)$.
\end{proof}

\begin{remark}\label{remark about other discrete degrees}
    The proof of \Cref{discrete in SL4 description} also yields that $\SL_1(\mc{O})$ has property $\Di_3$ if and only if $A$ is isomorphic to $\Ma_3(\Q)$ or a division algebra of degree $3$ which is non-ramified at at most one infinite place. 

    Concerning $\Di_2$, \cite[Remark 3.5]{FreebyFree} says that $\SL_1(\mc{O})$ has $\Di_2$ (i.e. is Kleinian) if and only if $\vcd (A) \leq 2$ or $A$ is quaternion division algebra which is ramified at all its infinite places and which has exactly one pair of complex embeddings. 
\end{remark}

In \cite{FreebyFree} it was proven that $\SL_1(\Z Ge)$ has $\Di_n$ for $n\leq 2$ for each $e \in \PCI(\Q G)$ if and only if $\Q G$ has no exceptional division components and the non-division components are of the form $\Ma_2(\Q(\sqrt{-d}))$ for some $d \in \N$. In our case, where we allow exceptional division algebras, the situation is more intricate, as the next example illustrates.

\begin{example}
Consider the group
\[C_3 \rtimes C_{2^m} = \langle a,b \mid a^3=b^{2^m}=1, a^b = a^{-1}\rangle,\]
with $m \geq 4$. In \eqref{decomp grp with exp comp} we obtained that 
\[
\begin{array}{rcr}
\mc{C}(\Q[C_3 \rtimes C_{2^m}]) = \Big\{ \Q(\zeta_{2^{\ell}}),  \qa{-1}{-3}{\Q}, \qa{\zeta_{2^t}}{-3}{\Q(\zeta_{2^t})}, & \hspace{-1.7cm}\Ma_2(\Q), \: \Ma_2(\Q(i)) \\[0,2cm] & \hspace{-.5cm} \mid 1 \leq \ell \leq n, \, 3 \leq t \leq m-1 \Big\}.
\end{array}
\]
Thus $C_3 \rtimes C_{2^m}$ has all non-division algebra components as in \cite{FreebyFree}, but also has exceptional division components which do not have $\Di_n$ for $n \leq 2$.
\end{example}

From \Cref{remark about other discrete degrees}, \Cref{discrete in SL4 description} and \Cref{broad form finite grps} we obtain the following corollary.

\begin{corollary}
Let $G$ be a finite group and $F$ a number field with ring of integers $R$. Suppose that $FG$ has $\Mexc$. Then 
\[\SL_1(R Ge) \text{ has } \Di_n \text{ with } n \mid 4,\]  
for each $e \in \PCI(F G)$ such that $F Ge$ is not a division algebra. If $FG$ has no exceptional division components, then this holds for all $e \in \PCI(FG)$.
\end{corollary}

Consider the group 
\[G := C_5 \rtimes C_8 = \langle a,b \mid a^5 =b^8=1, a^b=a^3 \rangle.\]
It can be verified that
\[\mc{C}(\Q G) = \{\Q, \Q(i), \Q(\zeta_8), \Ma_2(\qa{-2}{-5}{\Q}), \Ma_4(\Q)\}.\] From \Cref{discrete in SL4 description} and \Cref{remark about other discrete degrees} we see that $\SL_1(\Q Ge)$ has $\Di_n$ with $n \mid 4$ for each primitive central idempotent $e$. However, $\Q G$ does not have $\Mexc$. Following \Cref{broad form finite grps}, the group theoretical origin for the failure of $\Mexc$ is that $[G:\Fit(G)]=4$.  

Through a precise classification of finite groups $G$ such that $\Q G$ has $\Mexc$ we will obtain in the companion paper \cite{CJT} the following.

\begin{theorem}\label{final theorem dis}
Let $G$ be a finite group and $F$ a number field. Then the following are equivalent:
\begin{enumerate}
\item $F G$ has $\Mexc$ (resp. and has no exceptional division algebra components)
\item $[G:\Fit(G)] \leq 2$ and $\Fit(G)$ has class at most $3$ and
\[\SL_1(R Ge) \text{ has } \Di_n \text{ with } n \mid 4,\]
for each $e \in \PCI(F G)$ such that $F Ge$ a non-division algebra (resp. for each $e \in \PCI(F G)$).
\end{enumerate}
\end{theorem}

\begin{remark*}
As a preliminary piece of evidence, note that if we add the assumption that $FG$ has $\wMexc$, then part \eqref{item: Mexc versus wMexc} of \Cref{broad form finite grps} yields the statement of \Cref{final theorem dis}, but with $[G:\Fit(G)] \leq 2$ replaced by $[Ge:\Fit(Ge)] \leq 2$.
\end{remark*}

\subsection{The good property}\label{subsectie good}

In this section, we fix for each $e \in \PCI(G)$ a maximal order $\Ma_{n_e}(\O_e)$ in $\Q Ge$. The arguments will however be independent of this choice. Furthermore, we denote by $\widehat{\Gamma}$ the profinite completion of a group $\Gamma$. Recall, following terminology introduced by Serre, that $\Gamma$ is called \emph{good} if the map
\[H^j(\widehat{\Gamma},M) \rightarrow H^j(\Gamma,M),\]
induced by the inclusion of $\Gamma$ in its profinite completion, is an isomorphism for any $j$ and any finite $\Gamma$-module $M$. Note that any finite group has the good property, since finite groups are isomorphic to their profinite completion. Moreover, free groups have the good property (see \cite[Proposition 3.7]{GJZGood}).

By \cite[Lemma 3.2, Proposition 3.4]{GJZGood}, the good property is preserved under commensurability and is closed under direct products. Writing $FG = \prod_{i=1}^m \Ma_{n_i}(D_i)$, $R$ an order in $F$ and $\O_i$ an order in $D_i$, \Cref{SL order commensurable} implies that $\SL_1(R G)$ is good exactly when $\SL_{n_i}(\O_i)$ is good for all $1 \leq i \leq m$.  One has moreover the following statement which follows directly by assembling results in the literature and is in fact implicit in the proof of \cite[Theorem 1.5]{CdRZ}.  For convenience of the reader, we record it here explicitly and provide a proof.

\begin{proposition}\label{good for exceptional}
Let $G$ be a finite group and $F$ a field of characteristic $0$. Further let $\O$ be an order in $FGe$ with $e \in \PCI(FG)$ such that $FGe$ is not a division algebra. Then $\SL_1(\O)$ is good if and only if $FGe$ is an exceptional matrix algebra.
\end{proposition}
\begin{proof}
Let $FGe = \Ma_n(D)$. Then, thanks to \Cref{SL order commensurable} and the fact that the good property is preserved under commensurability, we may assume that $\O$ is of the form $\Ma_n(\O')$ with $\O'$ a maximal order in $D$. 

Suppose $FGe$ is an exceptional matrix algebra, i.e. $n = 2$, and $\O'$ is $\mc{I}_d$, the ring of integers in $\Q(\sqrt{-d})$ for $d \geq 0$, or $\O'$ is an order in a totally definite quaternion algebra. In \cite[Theorem 1.1]{GJZGood}, it is shown that $\PSL_2(\mc{I}_d)$ is good, for all $d \geq 1$. Since the centre of $\SL_2(\mc{I}_d)$ is finite, it follows moreover from \cite[Lemma 3.3]{GJZGood} that $\SL_2(\mc{I}_d)$ is good. Additionally, $\SL_2(\Z)$ is good since it is commensurable to a free group. 

As explained in the proof of \cite[Remark 3.5]{CdRZ}, $\SL_2(\O')$ with $\O'$ an order in a totally definite quaternion algebra is commensurable to a standard arithmetic lattice of $\SO(1,5)$. Such a subgroup of $\SO(1,5)$ admits a separable hierarchy as it is virtually special \cite[Theorem 1.10]{BHD}. Consequently, one can apply \cite[Theorem 3.9]{GJZGood}, which gives the good property whenever such a hierarchy exists. We conclude that $\SL_1(\O)$ is good when $FGe$ is an exceptional matrix algebra.

Conversely, in \cite[Proposition 5.1]{GJZGood} it was proven that $\SL_n(\O')$ is not good if it enjoys the subgroup congruence property. In particular  if $n\geq 3$ or $n=2$ and $\U(\O')$ is infinite, then it is not good, see \cite{BassMilnorSerre} for the case $n \geq 3$ and \cite{SerreSL2} for the case $n = 2$. Suppose $n = 2$. Now $\U(\O')$ is finite if and only if $D$ is either isomorphic to $\Q$, to an imaginary quadratic field or to a totally definite quaternion algebra, see \cite[Theorem 2.10]{BJJKT}, i.e. in the case at hand if and only if $FGe$ is an exceptional matrix algebra. It follows that when $FGe$ is a non-exceptional matrix component, $\SL_1(\O)$ does not have the good property, as desired.
\end{proof}

We now readily obtain the following characterisation.

\begin{corollary}\label{good for group rings}
Let $G$ be a finite group with no exceptional division components, $F$ a number field with ring of integers $R$. Then the following are equivalent:
\begin{enumerate}
    \item\label{item SL good} $\SL_1(R G)$ is good,
    \item \label{item unit good} $\U ( R G)$ is good,
    \item \label{mexc versus good} $F G$ has $\Mexc$.
\end{enumerate}
\end{corollary}
\begin{remark}\label{good and exc div}
The proof of \Cref{good for group rings} will show that for any $G$ the group $\SL_1(R G)$ is good if and only if $\U (R G)$ is good. In fact, one expects that 
\[\SL_1(R G)\text{ is good } \iff  F G \text{ has } \Mexc \text{ and has no exceptional division components.}\]
Indeed, from  \Cref{broad form finite grps} and \Cref{remark on larger field} one obtains explicitly the possible exceptional division components $D$. Now if $S$ denotes the set of infinite places of $D$, then for each of them a direct computation shows that the $S$-rank of $D$ is at least two. Hence Serre's conjecture predicts that $\SL_1(\O)$ satisfies the subgroup congruence property, for $\O$ some order in $D$, and hence $\SL_1(\O)$ would not be good by \cite[Proposition 5.1]{GJZGood}.

In the companion paper \cite{CJT} it will be shown that when $\Q G$ has $\Mexc$, then $\Q G$ has an exceptional matrix component if and only if the Sylow $2$-subgroups of $G$ are abelian and $16 \mid \exp(G)$. 
\end{remark}

\begin{proof}[Proof of \Cref{good for group rings}]
For the equivalence between \eqref{item SL good} and \eqref{item unit good}, it follows from \cite[Corollary 5.5.3]{EricAngel1} that $\U(R G)$ is commensurable with $\U(\ZZ(R G))) \times \SL_1(R G)$.
Moreover, $\ZZ(R G)$ is an order in $\ZZ( F G)$ and hence $\U(\ZZ(R G))$ is finitely generated by \cite[Theorem 5.3.1]{EricAngel1}.
Because $\U(\ZZ(R G))$ is a finitely generated abelian group, it is good (since $\Z$ is good and finite groups are good). Hence $\U(\ZZ(R G)) \times \SL_1(R G)$ is good if and only if $\SL_1(R G)$ is good\footnote{The argument in the proof of \cite[Proposition 3.4]{GJZGood} using the Künneth theorem may be adapted to show that if $G_1 \times G_2$ is good, and $G_1$ is good, then necessarily $G_2$ is good. This follows by considering the $0$\textsuperscript{th} cohomology groups for $N_1$ in the proof, which coincide with $\F_p$.}.  This finishes the equivalence between \eqref{item SL good} and \eqref{item unit good}.

For the equivalence between \eqref{item SL good} and \eqref{mexc versus good} we will use that $F G$ is assumed to have no exceptional division components. Namely, recall that $\SL_1$ over a non-exceptional division component is finite, e.g. see \cite[Theorem 2.10]{BJJKT}. Hence they are good. We obtain the equivalence between \eqref{item SL good} and \eqref{mexc versus good} by \Cref{good for exceptional} and the paragraph above it.
\end{proof}

\section{\texorpdfstring{Congruence subgroups of rank $1$ have virtually free quotient}{Congruence subgroups of rank 1 have virtually free quotient}}\label{sectie vFQ}

The aim of this section is to show \emph{largeness} for principal congruence subgroups of $\SL_2(\O)$ with $\O$ an order in a division algebra $D$ such that $\Ma_2(D)$ is an exceptional matrix algebra. A group is called\footnote{Such groups are also referred to as having property $\vFQ$, see \cite{Lubot96}. However, we will avoid this notation to avoid confusion with property vQL.} \emph{large} if it has a quotient which is virtually a non-abelian free group. A group \emph{has property $\vQL$} if it has a finite index subgroup which maps onto a non-abelian free group, i.e. it virtually has a non-abelian free quotient. Note that largeness implies $\vQL$.

\begin{theorem}\label{infinite abelianization SL}
Let $\qa{u}{v}{\Q}$ be a totally definite quaternion algebra with centre $\Q$ and let $\LL_{u,v}$ be the $\Z$-order with basis $\{1,i,j,k\}$. Then every torsion-free principal congruence subgroup of $\SL_2(\LL_{u,v})$ is large. 
\end{theorem}

In \Cref{torsionfree,} it is shown that sufficiently deep (depending on $u$) principal congruence subgroups are torsion-free. As explained in the introduction, the analogue of \Cref{infinite abelianization SL} for $\SL_2(\mc{I}_d)$ was obtained in \cite{GS, Lubot96}. Furthermore, $\SL_2(\Z)$ is itself virtually free. Combining these results with  \Cref{infinite abelianization SL}, understanding the remaining exceptional matrix algebra  $\SL_2(\O)$ with $\O$ an order in $\qa{u}{v}{\Q}$,  we obtain the following.

\begin{corollary}\label{vQL for exceptional}
Let $A$ be an exceptional matrix algebra, $\O$ an order in $A$ and $\Gamma$ a group commensurable with $\SL_1(\O)$. Then $\Gamma$ has a finite index subgroup mapping onto a non-abelian free group.
\end{corollary}

In \Cref{subsection vFQ vsp} we will use \Cref{vQL for exceptional} to obtain another characterisation of property $\Mexc$.

The proof of \Cref{infinite abelianization SL} is geometric and uses tools from \cite{Lubot96} and \cite{Kiefer}. More precisely, in \Cref{congr_subgroups_SL2} we relate the principal congruence subgroups of $\SL_2(\LL_{u,v})$ to congruence subgroups of certain groups of isometries of the hyperbolic $5$-space, denoted $\SL_+\left(\Gamma_4^{u,v}(\Z)\right)$, called \emph{higher modular groups}. Once this connection is settled, \Cref{infinite abelianization SL} will follow from the following statement on higher modular groups 

\begin{theorem}\label{vFq for Con}
    Let $u,v \in \Z_{< 0}$. Let $\con_n$ (defined in \eqref{def con}) be a torsion-free principal congruence subgroup of the higher modular group $\SL_+\left(\Gamma_4^{u,v}(\Z)\right)$. Then $\con_n$ is large.
\end{theorem}

The fact that there exists some $n$ such that $\con_n$ is torsion-free, follows from \Cref{torsionfree} and the connection built in the proof of \Cref{infinite abelianization SL}.

Finally, in \Cref{congruence kernel qa} we use \Cref{infinite abelianization SL} to show that the congruence kernel for $\SL_2(\O)$ and $\SL_+\left(\Gamma_4^{u,v}(\Z)\right)$ contains a copy of the free group of countable rank. In case of $\SL_2(\mc{I}_d)$ this was obtained in \cite[Theorem B]{Lubot82} and our proof uses the methods in \emph{loc. cit.} The result obtained here will be used in \Cref{subsection vFQ vsp} to study the congruence kernel of the unit group of an integral group ring. 

\subsection{Background on \texorpdfstring{$\SL_2$}{SL2} over quaternionic orders}\label{clifford_background}

To prove the main result, we introduce some notation. We define, for $u$ and $v$ strictly negative integers, the following subset of $\qa{-1}{-1}{\R}$: \begin{eqnarray*}\label{Huv}
\mathcal{H}_{u,v}(\Q):= \left\lbrace a_0+a_1\sqrt{\vert u \vert}i+a_2\sqrt{\vert v \vert}j+a_3\sqrt{\vert uv\vert}k \mid a_0, a_1, a_2, a_3 \in \Q \right.\\ \left. \text{ and } i^2=j^2=-1, ij=-ji=k \right\rbrace.
\end{eqnarray*}

It is not hard to see that the following statement holds.
\begin{lemma}[{\cite[Lemma~6.1]{Kiefer}}]
$\mathcal{H}_{u,v}(\Q)$ is a subalgebra of $\qa{-1}{-1}{\R}$ isomorphic to $\qa{u}{v}{\Q}$. The isomorphism is explicitly given by \begin{align}
\Lambda\colon \mathcal{H}_{u,v}(\Q) & \rightarrow \left(\frac{u,v}{\Q}\right) \label{Lambda iso}, \\
a_0 + a_1\sqrt{\vert u \vert}i+a_2\sqrt{\vert v \vert}j+a_3\sqrt{\vert uv\vert }k & \mapsto  a_0 + a_1i+a_2j+a_3k. \notag
\end{align}
\end{lemma}

We will be using the well-established Clifford algebras $\Cliff_n(\R)$, i.e. the $\R$-algebras generated abstractly (as an algebra) by $n-1$ elements $i_1, \ldots, i_{n-1}$ satisfying the relations \begin{equation*}
i_hi_k = -i_ki_h, \textrm{ for } h \neq k\quad \textrm{ and }\quad i_h^2=-1 \textrm{ for } 1 \leq h \leq n-1.
\end{equation*}

For example, \[ \Cliff_1(\R) \cong \R, \quad \Cliff_2(\R) \cong \C \quad \text{ and } \quad \Cliff_3(\R) \cong \qa{-1}{-1}{\R}.\]

The following definitions are also well-established, see e.g. \cite{1985ahlfors}. Recall that the main conjugation $.'$ on $\Cliff_n(\R)$ is defined as the $\R$-linear automorphism determined by $i_h \mapsto -i_h$ for every $1\leq h \leq n-1$, and that there is an anti-involution $.^*$ defined as the $\R$-linear automorphism defined on an arbitrary basis element $i_{h_1}i_{h_2}\ldots i_{h_m}$ by reversing its order, i.e. by sending it to $i_{h_m}i_{h_{m-1}}\ldots i_{h_1}$.
The composition of these two commuting automorphisms yields a new anti-involution denoted by $\overline{a} := a'^*$. 

The vector space $\Cliff_n(\R)$ is endowed with a Euclidean norm \[ |a| := \sqrt{\sum a_I^2},\] where $a = \sum a_I I \in \Cliff_n(\R)$, $I=i_{j_1} \cdots i_{j_r},$ $1 \leq j_1 < ... < j_r \leq n-1$ (notice that such elements $I$ form a basis for $\Cliff_n(\R)$).

Moreover, for $a$ an element of the $n$-dimensional linear subspace $\V^n(\R) \leq \Cliff_n(\R)$ generated by the basis $\{i_0 := 1, i_1, \ldots, i_{n-1}\}$, one readily computes that $a\overline{a} = |a|^2$, showing that every non-zero vector of $\V^n(\R)$ is invertible. The group generated by all such vectors is denoted by $\Gamma_n(\R)$ and is called the \emph{Clifford group}.

This Clifford group also gives rise to general and special linear groups: \begin{eqnarray*}
\GL(\Gamma_n(\R)) & := & \left \lbrace \begin{pmatrix} a & b \\ c & d \end{pmatrix} \mid a,b,c,d \in \Gamma_n(\R) \cup \lbrace 0 \rbrace,\ ad^*-bc^* \in \R^{\ast}, \right.\\
& & \phantom{ \left \lbrace \begin{pmatrix} a & b \\ c & d \end{pmatrix} \mid \right. } \left. ab^*,\ cd^*,\ c^*a,\ d^*b \in \V^n(\R) \right \rbrace,\\
\SL_+(\Gamma_n(\R)) & := & \left \lbrace \begin{pmatrix} a & b \\ c & d \end{pmatrix} \in \GL(\Gamma_n(\R))  \mid ad^*-bc^* =  1 \right \rbrace,
\end{eqnarray*}
for which it can be shown that the map $\begin{psmallmatrix} a & b \\ c & d \end{psmallmatrix} \mapsto ad^* - bc^*$ is multiplicative. Lastly, $\Cliff_n(\Z)$ denotes the $\Z$-subalgebra of $\Cliff_n(\R)$ generated by $i_0, \ldots, i_{n-1}$, and we let $\Gamma_n(\Z) := \Cliff_n(\Z) \cap \Gamma_n(\R)$, the submonoid of products of vectors from $\Cliff_n(\Z)$ which are themselves invertible in $\Cliff_n(\R)$.

\begin{remark}\label{Iso_SL+Gamma4R} The groups $\SL_+(\Gamma_4(\R))$ and $\SL_2\left(\qa{-1}{-1}{\R}\right)$ are isomorphic. This non-obvious fact is proven in \cite[Section~$6$]{ElsGrunMen}, though the explicit isomorphism, as we will use later, can be found in \cite[Section~$5$]{Kiefer}.
For the sake of completeness, we mention it here. Note that the elements \begin{equation*}
    \varepsilon_1  =  \frac{1 + i_1i_2i_3}{2}, \quad 
    \varepsilon_2  =  \frac{1 - i_1i_2i_3}{2}
\end{equation*}
form a pair of central, orthogonal idempotents of $\Cliff_4(\R)$. Indeed, one readily checks that they commute with all generators of $\Cliff_4(\R)$:
\[ \varepsilon_i^2 = \varepsilon_i, \quad \varepsilon_1\varepsilon_2 = 0, \quad \text{ and } \quad \varepsilon_1 + \varepsilon_2 = 1.\]
Using this, we obtain that every element $\alpha \in \Cliff_4(\R)$ can be uniquely expressed as \[ \alpha = a \varepsilon_1 + b \varepsilon_2,\] with $a,b \in \Cliff_3(\R)$. Indeed, writing $\alpha = \sum \alpha_I I$ for $I \in \{1,i_1,i_2,i_3,i_1i_2,i_1i_3,i_2i_3,i_1i_2i_3\}$, and $a = \sum a_I I$, $b = \sum b_I I$ for $I \in \{1, i_1, i_2, i_1i_2\}$ one straightforwardly calculates that 
\[\begin{array}{ccccccc}
   a_1 & =  & \alpha_1 + \alpha_{i_1i_2i_3}, & \quad & b_1 & = & \alpha_1 - \alpha_{i_1i_2i_3}, \\
a_{i_1}  & = & \alpha_{i_1} - \alpha_{i_2i_3}, & \quad & b_{i_1}  & = & \alpha_{i_1} + \alpha_{i_2i_3}, \\
a_{i_2}  & = & \alpha_{i_2} + \alpha_{i_1i_3},& \quad & b_{i_2}  & = & \alpha_{i_2} - \alpha_{i_1i_3}, \\
a_{i_1i_2}  & = & \alpha_{i_1i_2} - \alpha_{i_3},& \quad & b_{i_1i_2}  & = & \alpha_{i_1i_2} + \alpha_{i_3}.
\end{array}\]
Because of the fact that $\varepsilon_1$ and $\varepsilon_2$ are central orthogonal idempotents, one can make an algebra homomorphism $\chi' \colon \Cliff_4(\R) \rightarrow \Cliff_3(\R), \: \alpha \mapsto a $. Using the isomorphism $\Cliff_3(\R) \cong \qa{-1}{-1}{\R}$ allows us to extend it to an algebra homomorphism $\chi \colon \Cliff_4(\R) \rightarrow \qa{-1}{-1}{\R}$, which can be explicitly expressed as \[\chi(\alpha) = \alpha_1 + \alpha_{i_1i_2i_3} + (  \alpha_{i_1} - \alpha_{i_2i_3} ) i + ( \alpha_{i_2} + \alpha_{i_1i_3} ) j + (\alpha_{i_1i_2} - \alpha_{i_3}) k.\] It is this homomorphism (which is obviously not an isomorphism) that extends, by entrywise application, to a group isomorphism between $\SL_+(\Gamma_4(\R))$ and $\SL_2\left(\qa{-1}{-1}{\R}\right)$.
\end{remark}
We now define a transformation $\sigma$ of $\Cliff_n(\R)$ as \[\sigma(z) = -z'.\] Clearly $\sigma^2 = \mathrm{Id}_{\Cliff_n(\R)}$.

We also define the following subset of $\Cliff_4(\R)$:
\begin{eqnarray*}\label{Cliffalgebran}
\Cliff_4^{u,v}(\Q):= \left\lbrace \alpha_0+\alpha_1\sqrt{\vert u \vert}i_1+\alpha_2\sqrt{\vert v \vert}i_2+\alpha_3\sqrt{\vert uv\vert}i_3+\alpha_4\sqrt{\vert uv\vert}i_1i_2
\right.\\ \left. +\alpha_5\sqrt{\vert v\vert}i_1i_3+\alpha_6\sqrt{\vert u\vert}i_2i_3+\alpha_7i_1i_2i_3  \mid \alpha_i \in \Q \right\rbrace,
\end{eqnarray*}
and one straightforwardly checks that it is a subalgebra of $\Cliff_4(\R)$.
Analogously, we define $\Cliff_4^{u,v}(\Z)$, a $\Z$-order of $\Cliff_4^{u,v}(\Q)$, $\Gamma_4^{u,v}(\Z)$ and $\SL_+\left(\Gamma_4^{u,v}(\Z)\right)$\label{SLforcliff2} (which is a subgroup of $\SL_+\left(\Gamma_4(\R)\right)$).
We consider the principal congruence subgroups of level $n \geq 1$ of the latter group, i.e. \begin{equation}\label{def con}
\con_n := \left\lbrace \begin{pmatrix} 1 + n a & nb\\ nc & 1+nd \end{pmatrix} \in \SL_+\left(\Gamma_4^{u,v}(\Z)\right) \mid a,b,c,d \in \Cliff_4^{u,v}(\Z) \right\rbrace.
\end{equation}

\subsection{Virtually free quotients of congruence subgroups}\label{congr_subgroups_SL2}

We start with proving that the congruence subgroups of higher modular groups are large.
\begin{proof}[Proof of \Cref{vFq for Con}]
Consider the hyperbolic space of dimension $5$, 
\[\mathbb{H}^5= \{ z_0 + z_1i_1 + z_2i_2 + z_3i_3 + z_4i_4 \mid z_i \in \R, z_4 \geq 0\} \subseteq \V^5(\R).\]
It is not hard to see that $\sigma$ restricted to $\mathbb{H}^5$ is simply a reflection along the plane with equation $z_0 = 0$.
It is well-known that $\PSL_+\left(\Gamma_4(\R)\right)$ acts faithfully via a M\"{o}bius transformation on $\mathbb{H}^5$: \[Mz = \begin{pmatrix} \alpha & \beta \\ \gamma & \delta\end{pmatrix}z = (\alpha z + \beta)(\gamma z + \delta)^{-1}, \quad M \in \PSL_+\left(\Gamma_4(\R)\right), \: z \in \mathbb{H}^5,\]
where we abuse the notation by representing an element of $\PSL_+\left(\Gamma_4(\R)\right)$ by a matrix. However, it is clear this choice does not impact the way it acts via M\"{o}bius transformation.

Let now $M=\begin{psmallmatrix} \alpha & \beta \\ \gamma & \delta\end{psmallmatrix}$ be in the subgroup $\con_n$, i.e. $\alpha = 1+na$, $\beta = nb$, $\gamma = nc$ and $\delta = 1+nd$ for some $a,b,c,d \in \Cliff_4^{u,v}(\Z)$. Denote $\operatorname{PCon}_n$ by its image in $\PSL_+\left(\Gamma_4(\R)\right)$ and again abusively say that $M \in \operatorname{PCon}_n$ (indeed, since $\con_n$ is torsion-free, we may assume $M \in \operatorname{PCon}_n$). We claim that $\sigma M \sigma$ is again in $\operatorname{PCon}_n$. Indeed, for every $z \in \mathbb{H}^5$ we have $$\left(\sigma M \sigma\right) z = - \left((-\alpha z' + \beta)(-\gamma z' + \delta)^{-1}\right) ' = (\alpha ' z - \beta ')(-\gamma ' z + \delta ')^{-1}. $$
Hence, $\sigma M \sigma$ corresponds to the M\"{o}bius transformation given by the matrix $\begin{psmallmatrix} \alpha ' & -\beta ' \\ -\gamma ' & \delta ' \end{psmallmatrix}$, which is an element of $\operatorname{PCon}_n$.

Hence, the subgroup $\operatorname{PCon}_n$ is normalised by a reflection. This shows, by \cite[Corollary~3.6]{Lubot96}, that $\operatorname{PCon}_n$ maps to a virtually free group.
\end{proof}

As indicated by \Cref{vFq for Con}, in order to prove \Cref{infinite abelianization SL}, we need to consider torsion-free principal congruence subgroups of $\SL_2(\LL_{u,v})$. We now show that these exist.

\begin{lemma}\label{torsionfree}
Let $\LL_{u,v}$ be as above. There exists a prime $p$ such that for every larger prime $q$, with $u$ not a square modulo $q$, the principal congruence subgroup of level $q$ of $\GL_2(\LL_{u,v}) = \SL_2(\LL_{u,v})$ is torsion-free.
\end{lemma}
\begin{proof}
Consider the map \[ \varphi \colon \GL_2(\LL_{u,v}) \rightarrow \GL_2(\LL_{u,v} / q\LL_{u,v}).\]
We want to find a prime $p$ and show that for $q$ described in the statement, the kernel of $\varphi$ is torsion-free.

It is immediately verified that \begin{align*}\iota \colon \qa{u}{v}{\Q} &\rightarrow \Ma_2(\Q(\sqrt{u})) \\
a+bi+cj+dk &\mapsto \begin{pmatrix} a+b\sqrt{u} & c+d\sqrt{u} \\ cv - dv\sqrt{u} & a-b\sqrt{u} \end{pmatrix},
\end{align*}
is an injective homomorphism, such that $\iota (\LL_{u,v}) \subseteq \Ma_2(\mathcal{I}_u)$. We may thus extend $\iota$ to an injection $\iota \colon \GL_2(\LL_{u,v}) \rightarrow \GL_4(\I_u)$.

For this larger group, we may also consider the principal congruence subgroup of level $q$, namely the kernel of the canonical morphism \[ \psi \colon \GL_4(\mathcal{I}_u) \rightarrow \GL_4(\mathcal{I}_u / q\mathcal{I}_u).\]

From the expression of $\iota$, it is not hard to see that $\iota(\ker \varphi) \leq \ker \psi$.
From the assumptions on $q$, i.e. that $u$ is not a square modulo $q$, it follows that $q\mathcal{I}_u$ is a prime ideal. This is well-known, for example see \cite[\textsection~5.4~Proposition~1]{Samuel2}. As such, from \cite[Lemma~9]{GurLor} it follows that every torsion element of the kernel has order some power of $q$.
On the other hand, by a result of Schur which is also stated in \cite[Theorem 14]{GurLor}, there exists a natural number $s$ such that the orders of the torsion elements of $\GL_4(\mathcal{I}_u)$ divide $s$.

So, if we let $p$ be a prime larger than $s$ and $q$ a prime larger than $p$ such that $u$ is not a square modulo $p$, then all torsion elements of $\ker \varphi$ have to have an order which is a power of $q$ and which divides $s$. This is only possible when the order is $1$. This proves that $\ker \varphi$ is torsion-free.
\end{proof}

\begin{proof}[Proof of \Cref{infinite abelianization SL}]
Let $G_n$ be the principal congruence subgroup of level $n$ of $\SL_2(\LL_{u,v})$.
To prove the result, we will assume $G_n$ to be a torsion-free principal congruence subgroup. The fact that such an index $n$ exists follows from \Cref{torsionfree}.
This will indeed suffice since congruence subgroups which are contained in one another are of finite index and $G_n \cap G_m = G_{\mathrm{lcm}(n,m)}$.

We expand the isomorphism $\Lambda$ of \eqref{Lambda iso} to an isomorphism $\Lambda$ between $\SL_2(\mathcal{H}_{u,v}(\Q))$ and $\SL_2\left(\qa{u}{v}{\Q}\right)$.
Setting $\mathcal{G}_n = \Lambda^{-1}(G_n)$, the $n^{\text{th}}$ principal congruence subgroup of $\SL_2\left(\Lambda^{-1}(\LL_{u,v})\right)$, it suffices to prove the desired statement for $\mathcal{G}_n$. Remark that $\mathcal{G}_n$ is also torsion-free.

By \Cref{Iso_SL+Gamma4R}, one deduces that for the explicit isomorphism $\chi \colon \SL_+(\Gamma_4(\R)) \rightarrow \SL_2\left( \qa{-1}{-1}{\R}\right)$, one has that $\chi(\con_n) \leq \mathcal{G}_n$ and thus that $\mathrm{Con}_n$ is torsion-free.

Even more so, $\chi\left(\SL_+\left(\Gamma_4^{u,v}(\Z)\right)\right) \leq \SL_2\left(\Lambda^{-1}(\LL_{u,v})\right)$, and, using the classical argument that the units of orders are commensurable, this index is finite.
However, since $\con_n$ is a finite index subgroup of $\SL_+\left(\Gamma_4^{u,v}(\Z)\right)$ and the same holds for $\mathcal{G}_n$ in $\SL_2\left(\Lambda^{-1}(\LL_{u,v})\right)$, we obtain that $[\mathcal{G}_n : \chi(\con_n)] < \infty$.
\Cref{vFq for Con} now finishes the proof.
\end{proof}

\subsection{Congruence kernel of higher modular groups}\label{congruence kernel qa}

Let $F$ be a number field and $S$ a non-empty finite set of places of $F$ containing the Archimedean places. Associated is the ring of $S$-integers $\mc{O}_S = \{ x\in F \mid |x|_v \leq 1 \text{ for all } v \notin S \}$. Now consider a linear algebraic $F$-group $\mathbf{G}$ and fix an $F$-embedding $\mathbf{G} \hookrightarrow \GL_n(F)$. Using this, the group of $S$-integral points is defined as $\mathbf{G}(\mc{O}_S) := \mathbf{G}(F) \cap \GL_n(\mc{O}_S)$. Any subgroup $\Gamma$ of $\mathbf{G}(F)$ commensurable with $\mathbf{G}(\mc{O}_S)$ is called an \emph{$S$-arithmetic subgroup}. In this generality, a principal congruence subgroup of $\Gamma$ is a subgroup of the form $\Gamma(J):=\Gamma \cap \mathbf{G}(J)$ for an ideal $J$ of $\mc{O}_S$, and 
\[\mathbf{G}(J) := \mathbf{G}(\O_S) \cap \SL_n(J) = \{ g \in \mathbf{G}(\O_S) \mid g \equiv 1 \text{ mod } J\}.\]
As $\O_S/J$ is finite, all the groups $\Gamma(J)$ have finite index in $\Gamma$. 

The congruence topology on $\Gamma$ has the subgroups $\Gamma(J)$ as basis of neighbourhoods of the identity, whereas the profinite topology considers all finite index subgroups. Denote by $\widetilde{\Gamma}$ and $\widehat{\Gamma}$ the respective completions. The \emph{subgroup congruence problem} (\textbf{SCP}) asks whether all finite index subgroups contain a principal congruence subgroup. In other words, whether the natural epimorphism 
$$\pi_{\Gamma} : \widehat{\Gamma} \rightarrow \widetilde{\Gamma}$$
is injective. The kernel $C(\Gamma) := \ker(\pi_{\Gamma})$ is called the \emph{congruence kernel}. Serre formulated a quantitative version of (\textbf{SCP}) by asking to compute $C(\Gamma)$ and conjecturing that it is finite if and only if the $S$-rank of $\Gamma$ is at least two.\medskip

In our setting, we consider a finite-dimensional division $\Q$-algebra $D$, a free $D$-module $V$ of rank two and $\mathbf{G} = \GL_{V}$ the algebraic $\Q$-group of automorphisms of the $D$-module $V$. Hence for any $\Q$-algebra $A$, the group $\GL_V(A)$ is the group of automorphisms of the $(D \ot_{\Q} A)$-module $V\ot_{\Q} A$. More precisely, $D= \qa{u}{v}{\Q}$ for some strictly negative integers $u,v \in \Z_{<0}$ and $\Gamma$ is any arithmetic subgroup of $\mathbf{G}(\Q)$. For convenience, we consider an order $\O$ in $D= \qa{u}{v}{\Q}$ (which were always assumed to be $\Z$-orders) and assume that $\Gamma$ is a finite index subgroup of $\SL_2(\O)$. 

In case that $\Gamma = \SL_2(\O)$, then by \cite[Theorem (22.4)]{Reiner} $\widetilde{\SL_2(\O)} = \SL_2(\widehat{\O})$ where $\widehat{\O}$ is the profinite completion of $\O$, i.e. $\widehat{\O}$ is the inverse limit of all $\O/J$ with $J$ a non-zero ideal of $\O$. The ring $\widehat{\O}$ is a profinite ring and so $\GL_2(\widehat{\O})$ is a profinite group.

We denote by $F_{\omega}$ the free group of countable rank and by $\widehat{F_{\omega}}$ its profinite completion. 

\begin{theorem}\label{congruence kernel over quaternion}
Let $u,v \in \Z_{<0}$ be strictly negative integers, $\O$ an order in $\qa{u}{v}{\Q}$ and $\Gamma$ a finite index subgroup of $\SL_2(\O)$. The congruence kernel $C(\Gamma)$ contains a closed subgroup isomorphic to $\widehat{F_{\omega}}$.
\end{theorem}

Via the comparison map $\chi: \SL_+(\Gamma_4(\R)) \rightarrow \SL_2\left( \qa{-1}{-1}{\R}\right)$, \Cref{infinite abelianization SL} was obtained as a consequence of the similar statement for the congruence subgroups $\con_n$ of the higher modular groups $\SL_+\left(\Gamma_4^{u,v}(\Z)\right)$, i.e. \Cref{vFq for Con}. This time, through that same comparison, the analogue of \Cref{congruence kernel over quaternion} for the higher modular groups follows as a corollary. More precisely, the principal congruence subgroups of $\SL_+\left(\Gamma_4^{u,v}(\Z)\right)$ are by definition the finite index subgroups $\con_n$. Hence we can consider the completions $\widehat{\SL_+\left(\Gamma_4^{u,v}(\Z)\right)}$ and $\widetilde{\SL_+\left(\Gamma_4^{u,v}(\Z)\right)}$ as above and the associated natural epimorphism and its kernel, again called the congruence kernel. 

\begin{corollary}\label{congruence kernel higher modular}
    Let $u,v \in \Z_{<0}$ be strictly negative integers, then the congruence kernel of $\SL_+\left(\Gamma_4^{u,v}(\Z)\right)$ contains a closed subgroup isomorphic to $\widehat{F_{\omega}}$.
\end{corollary}

\begin{proof}[Proof of \Cref{congruence kernel over quaternion}]
The proof proceeds as for \cite[Theorem B]{Lubot82}. Namely, reading between the lines in \textit{loc. cit.}, one obtains a reduction to the following claim.\smallskip

\noindent \underline{Claim:} Let $X$ be a set with $|X|\geq 2$ and $n\geq 1$. The profinite group $\widehat{F(X)}$ is not embedded as a profinite group in $\widetilde{\SL_n(\O)}= \SL_n(\widehat{\O})$.

We now explain how the claim indeed implies \Cref{congruence kernel over quaternion}. Let $\Gamma_f$ be a finite index subgroup of $\Gamma$ mapping onto a free group $F_m$ of rank $m \geq 2$, as yielded by \Cref{infinite abelianization SL}. It is well-known that this implies that $\widehat{\Gamma_f} \leq \widehat{\SL_2(\O)}$ has a subgroup $Z$ isomorphic to $\widehat{F_m}$. Indeed, $\widehat{F_m}$ being a free profinite group, the epimorphism from $\Gamma_f$ onto $F_m$ splits. The claim now implies that $Z \cap C(\Gamma)$ is non-trivial as otherwise $Z\cong \widehat{F_m}$ would embed into $\widehat{\Gamma}/C(\widehat{\Gamma}) = \widetilde{\Gamma} \leq \widetilde{\SL_2(\O)}$, which would contradict the above claim. 

It follows from \cite[Theorem 2.1.(a)]{Lubot82} that the non-trivial normal subgroup $Z \cap C(\Gamma)$ of $Z\cong \widehat{F_m}$ contains $\widehat{F_{\omega}}$. In particular also $C(\Gamma)$ contains $\widehat{F_{\omega}}$, as desired. Thus it indeed remains to prove the claim. \smallskip

First note that it is enough to obtain the claim in case that $\O$ is a maximal order, which we assume from now on. Suppose that $\widehat{F(X)}$ embeds continuously into $\SL_n(\widehat{\O})$. Since every finite group is a quotient of a subgroup of $\widehat{F(X)}$, the same holds for $\SL_n(\widehat{\O})$. We will show that the latter however does not hold. If $H$ is a quotient of a subgroup of $\SL_n(\widehat{\O}) = \varprojlim \SL_n(\O/J),$ then it must be of some $\SL_n(\O/J)$. Since $J$ is a two-sided ideal in a maximal order (over a Dedekind domain) one can decompose $J$ into a product of prime ideals \cite[Theorem (22.10)]{Reiner}, say $J = \prod_{i=1}^{\ell} P_i^{k_i}$. Hence $H$ is the quotient of a subgroup of some $\SL_n(\O/P_i^{m_i})$. Now, if $H$ is a non-abelian simple group, then it does not intersect the kernel of the map from $\SL_n(\O/P_i^{m_i})$ to $\SL_n(\O/P_i)$, as the kernel is nilpotent. Hence such $H$ is the quotient of a subgroup of some $\SL_n(\O/P_i)$. By \cite[Theorem (22.3)]{Reiner} $\O/P_i$ is a finite-dimensional simple algebra over the residue class field $\Z/(P_i \cap \Z)$. Moreover, $P_i$ lies over some prime ideal $(p_i)$ of $\Z$. Then, $p_i\O = P_i^{m_i}$ with $m_i$ the index of the division algebra $\qa{u}{v}{\Q_{p_i}}$.  

If $\qa{u}{v}{\Q_{p_i}} \cong \Ma_2(\Q_{p_i})$ is split, then $\O/P_i \cong \Ma_2(\F_p)$ by \cite[Theorem (22.4)]{Reiner}. Furthermore, if $\qa{u}{v}{\Q_{p_i}}$ is still a division algebra, then $\O/P_i \cong \F_{p^2}$. In conclusion, $\SL_n(\O/P_i)$ is of the form $\SL_{dn}(\F_{p^{t}})$ for some $d,t \in \N$. It follows from \cite[Lemma 2.4.(b)]{Lubot82} that there is a finite non-abelian simple group which is not the quotient of subgroup of $\SL_{dn}(\F_{p^{t}})$ for any $p$, yielding the desired contradiction and as such finishing the proof.
\end{proof}
\section{\texorpdfstring{Applications: $\Mexc$ through $\vQL$ and congruence kernel unit group}{Applications: Mexc through vQL and congruence kernel unit group}}\label{subsection vFQ vsp}

Let $\O$ be an order in an exceptional matrix algebra $A$. In \Cref{sectie vFQ} we obtained that $\SL_1(\O)$ has a finite index subgroup mapping onto a non-abelian free group. Subsequently, the latter property was used to obtain information on the congruence kernel of $\SL_1(\O)$. In this section we apply the aforementioned results to the unit group of $RG$ with $R$ the ring of integers in a number field.

\subsection{\texorpdfstring{Characterisation of $\Mexc$ via largeness of (higher) modular groups}{Characterisation of Mexc via largeness of (higher) modular groups}}

Recall that we say that a group $\Gamma$ \emph{has property $\vQL$} if it has a finite index subgroup which maps onto a non-abelian free group. This property is preserved under commensurability. Now consider the following class of groups:
\[\mc{G}_{\vQL} :=  \left\{ \Gamma \mid \Gamma \text{ is virtually indecomposable and has } \vQL  \right\} \sqcup \{ \text{ finite groups } \},\] 
and the associated class $\prod \mc{G}_{\vQL}$ defined in \eqref{prod G}. Recall that a group $\Gamma$ is called \emph{virtually indecomposable} if $\Gamma$ can only be commensurable to a direct product $\Gamma_1 \times \Gamma_2$ if $\Gamma_1$ or $\Gamma_2$ is finite.

\begin{theorem}\label{Mexc versus vFQ}
  Let $G$ be a finite group, $F$ a number field and $R$ its ring of integers. Then the following are equivalent:
\begin{enumerate}
    \item \label{item: vFQ} $\U(RG)$ is virtually-$\prod \mc{G}_{\vQL}$.
    \item \label{item: mexc vFQ} $FG$ has $\Mexc$ and no exceptional division components.
\end{enumerate}
Moreover, $FG$ has $\Mexc$ if and only if $\SL_1(RGe)$ has $\vQL$ for all $e  \in  \PCI(FG)$ such that $FGe$ is not a division algebra.
\end{theorem}

A main ingredient for the equivalence between \eqref{item: vFQ} and \eqref{item: mexc vFQ} is \Cref{virutally G-am imply special components}. In particular, we need to show that the class of virtually-$\prod \mc{G}_{\vQL}$ groups satisfies the necessary properties. The proof of \cite[Theorem 2.1]{FreebyFree} is particularly instructive for the direct factor property.

\begin{lemma}\label{class nicely closed}
  With notations as above, the class of groups which are virtually-$\prod \mc{G}_{\vQL}$ satisfies the following:
  \begin{itemize}
      \item If $\Gamma_1$ and $\Gamma_2$ are commensurable and $\Gamma_1$ is virtually-$\prod \mc{G}_{\vQL}$, then so is $\Gamma_2$.
      \item If $\prod_{i=1}^m \Gamma_i \times \Gamma$ is virtually-$\prod \mc{G}_{\vQL}$, where $\Gamma_i$ is virtually indecomposable and $\Gamma$ finitely generated abelian, then each $\Gamma_i$ is either finite or virtually-$\prod \mc{G}_{\vQL}$.
  \end{itemize}
\end{lemma}
\begin{proof}

Firstly note that the class $\mc{G}_{\vQL}$ is closed under commensurability. This implies that the class of virtually-$\prod \mc{G}_{\vQL}$ groups is so as well. Indeed, suppose that $\Gamma_1$ and $\Gamma_2$ are two groups such that $\Gamma_1 \cap \Gamma_2$ has finite index in both $\Gamma_i$. Clearly if $\Gamma_1 \cap \Gamma_2$ is virtually-$\prod \mc{G}_{\vQL}$, then so is each $\Gamma_i$. Hence we may assume that $\Gamma_2$ is a finite index subgroup of $\Gamma_1$ and it remains to consider the case that $\Gamma_1$ is virtually-$\prod \mc{G}_{\vQL}$. Let $\prod_{i=1}^q H_i$ be a finite index subgroup of $\Gamma_1$ with $H_i$ either finitely generated abelian or in $\mc{G}_{\vQL}$. Since $H_i \cap \Gamma_2$ has finite index in $H_i$, and $\mc{G}_{\vQL}$ is closed under commensurability, we obtain that $H_i \cap \Gamma_2$ is either finitely generated abelian or in $\mc{G}_{\vQL}$. Therefore $\prod_{i=1}^q (H_i \cap \Gamma_2)$ is a finite index subgroup of $\Gamma_2$ and a member of $\prod \mc{G}_{\vQL}$, proving the claim.\smallskip

Now we consider the closedness under direct factors as in \Cref{remark on vsp lemma}. Let $\prod_{i=1}^m \Gamma_i \times \Gamma$ be virtually-$\prod \mc{G}_{\vQL}$, with $\Gamma$ a finitely generated abelian group and $\Gamma_i$ non-abelian virtually indecomposable groups.

If $\prod_{i=1}^m \Gamma_i$ is finite, then there is nothing to prove. Hence assume that some factors $\Gamma_i$ are infinite. Now take a finite index subgroup $H$ which is in $\prod \mc{G}_{\vQL}$. By possibly going over to a finite index subgroup of $H$, we may assume that $H$ decomposes as \[H = \Z^{\ell} \times \prod_{x \in X} H_x,\] with $H_x$ mapping onto a non-abelian free group and $X \neq \emptyset$. Note that each $H_x$ is infinite.

Denote by $\pi_x$ the projection of $H$ onto $H_x$. Further, consider $S_i := \Gamma_i \cap H$ for each $1 \leq i \leq m$ and $B := \Gamma \cap H$. Since $H$ has finite index in $\prod_{i} \Gamma_i \times \Gamma$, the group $S_i$ has finite index in $\Gamma_i$ and $S := \prod_{i=1}^m S_i \times B$ has finite index in $H$. In particular, $S_i$ is virtually indecomposable. 

For the remainder of the proof fix some $1 \leq i \leq m$ such that $\Gamma_i$ is infinite. We need to show that $\Gamma_i$ is virtually-$\prod \mc{G}_{\vQL}$. As $S_i$ has finite index in $\Gamma_i$, it is enough to show that $S_i$ is virtually-$\prod \mc{G}_{\vQL}$. Since $S_i$ is infinite, there is some $x \in X$ such that $\pi_x(S_i)$ is infinite. For such an $x$, the group $H_x$ maps onto a non-abelian free group, say $F_{n_x}$, via an epimorphism $\psi_x \colon H_x \rightarrow F_{n_x}$. If $\psi_x(\pi_x(S_i))$ is a non-abelian free group, then $S_i$ would be in $\mc{G}_{\vQL}$, as desired. Thus we are reduced to the case that  whenever $\pi_x(S_i)$ is infinite, then $\psi_x(\pi_x(S_i))$ is not a non-abelian free group. Hence $\psi_x(\pi_x(S_i))$ is either trivial or infinite cyclic. \smallskip

\noindent \underline{Claim:} If $\pi_x(S_i)$ is infinite, then $\pi_x\left( \prod_{j \neq i} S_j \times B\right )$ is finite.  \smallskip

Denote $L := \prod_{j \neq i} S_j \times B$. As $S$ has finite index in $H$, also $\pi_x(S)$ has finite index in $\pi_x(H) = H_x$. Since $S_i$ commutes with $L$, it follows that $\pi_x(S_i) \cap \pi_x(L) \subseteq \ZZ(\pi_x(S))$. Now note that the finite index subgroup $\pi_x(S)$ of $H_x$ also maps onto a non-abelian free group under $\psi_x$, which we denote $F$. Therefore, $\ZZ(\pi_x(S))$ must be contained in $\ker(\psi_x)$. This implies that we obtain 
\[\psi_x(\pi_x(S))=\psi_x(\langle \pi_x(S_i), \pi_x(L)\rangle)  \cong \psi_x(\pi_x(S_i)) \times \psi_x(\pi_x(L)),\]
as a finite index subgroup of $F_{n_x}$. As $n_x \geq 2$, this means that either $\pi_x(S_i)$ or $\pi_x(L)$ is fully contained in $\ker(\psi_x)$. However, if $\pi_x(L) \subseteq \ker(\psi_x)$, as $\psi_x(\pi_x(S_i))$ is assumed at most infinite cyclic, then $\psi_x(\pi_x(S))$ would be at most infinite cyclic. This is a contradiction since $\pi_x(S)$ has finite index in $H_x$. In conclusion, $\pi_x(S_i) \subseteq \ker(\psi_x)$. Now using that free groups are the projective objects in the category of groups, we obtain a splitting $\pi_x(S) \cong K \rtimes F$ with $K$ containing $\pi_x(S_i)$ and $F$ isomorphic to a subgroup of $\pi_x(L)$. However, from the definition of $S$, such a splitting in fact yields a direct product $\pi_x(S) = \pi_x(S_i) \times \pi_x(L)$, which contradicts the virtual indecomposability of $H_x$, finishing the proof of the claim.\smallskip

Finally, consider the set $Y = \{x \in X \mid \pi_x(S_i) \text{ is infinite } \}$. The above claim implies that if $x \in Y$, then $\pi_x(S_i)$ has finite index in $H_x$, since $\pi_x(S) = \langle \pi_x(S_i), \pi_x(\prod_{j \neq i} S_j \times B) \rangle$. Note that $S_i$ contains a finite index subgroup $R_i$ with the property that $\pi_x(R_i)=1$ for $x \notin Y$. Since $S_i$ is virtually indecomposable, so is $R_i$. But by the above, $R_i$ is a finite index subgroup of $\prod_{x\in Y} H_x$, which would yield a virtual decomposition of $R_i$. Thus $Y= \{ x_0\}$ is a singleton and therefore $R_i$ inherits from $H_{x_0}$ that it is in $\mc{G}_{\vQL}$. In conclusion, $S_i$ and hence $\Gamma_i$ is virtually-$\mc{G}_{\vQL}$, finishing the proof. 
\end{proof}

\begin{proof}[Proof of \Cref{Mexc versus vFQ}]
Let $D$ be a finite-dimensional division algebra, $\O$ an order in $D$ and $\Gamma$ a finite index subgroup in $\SL_n(\O)$ with $n \geq 2$. Then the moreover part of \Cref{Mexc versus vFQ} follows from the following:
\begin{equation}\label{vFQ for simple}
\Gamma \text{ has } \vQL \iff \Ma_n(D) \text{ is exceptional.}
\end{equation}
Indeed, if $\Ma_n(D)$ is not exceptional, then $\Gamma$ satisfies property (FAb) by Margulis's normal subgroup theorem \cite{MargulisBook} and hence cannot satisfy largeness. If $\Ma_n(D) = \Ma_2(\Q(\sqrt{-d}))$ for $d \in \N$, then it has $\vQL$ by \cite{GS} (or \cite[Theorem 3.7]{Lubot96}). If $\Ma_n(D) = \Ma_2(\qa{-a}{-b}{\Q})$, then $\vQL$ is given by \Cref{infinite abelianization SL} and the fact that $\vQL$ is preserved by commensurability. This proves \eqref{vFQ for simple} and hence the moreover part.\smallskip

Now we focus on the equivalence between \eqref{item: vFQ} and \eqref{item: mexc vFQ}. By \Cref{class nicely closed} the class virtually-$\prod \mc{G}_{\vQL}$ satisfies the necessary properties to apply \Cref{virutally G-am imply special components}. Consequently, $\U(RG)$ is virtually-$\prod \mc{G}_{\vQL}$ if and only if $\SL_1 (R Ge)$ is either finite or virtually-$\mc{G}_{\vQL}$ for each $e \in \PCI(FG)$. The former only occurs if $FGe$ is a field or totally definite quaternion algebra \cite[Proposition 5.5.6]{EricAngel1}. To understand when $\SL_1 (R Ge)$ is virtually-$\mc{G}_{\vQL}$ one needs to distinguish whether $FGe$ is a division algebra or not. 

If $FGe$ is not a division algebra, then $\SL_1(RGe)$ has $\vQL$ if and only if $FGe$ is an exceptional matrix algebra by \eqref{vFQ for simple}. Moreover $\SL_1(RGe)$ is virtually indecomposable by \cite[Theorem 1]{KleRio}. This finishes the proof that \eqref{item: vFQ} implies that $FG$ has $\Mexc$. Therefore, now suppose that $FGe$ has is one of the possible non-commutative division algebra component of a group algebra with $\Mexc$. Then $D$ is either exceptional of the form $\qa{\zeta_{2^t}}{-3}{\Q(\zeta_{2^t})}$ with $t \geq 4$ or it is a totally definite quaternion algebra. In the latter case $\U(\O)$ is finite by \cite{Kleinert}. In the former case the $S$-rank, with $S$ the set of infinite places, is at least $2$ and hence by Margulis's normal subgroup theorem $\SL_1(\O)$ has property (FAb) and thus not $\vQL$. This both finishes the proof that \eqref{item: vFQ} implies \eqref{item: mexc vFQ} and the converse.
\end{proof}

\subsection{Congruence kernel of group rings}

Let $A$ be a finite-dimensional semisimple $\Q$-algebra. Then  
\[
A = \End(V_1) \times \cdots \times \End(V_q),
\]
with $V_i$ an $n_i$-dimensional module over some finite-dimensional division $\Q$-algebra $D_i$, $i = 1, \ldots, q$. 
In consequence, the $\Q$-group of units of $A$ identifies with the reductive group
\begin{align*} \label{eq:decompositionunitsalgebra}
\mathbf{G} &= \GL_A = \GL_{D_1^{n_1}} \times \cdots \times \GL_{D_m^{n_m}}. 
\end{align*}
If $\O$ is an order in $A$, then $\SL_1(\O)$ is an arithmetic subgroup of $\SL_1(A)$. In \Cref{congruence kernel qa} we have recalled the definition of the congruence kernel of $\SL_1(\O)$. For ease of the reader, we describe in the setting of the unit group of an order in a semisimple algebra the definitions with down-to-earth terminology.  

A \emph{principal congruence subgroup} of $\SL_1(\O)$ is a subgroup of the form
\[\SL_1(\O,m) := \{ u \in \SL_1(\O) \mid u-1 \in m\O \},\]
with $m \in \N_0$. 
Denote by $\mc{N}$ the set of all finite index normal subgroups of $\SL_1(\O)$ and by $\mc{C}o$ the set of all principal congruence subgroups. Both sets define a basis of of neighbourhoods of $1$ for group topologies in $\SL_1(\O)$ and hence we have associated completions:
\[\widehat{\SL_1(\O)} := \varprojlim\limits_{N \in \mc{N}} \SL_1(\O)/N \,\, \text{ and }\,\, \widetilde{\SL_1(\O)} := \varprojlim\limits_{m \in \N} \SL_1(\O)/\SL_1(\O,m).\]
Since $\mc{N}$ contains the set $\mc{C}o$, the identity map on $\SL_1(\O)$ induces a surjective group homomorphism $\pi_{\O}$ from $\widehat{\SL_1(\O)}$ to $\widetilde{\SL_1(\O)}$. The kernel $\ker(\pi_{\O})$ is called the \emph{congruence kernel}. Note that if $A = \Q G$ and $\O = \Z G$, then the congruence kernel is simply the kernel of the natural morphism $\widehat{\SL_1(\Z G)} \rightarrow \SL_1(\widehat{\Z} G)$, which we denote by $C_{G}$.\smallskip

The congruence kernel of an integral group ring has been investigated in \cite{CdR1, CdR2, CdRZ}. In \emph{loc. cit.} several families of finite groups $G$ are described with the property that if $C_G$ is infinite, then $G$ is a member of one these families.  Under the assumption that $\Q G$ has no exceptional division components, an upper-bound on the virtual cohomological dimension of $C_G$ is obtained in \cite[Theorem 5.1]{CdRZ}. We contribute towards bounding from below, by showing that whenever $\Q G$ has an exceptional matrix component, then the congruence kernel contains a profinite free group of countable rank. This is a direct consequence of \cite[Theorem B]{Lubot82}, \Cref{congruence kernel over quaternion} and the well-behavedness of the congruence kernel with respect to direct products.

    \begin{corollary}\label{congruence kernel with excep component}
        Let $G$ be a finite group. Then the following hold:
        \begin{enumerate}
            \item If $\Q G$ has an exceptional matrix algebra, then $C_{G}$ contains $\hat{F_{\omega}}$,
            \item If $\Q G$ has $\Mexc$, then $C_{G}$ contains $\prod_{e \in \PCI_{\neq 1}} \hat{F_{\omega}}$.
        \end{enumerate}
    where $\PCI_{\neq 1} := \{ e \in \PCI(\Q G) \mid \Q Ge \text{ is not a division algebra}\}$.
    \end{corollary}

\begin{remark}
    If $\Q G$ has no exceptional division components, then $C_{G}$ is infinite if and only if $\Q G$ has an exceptional matrix component. As pointed out in \Cref{good and exc div}, for division algebra components it is still open whether $\SL_1(D)$ satisfies the subgroup congruence problem. However, they are expected to do so whenever their $S$-rank, with $S$ the set of infinite places of $\ZZ(D)$, is at least $2$. Under that condition one can already apply Margulis's normal subgroup theorem, yielding that every non-central normal subgroup has finite index (in particular yielding property (FAb)). 

  It is also not known if $S$-rank $1$ for $\SL_1(D)$ implies the existence of a large congruence subgroup, such as in \Cref{infinite abelianization SL}. In fact even infinite abelianisation is still open. However, in case that $D$ is a quaternion algebra over a number field, more is known. Indeed, whenever Margulis's theorem does not apply, then the corresponding symmetric space is either the hyperbolic plane $\mathbb{H}^2$ or the hyperbolic $3$-space $\mathbb{H}^3$. Then a torsion-free finite index subgroup of $\SL_1(\mc{O})$ is either the fundamental group of a compact surface, hence trivially has infinite abelianisation, or the fundamental group of a compact hyperbolic $3$-manifold. In the second case, Agol's positive answer of the virtual Hacken conjecture tells us that there is a finite index subgroup with infinite abelianisation. It would be interesting to know whether some congruence subgroup is large. 
\end{remark}

\section{The blockwise Zassenhaus and subgroup isomorphism property}\label{block zass and iso}

Three conjectures concerning finite subgroups $H$ of the normalised unit group $V(\Z G)$ of an integral group ring were attributed to Zassenhaus \cite{Zass,SehAntwerp}. 
The strongest (referred to as the \emph{third}) Zassenhaus conjecture predicts that $H$ is conjugated over $\Q G$ to a subgroup of $G$. The \emph{second} asserts the same, but only under the hypothesis that $|H|=|G|$, and the \emph{first} (and weakest) concerns the case that $H$ is cyclic. All of these conjectures have been disproven \cite{Scott, Rogg, EisMar}, transforming them into the question for which $H$ and $G$ the above conjugation property holds. In \Cref{Section what is block Zassenhaus} we introduce a \emph{blockwise variant} of the conjectures, i.e. investigating whether the image of $H$ under any irreducible $\Q$-representation $\rho$ of $G$ is conjugated in $\rho(\Q G)$ to a subgroup of $\rho(G)$. Furthermore, we define the \emph{Zassenhaus property} for semisimple $F$-algebras with $F$ a field of characteristic $0$. The rest of the section is devoted to the investigation of the Zassenhaus property for exceptional matrix algebras. Among other things, in \Cref{sectioin sbgrps excetpional} we describe the conjugacy classes of finite subgroups of exceptional matrix algebras.

\noindent {\it Convention:} With a \emph{group basis} of a group ring $FG$, we will mean a finite subgroup $\Gamma$ of the normalised unit group of $FG$ which forms a $F$-basis of $FG$.

\subsection{The Zassenhaus property for semisimple algebras}\label{Section what is block Zassenhaus}

Let $A$ be a finite-dimensional semisimple algebra over a field $F$, with $\Char(F) = 0$. Let $R$ be a subring of $A$. We introduce the following definition.
\begin{definition}
     The \emph{set of spanning (sub)groups in $R$} is the set of isomorphism classes of finite subgroups of $\U(R)$ such that their $F$-linear span equals $A$:
    \begin{equation}
        \mc{S}_F(R) := \left\{ \Gamma \leqslant \U(R) \mid \abs{\Gamma} < \infty, \: \Span_{F}\{g \in \Gamma\} = A\right\}/\cong.
    \end{equation}
If $F = \Q$ we will omit the index and write $\mc{S}(R)$.
\end{definition}
Note that for most subrings $R\leqslant A$, the set $\mc{S}_F(R)$ is empty. However, if $R$ is an order in $A$, or $A = R$ itself, then by definition $\mc{S}_F(R)\neq \emptyset$.

\begin{definition}\label{def Zassenhaus prop}
    Let $R$ be a subring of $A$ such that $\mc{S}_F(R)$ is non-empty. Furthermore, let $\mc{G}$ be a set of finite subgroups of $\U(R)$, and $[\Gamma] \in \mc{S}_F(R)$ an isomorphism class of spanning subgroups. Then $R$ is said to have \emph{the subgroup isomorphism property relative to $\Gamma$ and $\mc{G}$} if all $H \in  \mc{G}$ are isomorphic to a subgroup of $\Gamma$. If there moreover exists an element $\alpha_H \in \U(A)$ such that 
    \[H^{\alpha_H} \leqslant \Gamma,\]
    then $R$ is said to have \emph{the Zassenhaus property relative to $\Gamma$ and $\mc{G}$}.
\end{definition}

\begin{example}\label{Zc conj as instance}
    With the above terminology, the Zassenhaus conjectures may be reformulated as follows. Let $A = \Q G$ and $R = \Z G$. Then $R$ satisfies the \emph{$i$\textsuperscript{th} Zassenhaus conjecture} if it has the Zassenhaus property relative to $G$ and $\mc{G}_i$, where 
    \begin{align*}
        \mc{G}_1 & := \left\{ H \leqslant V(\Z G) \mid H \text{ cyclic and finite}\right\}, \\
        \mc{G}_2 & := \left\{ H \leqslant V(\Z G)  \mid \abs{H} = \abs{G}\right\}, \\
        \mc{G}_3 & := \left\{ H \leqslant V(\Z G) \mid H \text{ finite} \right\}.
    \end{align*}

The subgroup isomorphism property for the sets $\mc{G}_i$, with $i=1,2,3$, has also been investigated, see \cite{Margolis,PdRV} and references therein. The subgroup isomorphism property for the set $\mc{G}_2$ and twisted group rings has recently also received considerable attention, see \cite{MaSc, MaSc2, MaSc3, HKSa, HKSi}. 

Also, implicitly the blockwise Zassenhaus property has been considered in \cite[Lemma 3]{HerSin}.
\end{example}

\begin{remark}\label{independant of representative}
We defined the subgroup isomorphism property and Zassenhaus property relative to a representative $\Gamma$ of the isomorphism class $[\Gamma]$. The subgroup isomorphism property clearly does \emph{not} depend on this choice of representative. If $A$ is a simple algebra, \Cref{upgrade iso to conjugation} will imply that isomorphic spanning groups are conjugated in $\U(A)$. Hence, enjoying the Zassenhaus property is also independent of the choice of representative in case of a simple algebra. As shown by the next result, the latter also holds when $A$ is a group ring and $\mc{G}$ consists of cyclic subgroups. We are grateful to Leo Margolis for providing us with a proof, which he learned from Wolfgang Kimmerle. 
\end{remark}

\begin{proposition}
    \label{folklore remark}
    Suppose $G$ is a finite group and $R$ a $G$-adapted\footnote{Recall that a $G$-adapted ring is an integral domain $R$ of characteristic $0$ such that no prime divisor of $|G|$ is invertible in $R$.} ring with field of fractions $F$. Let $H \leqslant V(RG)$ be a finite cyclic group, and suppose there exists a group basis $\Gamma$ of $RG$ and an $y \in \U(FG)$ such that $H \leqslant \Gamma^y$. Then $H$ is conjugated by an element of $\U(FG)$ to a subgroup of $G$.
\end{proposition}
\begin{proof}
    Suppose first that $H \leqslant \Gamma$. Denote $H = \langle h \rangle$. Then from \cite[Theorem V.1]{RoggenkampTaylor} it follows that there is a bijection $\varphi \colon \Gamma \to G$ such that for all $x \in \Gamma$ the following holds: (i) $var\phi(x)$ and $x$ have the same order, (ii) $\varphi(x)^n$ and $\varphi(x^n)$ are conjugate for all $n \in \Z$, and (iii) $\chi(x) = \chi(\phi(x))$ for any irreducible $K$-character, with $K$ some field containing $F$. Then fixing $x \in \Gamma$ such that $\phi(x) = h$, the restriction $\varphi|_{\langle x \rangle} \colon \langle x \rangle \to H$ is a group isomorphism. Moreover, for any $\chi$ as above and any $j \in \Z$, $\chi(x^j) = \chi(h^j)$. It follows from\footnote{Note that the statement is only formulated for $F = \Q$, but the proof can also be generalised to this context, by replacing $\C$ in the proof of \cite[Lemma 37.6]{SehgalBook93} with a splitting field of $G$ containing $F$.} \cite[Lemma 37.6]{SehgalBook93} that $x$ and $h$ are conjugate over $FG$. 
    
    Now suppose that $H \leq \Gamma^y$ for $y \in \U(FG)$. Then $H^{y^{{-1}}} \leqslant \Gamma$, and hence by the above $H^{y^{{-1}}}$ is $FG$-conjugated to a subgroup of $G$. In particular, so is $H$, finishing the proof.
\end{proof} 

\begin{remark}
One could wonder whether there is a version of \Cref{folklore remark} with $\Gamma^y$ replaced by any group basis of $\Q G$. This however seems not to be possible. Concretely, Leo Margolis communicated to the authors that the cyclic subgroup $C_n$ serving as counterexample to the first Zassenhaus conjecture (see \cite{EisMar}) is expected not to be contained in any group basis of $\Q G$. It would be interesting to investigate whether there is some $e \in \PCI(\Q G)$ such that $C_ne$ is contained in no spanning subgroup of $\U(\Q Ge)$ of the form $\Gamma^y$ with $\Gamma \leq V(\Z G)$ and $y \in \Q Ge$. In particular, the following is natural to ask.
\end{remark}

\begin{question}\label{no first iso but block yes}
Does the counterexample to the first Zassenhaus conjecture in \cite{EisMar} satisfy the blockwise Zassenhaus property relative to cyclic subgroups?
\end{question}

In practice there are obvious conditions one needs to specify on the groups $H \in \mc{G}$ so that $\O$ can have the Zassenhaus or subgroup isomorphism property relative to $\mc{G}$. For instance, if $H$ is isomorphic to a subgroup of $\Gamma\in B$, then $|H|$ divides $|\Gamma|$ and $\exp(H) \mid \exp(\Gamma)$. However, in case that $A$ is a group algebra these conditions are always satisfied (see \Cref{blockwise Cohn-Livingstone}) and hence do not need to be specified in \Cref{Zc conj as instance}. 

Note that for any groups $H$ and $\Gamma$, one has that $\exp(H) \mid \exp(\Gamma)$ holds if the set of orders of prime-power order elements of $H$ is included in that of $\Gamma$. More generally, for a group $\Gamma$ we call the set
\begin{equation}\label{def: spectrum}
\spec(\Gamma) := \{ n \in \N \mid \exists g \in \Gamma : o(g) = n\}
\end{equation}
of orders of elements in $\Gamma$ \emph{the spectrum} of $\Gamma$.  A main problem in group rings\footnote{In the literature the problem is formulated for $R = \Z$. However, the question makes sense in this generality, since for such rings of integers there is a variant of the Cohn--Livingstone theorem.}, is the \emph{spectrum problem} (\textbf{SP}):
\begin{equation}
\label{spec prob}
    \tag{\textbf{SP}} \spec(G) = \spec(V(RG))?
\end{equation}
The spectrum problem is the weaker version of the first Zassenhaus conjecture where ‘conjugate’ is replaced by ‘isomorphic to’. In contrast to the Zassenhaus conjecture, the spectrum problem is still open. In \cite{HertOrder}, it was notably proven to be true for $G$ solvable. Moreover, if $R$ is the ring of integers in a number field, a theorem of Cohn--Livingstone \cite{CL} yields that the unit group $V(RG)$ contains an element of prime power order $p^n$ if and only if $G$ does.
In particular, 
\begin{equation} 
\label{exp cohn livingstone}\exp(V(RG)) = \exp(G).
\end{equation}

\begin{definition}\label{Def: admissible sets}
Let $\Gamma \in \mc{S}_F(R)$ for some subring $R$ of a finite-dimensional semisimple $F$-algebra $A$ and let $\mc{P}$ be some group-theoretic property. Then a set $\mc{G}$ of finite subgroups of $\U(R)$ is called \emph{$\mc{P}$-admissible} if all $H \in \mc{G}$ have property $\mc{P}$.
\end{definition}

\begin{example*}
If $\mc{P}$ is the property “the order $|H|$ divides $|\Gamma|$”, we will speak of \emph{“order-admissible”}. Similarly for $\exp(\Gamma)$ and $\spec(\Gamma)$ we say \emph{ $\exp$-}, respectively \emph{$\spec$-admissible}.
\end{example*}

In case that $A$ is semisimple but not simple, there is a natural weaker property to consider. Namely, for each $e \in \PCI(A)$ we can consider the simple factor $Ae$ and the associated projection $\pi_e\colon A \rightarrow Ae$. For a subgroup $G$ of $\U(A)$ we denote 
\[Ge := \pi_e(G).\]

\begin{definition}
    Let $\mc{G}$ be a set of finite subgroups of $\U(R)$, and $\Gamma \subseteq \mc{S}_F(R)$. We say that $R$ \emph{has the blockwise Zassenhaus (resp. subgroup isomorphism) property relative to $\Gamma$ and $\mc{G}$} if for all $e \in \PCI(A)$, $Re$ has the Zassenhaus (resp. subgroup isomorphism) property relative to $\Gamma e$ and $\mc{G}e$.
\end{definition}
Note that if $R$ satisfies the Zassenhaus property relative to $\Gamma$ and $\mc{G}$, then $R$ satisfies the blockwise Zassenhaus property relative to $\Gamma$ and $\mc{G}$.

\subsection{Somme illustrative examples}
In this section we consider the unit group of some exceptional algebras and point out several types of behaviour. As a consequence we will obtain the following application to the blockwise Zassenhaus property.

\begin{theorem}\label{Zass prop for some excep components}
Let $G$ be a finite and $e \in \PCI(\Q G)$ such that $\Q Ge$ is either a field, a quaternion algebra or isomorphic to $\Ma_2(\Q(\sqrt{-d}))$ with $d\neq 3$. If $H$ is a finite subgroup of $V(\Z G)$, then $(He)^{\alpha} \leq Ge$ for some $\alpha \in \Q Ge$.
\end{theorem}

Recall that if $\Ma_2(\Q(\sqrt{-d}))$ is the simple component of some integral group, then $d \in \{ 0,1,2,3\}$. Hence the above is a combination of \Cref{summary for M2 over Q},
\Cref{sqrt-2 ok as component}, \Cref{the case of quaternion components} and \Cref{the case of Qi}. We also need the upcoming fact that $|He|$ divides $|Ge|$, see \Cref{blockwise Cohn-Livingstone}. \medskip

\noindent \underline{The case of $\Ma_2(\Q)$}\smallskip

Consider $A = \Ma_2(\Q)$ and its maximal order $\O = \Ma_2(\Z)$. It is well-known that 
\begin{equation}\label{amalgam M2 over Q}
\GL_2(\Z) \cong D_8 \ast_{C_2 \times C_2} D_{12},
\end{equation}
with $C_2 \times C_2 = \langle a^{o(a)/2}, b \rangle$, for $a$ and $b$ such that $D_n = \langle a,b\mid a^{n/2} = b^2 = 1, bab = a^{-1} \rangle$. By Bass--Serre theory (see e.g. \cite[I.4, Theorem 8]{Serre}), under the aforementioned isomorphism, an arbitrary finite subgroup $H$ of $\GL_2(\Z)$ is conjugate to a subgroup of $D_8$ or of $D_{12}$. In particular, if $H$ is not a $2$-group and $|H|\neq 12$, then $H$ is conjugated to a subgroup of $D_6$.
    
Suppose then that $G$ is a finite subgroup of $\GL_2(\Z)$ whose $\Q$-span is $\Ma_2(\Q)$. Since a spanning subgroup must be non-abelian and must contain at least $4$ elements, it follows that $G$ is conjugate to either $D_6$, $D_8$ or $D_{12}$. In fact, these three dihedral groups have a faithful irreducible $\Q$-representation into $\Ma_2(\Q)$. Hence, in view of \Cref{spanning set non empty iff} below,
\[\mc{S}(\O) = \{ D_m \mid m = 6,8, 12 \}.\]
Let $H \leqslant \GL_2(\Z)$ be finite such that $|H|$ divides $m$ for some $m \in \{6,8,12\}$. Then the above shows that if $\exp(H)\neq 2$, then the cardinality $|H|$ uniquely determines to which spanning subgroup $H$ is conjugated. If $\exp(H)=2$, then $H$ is conjugated into any spanning subgroup since the amalgamation of $\GL_2(\Z)$ is over $\langle a^{o(a)/2}, b \rangle$. In conclusion, we obtain the following.

\begin{proposition}\label{summary for M2 over Q}
    Let $\Gamma \in \mc{S}(\Ma_2(\Z))$ and $H$ an order-admissible subgroup of $\GL_2(\Q)$. Then $H$ is conjugated over $\GL_2(\Z)$ to a subgroup of $\Gamma$. In particular, it has the Zassenhaus property relative to $\Gamma$ and all order-admissible subgroups.
\end{proposition}

\begin{remark*}
We would like to emphasise that in this example one has an especially nice Zassenhaus property since the conjugation is already over the order. However, this is specific to $\Ma_2(\Q)$, as the upcoming example already illustrates well, see \Cref{Comments in sqrt-2}. Furthermore, \Cref{summary for M2 over Q} was already implicit in \cite[Theorem 1.5]{J1} were it is proven (using other terminology) that if all the non-division components of $\Q G$ are isomorphic to $\Ma_2(\Q)$, then $\Z G$ enjoys the blockwise Zassenhaus property relative to $G$ and the set of all finite subgroups. Moreover, conjugation can be realised over the maximal order. Importantly, the latter is really a property of the projections. Indeed, it is classical that $\Z S_3$ satisfies the third Zassenhaus conjecture, however there is a finite subgroup of $V(\Z S_3)$ which is conjugated to a subgroup of $S_3$ over $\Q S_3$, but not over $\Z S_3$.
\end{remark*}\medskip

\noindent \underline{The case of $\Ma_2(\Q(\sqrt{-2}))$} \smallskip

In \cite{Hat}, Hatcher gave a decomposition as an amalgamated product for $\PGL_2(\mc{I}_2)$, where $\mc{I}_2$ is the ring of integers of $\Q(\sqrt{-2})$. This decomposition can be adapted to yield the following for $\GL_2(\mc{I}_2)$, see \Cref{amalgam GL2 for sqrt-2} for details:
\begin{equation}\label{amalgam I2}
    \GL_2(\mc{I}_2) \cong (\GL_2(\F_3) \ast_{C_8} SD_{16}) \ast_{\SL_2(\Z)} (D_{12} \ast_{C_2 \times C_2} Q_8).
\end{equation}
Here $SD_{16} := \langle a,b \mid a^8=b^2=1, a^b=a^3 \rangle$ is the semidihedral group.

As in the case of $\Ma_2(\Q)$, from \eqref{amalgam I2} one can readily read off which Zassenhaus or subgroup isomorphism properties it possesses. For instance, again by Bass--Serre theory, every finite subgroup is conjugated in $\GL_2(\mc{I}_2)$ to a subgroup of $\GL_2(\F_3)$, $SD_{16}$, $D_{12}$ or $Q_8$. Of these groups, only the first two have an irreducible faithful representation onto $\Ma_2(\Q(\sqrt{-2}))$. Hence \Cref{spanning set non empty iff} yields that 
\[\mc{S}(\Ma_2(\mc{I}_2))= \{SD_{16}, \GL_2(\mathbb{F}_3)\}.\]
Interestingly, in this example there are two $\GL_2(\mc{I}_2)$-conjugacy classes of spanning subgroups isomorphic to $SD_{16}$, as any subgroup of $\GL_2(\F_3)$ isomorphic to $SD_{16}$ will also span $\Ma_2(\Q(\sqrt{-2}))$. However, by \Cref{upgrade iso to conjugation} they will be conjugated over $\GL_2(\Q(\sqrt{-2}))$.

\begin{proposition}\label{props for sqrt-2}
The ring $\Ma_2(\mc{I}_2)$ has the Zassenhaus property relative to
\begin{itemize}
    \item $\GL_2(\mathbb{F}_3)$ and all finite subgroups,
    \item $SD_{16}$ and all order-admissible subgroups.
\end{itemize}
\end{proposition}
\begin{remark}\label{Comments in sqrt-2}
From the amalgamated decomposition in \eqref{amalgam I2}, it follows that every finite subgroup of $\GL_2(\mc{I}_2)$ is conjugated inside $\GL_2(\mc{I}_2)$ to either a subgroup of $\GL_2(\F_3) \ast_{C_8} SD_{16}$ or $D_{12} \ast_{C_2 \times C_2} Q_8$. Moreover, both cases are mutually exclusive. Thus for example the subgroup $D_{12}$ of $D_{12} \ast_{C_2 \times C_2} Q_8$ \emph{cannot} be conjugated over $\GL_2(\mc{I}_2)$ to a subgroup of one of the two spanning subgroups $\GL_2(\F_3)$ or $SD_{16}$, but by \Cref{props for sqrt-2} it \emph{can} be conjugated over $\GL_2(\Q(\sqrt{-2}))$ to one. One of the subtleties that this example touches upon, is the problem whether a finite subgroup is contained in a spanning subgroup. Namely, $D_{12}$ is contained in no spanning subgroup contained in $\Ma_2(\mc{I}_2)$, but one contained in $\Ma_2(\Q(\sqrt{-2}))$. More precisely, it is contained in a spanning subgroup of $\Ma_2(\Q(\sqrt{-2}))$ of the form $G^y$ with $G \in \mc{S}(\Ma_2(\mc{I}_2))$ and $y \in \GL_2(\Q(\sqrt{-2}))$.

Interestingly, if the third Zassenhaus conjecture were to hold, then finite subgroups $H$ of  $V(\Z G)$ are always contained in a group basis $G^y$ with $y \in \Q G$. As shown in \Cref{folklore remark}, if $H$ is a cyclic subgroup, then $H$ being contained in a group basis $\Gamma^y$ with $y \in \Q G$ and $\Gamma \leq V(\Z G)$ with $|\Gamma|=|G|$, implies that $H$ is conjugated to a subgroup of the given basis $G$. As explained above, this does not hold for cyclic subgroups of $\GL_2(\mc{I}_2)$, illustrating an interesting difference between finite subgroups of integral group rings and finite subgroups of orders in a simple algebra. 
\end{remark}

Note that \Cref{props for sqrt-2} directly implies the following.

\begin{corollary}\label{sqrt-2 ok as component}
Let $G$ be a finite group such that $\Ma_2(\Q(\sqrt{-2})) \cong \Q Ge$ for some $e \in \PCI(\Q G)$. If $H$ is a finite subgroup of $V(\Z G)$, then $He$ is $\Q Ge$-conjugated to a subgroup of $Ge$.
\end{corollary}

Indeed, $Ge \in \mc{S}(\Ma_2(\mc{I}_2))$ and hence $Ge$ is isomorphic to $SD_{16}$ or $\GL_2(\F_3)$. Since $H$ is a finite subgroup of $V(\Z G)$, one has that $He$ is order-admissible for $Ge$ (see \Cref{blockwise Cohn-Livingstone} for details), showing that \Cref{sqrt-2 ok as component} follows from \Cref{props for sqrt-2}.

\begin{proof}[Proof of \Cref{props for sqrt-2}]
It is a fact that the maximal subgroups of $\GL_2(\F_3)$ are $\SL_2(\F_3)$, $ SD_{16}$ and $D_{12}$. Moreover, $Q_8$ is a maximal subgroup of $\SL_2(\F_3)$. Hence it follows by \eqref{amalgam I2} that any finite subgroup of $\GL_2(\mc{I}_2)$ is isomorphic to a subgroup of $\GL_2(\F_3)$. Next, it can be verified that all  subgroups of $\GL_2(\F_3)$ of order a power of $2$ are contained in its maximal subgroup $SD_{16}$. All this yields the statement, with “Zassenhaus property” replaced by “subgroup isomorphism property”.

To obtain the Zassenhaus property, it is enough to prove that the copies of $SD_{16}$ and $D_{12}$ in \eqref{amalgam I2} can be conjugated over $\GL_2(\Q(\sqrt{-2}))$ to a subgroup of $\GL_2(\F_3)$, and the same for $Q_8$ into $SD_{16}$. For this we will need to invoke the Skolem--Noether theorem. 

Let $H$ be one of the groups $SD_{16}$, $D_{12}$ or $Q_8$ and consider $\Span_{\Q(\sqrt{-2})} \{H \}$. Investigating the faithful components of their rational group algebras, one obtains that $\Span_{\Q(\sqrt{-2})} \{H \} = \Ma_2(\Q(\sqrt{-2}))$. Now denote by $\psi$ an isomorphism between $H$ and a subgroup $K$ of $\GL_2(\F_3)$. Applying Skolem--Noether to the $\Q(\sqrt{-2})$-linear extension of $\psi$ and the $\Q(\sqrt{-2})$-algebra map $\Span_{\Q(\sqrt{-2})} \{H \} \hookrightarrow  \Ma_2(\Q(\sqrt{-2}))$, yields the desired conjugation property.
\end{proof}

\noindent \underline{The case of fields and quaternion algebras}\smallskip

The advantage of the blockwise Zassenhaus property is that it can be approached one type of component at a time. For instance, we now show that field and quaternion algebra components of $\Q G$ will never provide a problem for the blockwise Zassenhaus property.

\begin{proposition}\label{the case of quaternion components}
Let $G$ be a finite group and $e \in \PCI(\Q G)$ such that $\Q Ge$ is some quaternion algebra or a field.  Then for any $H \leq V(\Z G)$ the group $He$ is conjugated over $\Q Ge$ to a subgroup of $Ge$.
\end{proposition}
An interesting feature of the proof of \Cref{the case of quaternion components} is that it can be reduced to understanding the Zassenhaus conjectures for a small class of groups. Namely, as will be explained in more detail in \Cref{section: gen props for grp ring}, the surjection $\Q G \rightarrow \Q Ge$ factors through to the $\Q$-linear extension $\Phi$ of $G \rightarrow Ge$. Using this one can prove the following:
\begin{enumerate}[(i)]
    \item If $Ge$ satisfies a Zassenhaus conjecture including groups isomorphic to $He$, then $He$ will be conjugated to a subgroup of $Ge$ in $\Q Ge$.
    \item $|He|$ divides $|Ge|$ for every $e \in \PCI(\Q G)$ (see \Cref{blockwise Cohn-Livingstone} below).
\end{enumerate}

\begin{proof}[Proof of \Cref{the case of quaternion components}]
First suppose that $\Q Ge$ is a field $F$. Then the ring of integers $R$ of $F$ is its (up to conjugation) unique maximal order. The unit group of $R$ is a finitely generated abelian group described by Dirichlet's unit theorem. In particular $He$ and $Ge$ are subgroups of the torsion subgroup of $\U(R)$, which is cyclic. Hence the dividing orders yield that $He \leqslant Ge$, as desired.

Next suppose that $\Q Ge$ is a quaternion algebra.  Then $Ge$ is one of the groups mentioned in \Cref{spanning of quaternion}. Now recall that the third Zassenhaus conjecture was proven for $Q_{4m}$ in \cite{CMdR}, for $\SL_2(\F_5)$ in \cite[Theorem 4.3]{DJPm} and in \cite[Theorem 4.7]{DJ} for $\SU_2(\F_3)$. In these cases, using \eqref{props Phi} and \eqref{middle quotient}, the desired conjugation is inherited from $V(\Z[Ge])$.

It remains to consider the case that $Ge \cong \SL_2(\F_3)$. In this case the first Zassenhaus conjecture was obtained in \cite{HK}, and $\Span_{\Q} \{ Ge \}= \qa{-1}{-1}{\Q}$. More precisely, $Ge$ is isomorphic to the unit group of the Hurwitz quaternions, which is up to conjugation the unique maximal order in $\qa{-1}{-1}{\Q}$ and thus $He$ is conjugated to a subgroup of $Ge$, finishing the proof.
\end{proof}

\noindent \underline{The case of $\Ma_2(\Q(\sqrt{-1}))$}\smallskip

This example will require combining the methods from the previous examples. For $\GL_2(\Z[\sqrt{-1}])$, it is not possible to follow the exact same strategy as for $\GL_2(\Z)$ and $\GL_2(\mc{I}_2)$, since $\GL_2(\Z[\sqrt{-1}])$ has Serre's property (FA) \cite[Theorem 5.1]{BJJKT} and thus cannot be decomposed through amalgamation. Nevertheless, in \cite[Theorem 4.4.1]{FineBook} an amalgamated decomposition was constructed for $\PSL_2(\Z[\sqrt{-1}])$ from which one readily deduces the following decomposition for $\SL_2(\Z[\sqrt{-1}])$:
\begin{equation}\label{decomp SL}
    \SL_2(\Z[\sqrt{-1}]) \cong  (Dic_3 \ast_{C_6} \SL_2(\F_3)) \ast_{\SL_2(\Z)} (Dic_3 \ast_{C_4} Q_8).
\end{equation}

Now, in general if $R$ denotes the maximal order of the centre of a finite-dimensional division algebra $D$ and $\O$ an order in $D$, then $\SL_n(\O) = \ker(\operatorname{Rnr} \colon \Ma_n(D) \rightarrow R^*)$. Thus if $\Gamma$ is a subgroup of $\SL_n(\O)$, then it fits in an exact sequence
\[1\rightarrow \SL_n(\O) \cap \Gamma \rightarrow\Gamma \rightarrow \operatorname{Rnr}(\Gamma) \rightarrow 1.\]
Moreover, if two subgroups $\Gamma_1$ and $\Gamma_2$ are conjugated over $\Ma_n(D)$, then so are the subgroups $\SL_n(\O) \cap \Gamma_i$ and $\operatorname{Rnr}(\Gamma_1)= \operatorname{Rnr}(\Gamma_2)$. In the case that $D$ is a subfield of $\C$ and $n=2$, then conjugacy classes of finite subgroups in $\SL_2(D)$ are a well-understood problem. 

For the purpose of the subgroup isomorphism (resp. Zassenhaus) property, one wants to understand the maximal subgroups of $\GL_2(\Z[\sqrt{-1}])$. The above tells us that we need to understand the extensions of $H \in \{ Q_8,Dic_3, \SL_2(\F_3)\}$ by $C_2$ or $C_4$. Fortunately, in those cases, $\Out(H)$ is small. Namely, $\Out(Q_8) \cong S_3$, $\Out(Dic_3) \cong D_{12}$ and $\Out(\SL_2(\F_3)) \cong C_2$. One needs to be cautious since not all actions yield subgroups of $\GL_2(\Z[\sqrt{-1}])$. As $Q_8$ is characteristic in $\SL_2(\F_3)$, all extensions built on $Q_8$ will be contained in an extension built on $\SL_2(\F_3)$. 

Working all this out would yield two maximal subgroups: $\SL_2(\F_3) \rtimes C_4 $ and $Dic_3 \rtimes C_2 = C_4 \times S_3$, which respectively have \textsc{SmallGroupID}'s \texttt{[96,67]} and \texttt{[24,5]}. Both groups have an irreducible $\Q$-representation onto $\Ma_2(\Q(\sqrt{-1}))$ and hence they are spanning subgroups of $\GL_2(\Z[\sqrt{-1}])$ (see \Cref{spanning set non empty iff}). In summary, every finite subgroup of $\GL_2(\Z[\sqrt{-1}])$ is isomorphic to a subgroup of $\SL_2(\F_3) \rtimes C_4$ or $Dic_3 \rtimes C_2$. This conclusion will also follow from the more general results of \Cref{sectioin sbgrps excetpional}, whose proofs are given in detail. 

\begin{proposition}\label{the case of Qi}
Let $G$ be a finite group such that $\Ma_2(\Q(\sqrt{-1})) = \Q Ge$ for some $e \in \PCI(\Q G)$. The following hold:
\begin{enumerate}
    \item\label{Zassenhaus prop Qi} Every finite subgroup of $\GL_2(\Z[\sqrt{-1}])$ is conjugated in $\Ma_2(\Q(\sqrt{-1}))$ to a subgroup of some $\Gamma \in \mc{S}(\Ma_2(\Z[\sqrt{-1}]))$.
    \item\label{Ok for Qi component} If $H$ is a finite subgroup of $V(\Z G)$, then $He$ is $\Q Ge$-conjugated to a subgroup of $Ge$.
\end{enumerate}
\end{proposition}

\begin{remark*}
One could also have obtained a classification of the (maximal) subgroups of $\GL_2(\Z[\sqrt{-1}]))$ via a reduction to $\PGL_2(\Z[\sqrt{-1}])$. More precisely, every finite subgroup $\Gamma$ is a central extension
\[1 \rightarrow \ZZ(\Gamma) \rightarrow \Gamma \rightarrow Q \rightarrow 1,\]
with $Q$ a subgroup of $\PGL_2(\Z[\sqrt{-1}])$. Subgroups of the latter are well-known (e.g. see \cite{Beau}): $C_n$ with $n \mid 4,6$, $D_6$, $D_8$, $A_4$ and $S_4$. A computation of their second group cohomology group with coefficients in $C_2$ or $C_4$ would yield the same conclusion.
\end{remark*}

\begin{proof}[Proof of \Cref{the case of Qi}]
If one replaces conjugation by isomorphism in \eqref{Zassenhaus prop Qi} a proof has been sketched before the statement of \Cref{the case of Qi}. The fact that conjugation holds, follows from an argument as in the proof of \Cref{props for sqrt-2}. Alternatively, see the upcoming \Cref{iso implies conj for exc} for a general statement.

Now consider a finite subgroup $H \leqslant V(\Z G)$. By \eqref{Zassenhaus prop Qi} we know that $He$ is $\Ma_2(\Q(\sqrt{-1}))$-conjugated to some $\Gamma \in \mc{S}(\Ma_2(\Z[\sqrt{-1}]))$. There are two spanning subgroups which are maximal (for the inclusion) among the members of $\mc{S}(\Ma_2(\Z[\sqrt{-1}]))$: the groups $\SL_2(\F_3) \rtimes C_4$ and $C_4 \times S_3$.\smallskip

\noindent \underline{Claim 1:} The group $C_4 \times S_3$ satisfies the Higman subgroup property, i.e. if $H$ is a finite subgroup of $V(\Z[C_4 \times S_3])$, then $C_4 \times S_3$ contains a copy of $H$.\smallskip

By the linear independence of finite subgroups of $V(\Z G)$, it follows that $|H|$ divides $|C_4 \times S_3| = 24$. If $|H|=24,$ then $H \cong C_4 \times S_3$ as $C_4 \times S_3$ is metabelian and hence satisfies the isomorphism problem by Whitcomb's theorem. Next let $|H| \mid 8$ or $12$. Moreover, since $C_4 \times S_3$ is solvable, it follows from \cite{HertOrder} that $\spec(H) \subseteq \spec(C_4 \times S_3) = \{n \, : \, n \mid 12 \}$. This in particular solves the case that $H$ is cyclic. Next recall that the subgroup isomorphism problem is known for subgroups of the form $C_2 \times C_2$ (\cite{HertOrder2}), $C_2 \times C_4$ (\cite{MarSIP}) and $C_p \times C_q^{\ell}$ for $p \neq q$ primes (\cite{HertOrder}). Hence for $H$ abelian, the only remaining case to investigate is $H \cong C_2^3$.  However, this is not possible since $H$ embeds into $\Z Ge$, where the maximal rank of an elementary abelian $2$-subgroup is $2$.

It remains to consider the case $H$ non-abelian. Since $C_4 \times S_3$ contains $D_{12}$ and $Dic_3$, we need to show that $H$ is not isomorphic to $D_8$, $Q_8$ or $A_4$.  First suppose that $H \cong Q_8$ or $D_8$. Denote $S_3 = \langle a,b \mid a^3 = b^2 = 1, a^b = a^{-1}\rangle$ and consider the epimorphism \[\Z[C_4 \times S_3] \rightarrow \Z[(C_4 \times S_3)/\langle a \rangle] \cong \Z[C_4 \times C_2],\] whose kernel is the relative augmentation\footnote{I.e. the kernel of the canonical map $R [\Gamma] \to R[\Gamma/\Lambda]$ for any group $\Gamma$ with normal subgroup $\Lambda$.} ideal $\omega(C_4 \times S_3, \langle a \rangle)$. This induces an epimorphism $\pi$ from $V(\Z[C_4 \times S_3])$ to $V(\Z[C_4 \times C_2])$ under which finite $2$-subgroups are mapped isomorphically. In particular $\pi(H)$ would be a non-abelian subgroup of the abelian group $V(\Z[C_4 \times C_2])$, a contradiction. Finally suppose that $H \ncong A_4$. As finite subgroups of the normalised unit group are $\Q$-linearly independent, one has that $\Q[H]$ is a subring of
\[\Q[C_4 \times S_3] \cong  \Q^4 \oplus \Q(i)^2 \oplus  \Ma_2(\Q) ^2 \oplus \Ma_2(\Q(i)).\] 
Now $\Q[H] = \oplus_{e \in \PCI(C_4 \times S_3)} \Q[H]e$. If $H \cong A_4$, then $\Q[H]\cong \Q \oplus \Q(\sqrt{-3}) \oplus \Ma_3(\Q)$, which one readily proves is not compatible with the above decomposition. This finishes the proof of Claim 1.\smallskip

Now we consider the other $\Gamma \in \mc{S}(\Ma_2(\Z[\sqrt{-1}]))$, which are all subgroups of $\operatorname{U}_2(\F_3):= \SL_2(\F_3) \rtimes C_4 =$ \texttt{[96,67]}.\smallskip

\noindent \underline{Claim 2:} If $\Gamma \in \mc{S}(\Ma_2(\Z[\sqrt{-1}]))$ and $|\Gamma| \neq 48$ or $96$, then $\Gamma$ satisfies the Higman subgroup property.\smallskip

If $\Gamma$ is nilpotent, then it even satisfies the $3$\textsuperscript{rd} Zassenhaus conjecture \cite{WeissPGroup, WeissNilpotent}. The only non-nilpotent $\Gamma \in \mc{S}(\Ma_2(\Z[\sqrt{-1}]))$ with $|\Gamma| \neq 48$ or $96$, is the metacyclic group $C_3 \rtimes C_8$ (as can be seen from  \Cref{tableB} in \Cref{appendix section}). Now note that all strict subgroups of $C_3 \rtimes C_8$ are cyclic. As $C_3 \rtimes C_8$ is metacyclic it satisfies the $1$\textsuperscript{st} Zassenhaus conjecture \cite{HertCC, CMdR}. Moreover by Whitcomb's theorem it also satisfies the isomorphism problem, thus altogether indeed has the subgroup isomorphism property for all its subgroups, as claimed.\smallskip

At this stage we have proven that if $|Ge|\neq 48$ or $96$, then statement  \eqref{Ok for Qi component} holds. If $|Ge| =48$, then $Ge \cong \SL_2(\F_3) \rtimes C_2 = $ \texttt{[48,33]} and if $|Ge| =96$, then $Ge \cong \mathrm{U}_2(\F_3)$.  Note that $ \SL_2(\F_3) \rtimes C_2$ is a maximal subgroup of $\mathrm{U}_2(\F_3) $ which contains copies of all subgroups which are not a spanning subgroup. Moreover, $He$ cannot be another spanning subgroup, as for those groups $\spec(He)$ is not contained in the spectrum of the solvable group $\SL_2(\F_3) \rtimes C_2$. Hence if $H$ is $\Q G e$-conjugated to a subgroup of $\mathrm{U}_2(\F_3)$ and $H$ is order-admissible with $\SL_2(\F_3) \rtimes C_2$, then it is actually conjugated to a subgroup of the latter. This reduces the problem to the largest spanning subgroup $\mathrm{U}_2(\F_3)$. 

Now the only finite subgroups of $\GL_2(\Z[\sqrt{-1}])$ which are not isomorphic to a subgroup of $\mathrm{U}_2(\F_3)$ are $D_{12}$ and $Dic_3$. Thus by part \eqref{Zassenhaus prop Qi}, it remains to show that $He \ncong D_{12}$ or $Dic_3$. Now if $He \cong D_{12}$, then this would imply that $V(\Z[Ge])$, with $Ge \cong \mathrm{U}_2(\F_3)$, contains a subgroup isomorphic to $D_6$. However, the latter is a Frobenius group and thus by \cite[Theorem D]{PdRV}, $Ge$ would also contain a copy of $D_6$, a contradiction. The dicyclic case is treated similarly.
\end{proof}

\subsection{Some general properties for simple algebras}\label{section: gen props for grp ring}
\subsubsection{The setting of components of group rings}
In the setting of the blockwise Zassenhaus property, one is interested in the images of finite subgroups under 
\[\pi_e \colon A \rightarrow Ae, \quad  x \mapsto xe,\]
for $e \in \PCI(A)$. In case that $A=FG$ is a group algebra, the map $\pi_e$ factorises through a meaningful (non-simple) quotient, as mentioned after \Cref{the case of quaternion components}. Namely, extend $R$-linearly the group epimorphism $\varphi_e \colon G \rightarrow Ge$ to an $R$-algebra homomorphism $\Phi_e\colon R [G] \rightarrow R [Ge]$.  Note then that 
\begin{equation}\label{props Phi}
\ker(\Phi_e) = \omega(G,\ker(\varphi_e)) \text{ and } \Phi_e(V(R[G])) \subseteq V(R[Ge]),
\end{equation}
where in general $\omega(\Gamma,\Lambda)$ denotes the relative augmentation ideal. Since $\pi_e(n-1) =0$ for all $n \in \ker(\varphi_e) = \{ g \in G \mid ge = e \}$, we have that $\omega(G,\ker(\varphi_e)) \subseteq \ker(\pi_e)$. Therefore there exists a unique morphism 
\begin{equation}\label{middle quotient}
\sigma_e \colon R[Ge] \rightarrow R[G]e\quad  \text{ such that } \quad \pi_e = \sigma_e \circ \Phi_e.
\end{equation}

\begin{lemma}\label{blockwise Cohn-Livingstone}
Let $G$ be a finite group and $R$ a $G$-adapted integral domain with field of fractions $F$. For any finite subgroup $H$ in $V(R G)$ and $e \in \PCI(F G)$, the following hold:
\begin{enumerate}[(i)]
    \item\label{order in projection} $\abs{He}$ divides $\abs{Ge}$,
    \item \label{exp projection} $\exp({H}{e})$ divides $\exp(Ge)$,
\item \label{item: spectrum in quotient} If $G$ satisfies \eqref{spec prob} for $R$, then $\spec(He) \subseteq \spec(Ge)$.
\end{enumerate}
\end{lemma}
\begin{proof}
By \eqref{props Phi} we have that $\Phi_e(H)$ is a finite subgroup of $ V(R[Ge])$, and therefore $|\Phi_e(H)| \mid |Ge|$, see for example \cite[Corollary 2.7]{AngelSurvey}. 
To prove \eqref{order in projection}, we show that $|He| \mid |\Phi_e(H)|$. For this, note that \eqref{middle quotient} implies that $He = \pi_e(H) = \sigma_e( \Phi_e(H))$ is an epimorphic image of $\Phi_e(H)$ and hence $|He|$ divides $|\Phi_e(H)|$, as desired. Since $V(R [Ge])$ has exponent $\exp(Ge)$ by \eqref{exp cohn livingstone}, it follows moreover that $\exp(\Phi_e(H)) \mid \exp(Ge)$ (since $\Phi_e(H) \leqslant V(R[Ge])$ and the fact that the exponent of a subgroup divides the exponent of the overlying group). But since $He$ is an epimorphic image of $\Phi_e(H)$, it follows that $\exp(He) \mid \exp(\Phi_e(H))$, which finishes \eqref{exp projection}.

Finally, suppose then that $G$ satisfies \eqref{spec prob}, i.e. $\spec(G) = \spec(V(RG))$. Then from $H \leqslant V(RG)$ it follows that $\spec(H) \subseteq \spec(V(RG))$, which implies that 
\begin{align*}
    \spec(He)  =  \spec(\sigma_e(\Phi_e(H))) \subseteq \spec(\Phi_e(H)) & \subseteq \spec(\Phi_e(V(RG))) \\
    & = \spec(\Phi_e(G)) = \spec(Ge).
\end{align*}
Here the first inclusion follows since for any group homomorphism $\psi \colon \Gamma \to \Lambda$ (of finite groups), one has that $\spec(\psi(\Gamma)) \subseteq \spec(\Gamma)$.
\end{proof}

Part \eqref{item: spectrum in quotient} of \Cref{blockwise Cohn-Livingstone} shows that the spectrum problem implies the following weaker version of \eqref{spec prob}: 

\begin{conjecture}[Blockwise Spectrum Problem]\label{block spectrum problem}
Let $G$ be a finite group. Then $\spec(Ge) = \spec(V(\Z G)e)$ for every  primitive central idempotent $e \in \PCI(\Q G)$.
\end{conjecture}

We would also like to emphasise that Hertweck already formulated in his Habilitations-schrift the following.

\begin{conjecture}[Blockwise isomorphism problem]\label{block ISO}
Let $G$ be a finite group and $H \leq V(\Z G)$ such that $\Z G = \Z H$. Then $Ge \cong He$ for every $e \in \PCI( \Q G)$.
\end{conjecture}

Finally, in the setting of the blockwise properties there is a class of finite groups that deserve special attention, namely those which have a faithful irreducible representation. In particular, the following question is interesting.

\begin{question}\label{iso for faith irr}
Let $G$ be a finite group having a faithful irreducible representation. Does it have the subgroup isomorphism property for any finite subgroup? 
\end{question}

\subsubsection{The setting of simple algebras}

Given an order $\O$ in a simple algebra $\Ma_n(D)$, the following result implies that $\mc{S}_F(\mc{O})$ is non-empty if and only if $\Ma_n(D)$ is the simple component of some group algebra.

\begin{lemma}\label{spanning set non empty iff}
Let $G$ be a finite subgroup of some $\GL_n(D)$ with $D$ a finite-dimensional division $F$-algebra with $F$ a field with $\Char(F)=0$. Then $\Span_{F}\{ g \in G\} = \Ma_n(D)$ if and only if $\Ma_n(D) \in \mc{C}(F G)$.
\end{lemma}
\begin{proof}
Consider the canonical representation $\varphi: FG \rightarrow \Ma_n(D)$ associated to the embedding of $G$ in $\GL_n(D)$. If one can show that $\varphi$ is irreducible, then $\Ma_n(D) \in \mc{C}(F G)$ by the correspondence between irreducible $F$-representations of $G$ and the simple components of $FG$. But the irreducibility follows from the spanning condition. Indeed, write $\Ma_n(D) = \End_{FG}(V)$. Since $\Span_{F}\{ g \in G\} = \Ma_n(D)$, we have that $\Span_F\{ g\cdot v \mid g \in G\}= V$ for any non-zero $v \in V$. This shows that $V$ cannot have a non-zero proper $G$-invariant subspace, as claimed.

Conversely, suppose that $\Ma_n(D) \in \mc{C}(F G)$. In other words, there exists an irreducible $F$-representation $\rho \colon FG \rightarrow \Ma_n(D)$. It is a classical theorem by Burnside that \[\Span_F \{ \rho(g) \mid g \in G \} = \Ma_n(D).\qedhere\]
\end{proof}

In practice it is easier to simply verify that a given group $H$ is isomorphic to a subgroup of a group $\Gamma$, than to verify conjugation in some ambient group. However, in certain cases the Skolem--Noether theorem implies the existence of an isomorphism by conjugation:

\begin{lemma}\label{upgrade iso to conjugation}
Let $A$ be a simple $F$-algebra and $H$ and $G$ finite subgroups of $\U(A)$. Suppose that $\Span_F \{G\} =A$ and that $\Span_F \{H\}$ is a simple subalgebra of $A$ containing $\ZZ(A)$. If $H$ is isomorphic to a subgroup of $G$, then there exists an $\alpha \in \U(A)$ such that $H^{\alpha} \leqslant G$. Consequently, if $FG$ satisfies the blockwise subgroup isomorphism property relative to group bases, then it satisfies the blockwise Zassenhaus property relative to group bases.
\end{lemma}
\begin{proof}
Denote $B = \Span_F \{ H\}$ and $\varphi \colon H \rightarrow L$ the isomorphism with a subgroup $L$ of $G$. We have two morphisms from $B$ to $A$. Namely a first from the embedding of $B$ in $A$ and another through the linear extension of $\varphi$. The Skolem--Noether theorem then precisely asserts that both images are conjugate.  The consequence follows from the first statement. Indeed, for a group basis $H$ of $FG$ one has that $\Span_F\{ He \}=FGe$ and in particular $\Span_F\{He\}$ contains $\ZZ(FGe)$.
\end{proof}

Note that if $\ZZ(A) = \Q$, then \Cref{upgrade iso to conjugation} in particular indeed reduces the problem of showing that $H$ is conjugate to a subgroup of $G$ to showing that $H$ is isomorphic to a subgroup of $G$.

\subsubsection{The setting of simple algebras of reduced degree 2}\label{subsectie props of red deg 2}

Let $D$ be an arbitrary finite-dimensional division $\Q$-algebra. It will be useful to distinguish primitive and imprimitive groups. Recall the following definition (see \cite[section 2.1]{Banieqbal}).

\begin{definition}
    Let $V$ be the $2$-dimensional $D$-module given by column vectors in $D$. Let $G\leqslant \GL_2(D)$. Then $G$ has a natural action on $V$ by left multiplication. The group $G$ is said to be \emph{imprimitive} if and only if $V$ decomposes as a direct sum \[V = V_1 \oplus V_2 \]
    of $1$-dimensional $D$-modules such that for any $g \in G$, one has $gV_i = V_{\sigma(i)}$, for $i \in \{1,2\}$ and $\sigma \in \Sym(\{1,2\})$. A group which is not imprimitive is said to be \emph{primitive}.
\end{definition}
Note that for a primitive group $G \leqslant \GL_2(D)$, the action of $G$ on $V$ is irreducible. 

Now suppose that $G \leqslant \GL_2(D)$ is primitive. Then several cases can arise for the isomorphism type of $\Span_{\Q} \{G\} := \Span_{\Q}\{ g\in G\}$. Since there is a $\Q$-linear map $\Q G \to \Span_{\Q}\{G\}$ induced by the identity map, $\Span_{\Q}\{G\}$ in particular has the structure of a semisimple $\Q$-algebra. Since $G \leqslant \GL_2(D)$, the possible isomorphism types of $\Span_{\Q}\{G\}$ then are:
\begin{enumerate}[(i)]
    \item A division algebra $D_G$;
    \item $\Ma_2(D_G)$ for some division algebra $D_G$.
\end{enumerate}
Note that the possibility that $\Span_{\Q} \{G\}$ is the direct sum of two simple algebras does not occur, since this would contradict the assumption that $G$ is primitive and in particular irreducible.\medskip

If $\Ma_2(D)$ is exceptional, then the isomorphism types of $\Span_{\Q} \{G\}$ can be described explicitly.
\begin{lemma}\label{subalgebra of exceptional}
Let $\Ma_2(D)$ be an exceptional algebra and $G \leqslant \GL_2(D)$ a finite subgroup which acts irreducibly on $V$. Then $\Span_{\Q} \{G\}$ is either a field, a quaternion algebra, or an exceptional matrix algebra. 
\end{lemma}

\begin{remark}
\label{remark degree quat alg}
The proof of \Cref{subalgebra of exceptional} will be more precise on the possibilities for $\Span_{\Q} \{G\}$. For instance from the proof it follows that for any finite group $G \leqslant \GL_2(D)$ which spans a quaternion algebra $\qa{-a}{-b}{F}$, one has that $[F:\Q] \leq 2$. Conversely, $[F:\Q]=2$ can only occur if $D$ is a quaternion algebra over $\Q$. 
\end{remark}

\begin{proof}[Proof of \Cref{subalgebra of exceptional}]
As explained above, $B := \Span_{\Q} \{G\}$ is either a division algebra $D_G$ or of the form $\Ma_2(D_G)$. Suppose that $B \lneq \Ma_2(D)$ and that $B$ is non-commutative.  

First we consider the case that $\ZZ(D)= \Q$. Then $B$ is a simple $\ZZ(D)$-subalgebra of $\Ma_2(D)$. Therefore, the double centraliser theorem \cite[Theorem 2.10.1]{EricAngel1} implies that $\dim_{\Q} (B)$ divides $\dim_{\Q} \Ma_2(D)$. The latter in turn is a divisor of $16$ since $\Ma_2(D)$ is exceptional. Now recall that  $\dim_{\ZZ(B)} (B)$ is a square and the reduced degree of $B$ is less than or equal to $2$ by \cite[pg 455, Lemma]{Banieqbal}. Together with the assumptions on $B$, and the fact that
\[\dim_{\Q} (\ZZ(B)) \cdot \dim_{\ZZ(B)}(B) = \dim_{\Q} (B) \mid 16,\]
this then implies that $\dim_{\ZZ(B)}B = 4$ and $\dim_{\Q} \ZZ(B)= 1$ or $2$. If $\dim_{\Q} \ZZ(B)= 1$, then $B$ is either $\Ma_2(\Q)$ (if $D \neq \Q$) or a quaternion algebra with centre $\Q$. If $\dim_{\Q} \ZZ(B)= 2$, then $B$ is either a quaternion algebra with centre $F$ or $B = \Ma_2(F)$ with $F$ a quadratic extension of $\Q$. This case can only occur if $\dim_{\Q} \Ma_2(D)=16$ and hence $D= \qa{-a}{-b}{\Q}$ for some $a,b \in \N$.

Now no quadratic real extension $K$ of $\Q$ can embed in $\qa{-a}{-b}{\Q}$ as it is totally definite. Indeed, suppose that we have a $\Q$-algebra embedding of $K$ into $\qa{-a}{-b}{\Q}$. Then we would also have one of $K\ot_{\Q}\R \cong \R \times \R$ into $\qa{-a}{-b}{\Q}\ot_{\Q}\R \cong \qa{-a}{-b}{\R}$. However $\qa{-a}{-b}{\R}$, being a division algebra, cannot contain $\R \times \R$, as the latter contains non-trivial idempotents. In conclusion, $F$ is an imaginary quadratic extension of $\Q$ and $B$ is an exceptional matrix algebra. 

Next we consider the case that $\ZZ(D) \neq \Q$ and thus $D = \Q(\sqrt{-d})$ for some square-free natural number $d \in \N_{>0}$. In this case the double centraliser theorem does not apply, but nevertheless we still have that 
\begin{equation}\label{bound for larg center}
\dim_{\Q} (\ZZ(B)) \cdot \dim_{\ZZ(B)} (B) = \dim_{\Q} (B) \lneq 8, 
\end{equation}
since $B \lneq \Ma_2(\Q(\sqrt{-d}))$. Since $\dim_{\ZZ(B)} B$ is square and $B$ is non-commutative, \eqref{bound for larg center} implies that $\dim_{\ZZ(B)} B= 4$ and $\ZZ(B) = \Q$. Therefore either $B = \Ma_2(\Q)$ or $B$ is a quaternion algebra with centre $\Q$.
\end{proof}

In case of exceptional matrix algebras, the subgroup isomorphism property implies the stronger Zassenhaus property. 

\begin{lemma}\label{iso implies conj for exc}
Let $A$ be an exceptional matrix algebra, and $G$, $H$ finite subgroups of $\U(\O)$. Suppose that $\Span_{\Q} \{G \}=A$ and that $H$ is isomorphic to a subgroup of $G$. Then $H^{\alpha} \leq G$ for some $\alpha \in \U(A)$. 
\end{lemma}
We prove \Cref{iso implies conj for exc} in section \ref{sect imprimitive subgr} below. The proof consists of applying \Cref{upgrade iso to conjugation}, and additionally requires a better understanding of imprimitive groups, which we cover in section \ref{sect imprimitive subgr}.

\subsection{Finite subgroups of exceptional simple algebras}\label{sectioin sbgrps excetpional}

Amitsur \cite{Amitsur} classified all finite subgroups of a finite-dimensional division algebra. Moreover, for each of these groups, he described an explicit division algebra containing it. From this, one can in particular deduce the possible finite spanning subgroups of a given division algebra. In a highly technical work, Banieqbal \cite{Banieqbal} described the primitive spanning subgroups of $\Ma_2(D)$. However, in case of an exceptional (simple) matrix algebra, a direct and short proof was provided in \cite{EKVG} for the finite groups spanning such simple algebras. In this section we also describe the non-spanning finite subgroups. 

From now on assume that $\Ma_2(D)$ is an exceptional matrix algebra and $\O$ an order in $D$. Recall from \Cref{subsectie props of red deg 2} that a finite subgroup $H \leqslant \GL_2(D)$ satisfies at least one of the following three criteria:
\begin{enumerate}[(i)]
    \item $H$ is imprimitive,
    \item $\Span_{\Q} \{ H \}$ is either a field or a quaternion algebra,
    \item $\Span_{\Q} \{ H \}$ is an exceptional matrix algebra.
\end{enumerate}

We now describe the possible finite subgroups of $\GL_2(\O)$ for each of the three cases above. As a consequence we will obtain the following application.

\begin{theorem}\label{all in a spanning}
Let $A$ be an exceptional matrix algebra and $\O$ the maximal order in $A$. Suppose that $\mc{S}(\mc{O}) \neq \emptyset$ and $H$ a finite subgroup of $\U(\O)$ satisfying one of the following:
\begin{itemize}
    \item $H$ is primitive and not in \Cref{tab:exceptions spanning},
    \item $H$ is imprimitive and $|H| \mid |G|$ for some $G \in \mc{S}(\mc{O})$.
\end{itemize} 
Then $H$ is conjugated in $A$ to a subgroup of some $\Gamma \in \mc{S}(\mc{O})$. 
\end{theorem}

If $H$ is imprimitive, the condition that $|H| \mid |G|$ for some $G \in \mc{S}(\mc{O})$ is always satisfied except if $A \cong \Ma_2(\Q(\sqrt{-2}))$ or $\Ma_2(\qa{-2}{-5}{\Q})$, see \Cref{tab:imprimitive} for more details. The exceptions in the primitive case concern a small list of subgroups $H$ such that $\Span_{\Q} \{ H\} \cong \Ma_2(\Q(\sqrt{-d}))$ and $A \cong \Ma_2(\qa{-a}{-b}{\Q})$ for very specific values of $a,b,d \in \N$, as recorded in \Cref{tab:exceptions spanning}.

\subsubsection{The imprimitive subgroups}
\label{sect imprimitive subgr}
We start by describing the imprimitive groups. For the following statement, we follow the techniques from the proof of \cite[Lemma 2.2]{Banieqbal}.
\begin{lemma}\label{description imprimitive}
    Suppose $H \leqslant \GL_2(D)$ is imprimitive. Let $V = V_1 \oplus V_2$ be the associated decomposition into 1-dimensional $D$-modules. If there exists some $g \in H$ such that $gV_1= V_2$, then 
    \[H \cong (\Gamma \times \Gamma) \rtimes C_2,\]
    where $\Gamma \leqslant \U(D)$ and where $C_2$ acts by exchanging the factors. Otherwise, $H \cong \Gamma_1 \times \Gamma_2$ for some subgroups $\Gamma_1, \Gamma_2 \leqslant \U(D)$.
\end{lemma}
\begin{proof}
    Suppose $gV_i = V_i$ for all $g \in H$. Then up to conjugation by a base change matrix $A \in \GL_2(D)$, $H$ consists of diagonal matrices. It follows that $H \cong \Gamma_1 \times \Gamma_2$ for some finite subgroups $\Gamma_1, \Gamma_2 \leqslant \U(D)$. 

    Suppose then that there exists some $g \in H$ such that $gV_1 = V_2$. Let $K \leqslant H$ be the subgroup of $H$ such that $kV_i = V_i$ for all $k \in K$ and $i \in \{1,2\}$. Again up to conjugation of $H$ by a base change matrix $A \in \GL_2(D)$, $g$ is of the form  
    $\begin{psmallmatrix} 
    0 & g_1\\ g_2 & 0 
    \end{psmallmatrix},$ while $K$ consists precisely of the diagonal matrices of $H$. Moreover, $K$ has index $2$ in $H$. 
    Indeed, $H = K \sqcup gK$, since if $x := \begin{psmallmatrix}
        0 & x_1 \\ x_2 & 0
    \end{psmallmatrix}$ is another antidiagonal matrix in $H$, then 
    \[k := g^{-1} x = \begin{pmatrix}
        g_2^{-1} x_2 & 0 \\ 0 & g_1^{-1}x_1
    \end{pmatrix} \in K \]
    is an element such that $gk = x$. By the fact that $G$ is imprimitive, any $x \in G$ is either diagonal or antidiagonal and hence $H = K \sqcup gK$. In particular, $K$ is a normal subgroup of $H$, and the action of $g$ on $K$ by conjugation switches the diagonal entries. It follows that $G \cong (\Gamma \times \Gamma ) \rtimes C_2$, for some finite subgroup $\Gamma \leqslant \U(D)$.
\end{proof}
\begin{remark}
\label{rk maximal imprimitive in order}
    In particular, an imprimitive subgroup $H \leqslant \GL_2(D)$ which is contained in $\GL_2(\O)$, is isomorphically contained in a maximal imprimitive subgroup $(\U(\O) \times \U(\O) ) \rtimes C_2$, since $H$ acts on a $2$-dimensional $D$-space $V$ by left multiplication.
\end{remark}

Hence in particular the finite subgroups of $\GL_2(D)$ of imprimitive type are easy to describe in terms of finite subgroups of $\U(D)$. We will mainly be interested in finite subgroups of $\GL_2(\O)$. By \Cref{spanning set non empty iff} and \cite{EKVG}, if $\Ma_2(D)$ is exceptional, we have that 
\begin{equation}\label{when span set non-empty}
\mc{S}(\Ma_2(\O))\neq \emptyset \Leftrightarrow 
\left\lbrace \begin{array}{ll}
    D = \Q(\sqrt{-d})& \text{ with }  d = 0,1,2,3\\
    D = \qa{-a}{-b}{\Q} & \text{ with  (}a,b) = (-1,-1),(-1,-3),(-2,-5).
\end{array} \right.
\end{equation}
Moreover, the above division algebras are left norm euclidean, see \cite[Remarks 3.13 \& 6.14]{BJJKT}, and hence have up to conjugation a unique maximal order \cite[Proposition 2.9]{CeChLez}. The unique maximal order of $\Q$ is $\Z$, which has unit group $C_2 = \{\pm 1\}$.  For $\Q(\sqrt{-d})$ it is its ring of integers, denoted $\mc{I}_d$ whose unit group is as following, see \cite[Section 3]{BJJKT}: 
\begin{equation}\label{unit group ring of integers}
\U(\mc{I}_d) = \left\lbrace \begin{array}{ll}
   \langle -1 \rangle \cong C_2 & \text{if } d \neq 1,3 \\
   \langle i \rangle \cong C_4  & \text{if } d =1 \\
    \langle -\zeta_3\rangle \cong C_6 & \text{if } d =3 \\
\end{array}\right.
\end{equation}

Next, denote \[\mathbb{H}_2 = \qa{-1}{-1}{\mathbb{Q}},\quad \mathbb{H}_3 = \qa{-1}{-3}{\mathbb{Q}}\quad \mbox{ and }\quad \mathbb{H}_5 = \qa{-2}{-5}{\mathbb{Q}}.\]
Their (up to conjugation) unique maximal orders will respectively be denoted by $\O_2,\O_3$ and $\O_5$. Their unit groups are the following:
\begin{equation}  \label{eq:units_of_quat_orders}
 \U(\O_2) \cong \SL_2(\F_3) \cong Q_8 \rtimes C_3, \quad
\U(\O_3)  \cong Q_{12} \cong C_3 \rtimes C_4, \quad \text{ and } \quad 
\U(\O_5) \cong C_6.
\end{equation}

\begin{proof}[Proof of \Cref{iso implies conj for exc}]
Suppose first that $H$ is imprimitive. Suppose in particular that the action of $H$ on the $2$-dimensional $D$-space $V$ consisting of column vectors decomposes into a direct sum of non-trivial $H$-invariant subspaces. Then, as in \Cref{description imprimitive}, $H$ is conjugated to a group $\begin{psmallmatrix} \Gamma_1 & 0 \\ 0 & \Gamma_2
\end{psmallmatrix} \leqslant \GL_2(\O)$, with $\Gamma_i \leqslant \U(\O)$ for $i \in \{ 1,2\}$. Say $z \in \U(A)$ such that $H^z = \begin{psmallmatrix} \Gamma_1 & 0 \\ 0 & \Gamma_2
\end{psmallmatrix}$. Since $H$ is isomorphic to a subgroup of $G$, there is some subgroup $K \leqslant G$ such that $V$ decomposes into two $1$-dimensional $K$-invariant subspaces. In particular, there is a $w \in \U(D)$ such that $K^w = \begin{psmallmatrix}
    \Lambda_1 & 0 \\ 0 & \Lambda_2
\end{psmallmatrix}$, with $\Gamma_i \cong \Lambda_i$ for $i \in \{1,2\}$.

An investigation of the conjugacy classes of finite subgroups of the groups \eqref{unit group ring of integers} and \eqref{eq:units_of_quat_orders} shows that all isomorphic subgroups of $\U(\O)$ with $\O$ an order in $D$ are conjugate. In particular, for $i \in \{1,2\}$, there exist $x_i \in \U(\O)$ such that $\Gamma_i^{x_i} = \Lambda_i$. In particular,
\[
\begin{pmatrix}
    x_1^{-1} & 0 \\
    0 & x_2^{-1}
\end{pmatrix} H^z \begin{pmatrix}
    x_1 & 0 \\ 0 & x_2
\end{pmatrix} = K^w. 
\]

Next suppose that the action of $H$ on $V$ does not decompose into non-trivial $H$-invariant subspaces. I.e. $H \cong (\Gamma \times \Gamma) \rtimes C_2$ for some $\Gamma \leqslant \U(\O)$, or $H$ is primitive. Then $\Span_{\Q}\{H\}$ is simple. Suppose first that $A = \Ma_2(\qa{-a}{-b}{\Q})$ or $\Ma_2(\Q)$ for some natural numbers $a,b \in \N_0$. Then the conclusion immediately follows from \Cref{upgrade iso to conjugation}.

Now suppose that $A = \Ma_2(\Q(\sqrt{-d}))$ with $d$ a square-free natural number.  If $\Span_{\Q} \{ H \} =A$, then the conclusion follows from \Cref{upgrade iso to conjugation}. Otherwise, \Cref{subalgebra of exceptional} yields that $\Span_{\Q} \{ H \}$ is either $\Ma_2(\Q)$ or a quaternion algebra with centre $\Q$. Therefore, $\Span_{\ZZ(A)} \{ H\}$ is $\ZZ(A)$-subalgebra of $A$ containing $\Span_{\Q} \{ H \}$. Since $\dim_{\Q} \Span_{\Q} \{ H \} = 4$ and $\dim_{\Q} A = 8$, one has that $\Span_{\ZZ(A)} \{ H\} =A$. In particular, $\Span_{\ZZ(A)} \{ H\}$ is a simple algebra and hence \Cref{upgrade iso to conjugation} yields the desired conclusion.
\end{proof}

\begin{proposition}\label{contained in a spanning - imprimitive}
    Let $H \leqslant \GL_2(\O)$ be a finite imprimitive subgroup. If there exist spanning subgroups of $\GL_2(D)$ for which $H$ is order-admissible, then there is some spanning subgroup containing an isomorphic copy of $H$.
\end{proposition}
\begin{table}[h!]
    \centering
    \caption{\footnotesize Per isomorphism type of the exceptional matrix algebra $A$, the isomorphism types of the maximal finite imprimitive group $H \leqslant \GL_2(\O)$ which  occurs as a subgroup of some $\Gamma \in \mc{S}(\Ma_2(\O))$.}
    \begingroup
\setlength{\tabcolsep}{10pt} 
\renewcommand{\arraystretch}{1.5} 
 \label{tab:imprimitive}
    {\footnotesize
    \begin{tabular}{l|l|l}
        \multicolumn{1}{c|}{$A$} & \multicolumn{1}{c|}{$H$}  & \textsc{SmallGroupID} of a $\Gamma \in \mc{S}(\Ma_2(\O))$ containing $H$ \\ \hline 
         $\Ma_2(\Q)$ & $(C_2 \times C_2) \rtimes C_2$ & \texttt{[8,3]} \\ \hline 
         $\Ma_2(\Q(i))$ & $(C_4 \times C_4) \rtimes C_2$ &  \texttt{[32,11]}\\ \hline 
         $\Ma_2(\Q(\sqrt{-2}))$ & $(C_2 \times C_2) \rtimes C_2$ & \texttt{[16,8]}\\ \hline 
         $\Ma_2(\Q(\sqrt{-3}))$ & $(C_6 \times C_6) \rtimes C_2$ & \texttt{[72,30]}\\ \hline 
         $\Ma_2(\mathbb{H}_2)$ & $(\SL_2(\F_3) \times \SL_2(\F_3))\rtimes C_2$ & \texttt{[1152,155468]} \\ \hline 
         $\Ma_2(\mathbb{H}_3)$ & $(Q_{12} \times Q_{12})\rtimes C_2$ & \texttt{[288,389]}\\ \hline 
         $\Ma_2(\mathbb{H}_5)$ & $(C_2 \times C_2) \rtimes C_2$ & \texttt{[240,90]} \\ 
    \end{tabular}}
    \endgroup
\end{table}
\begin{proof}[Proof of \Cref{contained in a spanning - imprimitive}]
    This follows from a straightforward case-by-case analysis for each isomorphism type of exceptional components $A$ such that $\GL_2(D) = \U(A)$, combined with the list of spanning subgroups given by \cite[Corollary 12.12, Tab. 12.1 and Tab. 12.2]{EricAngel1} (see also \Cref{tableB} in \Cref{appendix section}). By \Cref{description imprimitive} (see also \cref{rk maximal imprimitive in order}), all finite imprimitive subgroups are isomorphically contained in a largest imprimitive subgroup, which is of the form $(\U(\O) \times \U(\O))\rtimes C_2$. We summarise the maximal imprimitive subgroups which appear as a subgroup of a spanning group in \Cref{tab:imprimitive} by giving the \textsc{SmallGroupID} of a spanning group containing it. \Cref{description imprimitive} then allows to deduce all other imprimitive subgroups which are contained in a spanning group. Note that from \Cref{tab:imprimitive}, it follows that only for $\Ma_2(\mathbb{H}_5)$, there are no order-admissible spanning groups for the largest imprimitive subgroup $(\U(\O) \times \U(\O))\rtimes C_2$. The largest imprimitive subgroup which is order-admissible, is $(C_2 \times C_2) \rtimes C_2$, which indeed appears as a subgroup of a spanning group. 
    
    We also give, whenever it exists (i.e. in every case except for $\Ma_2(\Q(\sqrt{-2}))$ and $\Ma_2(\mathbb{H}_5)$), the \textsc{SmallGroupID} of a spanning subgroup \emph{isomorphic} to the maximal imprimitive subgroup under consideration.
\end{proof}

\subsubsection{Subgroups spanning an exceptional subalgebra}
Let $\Ma_2(D)$ be an exceptional matrix algebra, $\O$ a maximal order in $D$ and $G$ a finite subgroup of $\GL_2(\O)$. In this section, we describe when $\Span_{\Q} \{G\}$ is an exceptional matrix algebra. This problem consists of two steps: \smallskip

\begin{enumerate}
    \item[{\it Step 1:}] Determine the sets $\mc{S}(\Ma_2(\O))$. In other words, determine the isomorphism type of finite groups $G$ which have a faithful irreducible representation into $\Ma_2(D)$.  
    \item[{\it Step 2:}] Determine which $\Q$-subalgebras isomorphic to an exceptional matrix algebra $\Ma_2(D)$ contains.
\end{enumerate}

Step $1$ was realised by Eisele--Kiefer--Van Gelder in \cite[Table 2]{EKVG}. The sets $\mc{S}(\Ma_2(\O))$ can also be read from \Cref{tableB} in \Cref{appendix section}. We now deal with Step $2$.

Firstly, any matrix algebra $\Ma_2(D)$ contains a copy of $\Ma_2(\Q)$. Therefore, by \eqref{amalgam M2 over Q}, $\GL_2(D)$ contains $D_8$ and $D_{12}$ as subgroups. In summary, for $G$ a finite subgroup of $\GL_2(\I_d)$, one has that 
\begin{equation}\label{which exc matrix in Qsqrt-d}
   \Span_{\Q} \{ G\} \text{ is an exceptional matrix algebra} \iff G \leqslant D_8,D_{12} \text{ or } G \in \mc{S}(\Ma_2(\I_d)).
\end{equation}

Next, suppose that $D = \qa{-a}{-b}{\Q}$. Then it could be that $\Span_{\Q} \{G\} \cong \Ma_2(\Q(\sqrt{-d}))$ for some square-free $d \in \N$. Recall \eqref{when span set non-empty} saying that  $(a,b) \in \{(1,1),(1,3),(2,5)\}$ if $\mc{S}(\O) \neq \emptyset$. Using the Albert--Brauer--Hasse--Noether embedding theorem, one can compute the following.

\begin{lemma}\label{which Qsqrt-d in qa}
The quaternion algebra $\qa{-a}{-b}{\Q}$ with $(a,b) \in \{(1,1),(1,3),(2,5)\}$ contains the following subfields of the form $\Q(\sqrt{-d})$ with $d \in \{1,2,3 \}$:
    \begin{equation}\label{imaginary in quaternion}
        \begin{array}{ll}
           (1,1):  & d=1,2,3   \\
        (1,3):   & d= 1,3  \\
          (2,5):   & d= 2,3  \\
        \end{array}
    \end{equation}
Moreover, in these cases $\mc{I}_d$ embeds in the maximal order $\O_{a,b}$ of $\qa{-a}{-b}{\Q}$.
Consequently, $\GL_2( \qa{-a}{-b}{\Q})$ contains all finite subgroups of $\GL_2(\Q(\sqrt{-d}))$ for those values of $d$. 
\end{lemma}

Although $\Ma_2(\mc{I}_d)$ embeds in $\Ma_2(\O_{a,b})$, this does not mean that every spanning subgroup of $\Ma_2(\mc{I}_d)$ is isomorphic to a subgroup of a spanning subgroup of $\Ma_2(\O_{a,b})$. However, they usually are, which can be verified by hand using the tables in \Cref{appendix section}.

\begin{proposition}\label{contained in spanning full exc}
    Let $A$ be an exceptional matrix algebra and $\O$ a maximal order such that $\mc{S}(\O)$ is non-empty. Let $H$ be a finite subgroup of $\U(\O)$ such that $\Span_{\Q} \{ H \}\cong \Ma_2(\Q(\sqrt{-d}))$ for $d \in \N$. Then $H$ is conjugated in $\U(A)$ to a subgroup of some $\Gamma \in \mc{S}(\O)$, with exception of the cases in \Cref{tab:exceptions spanning}.
\end{proposition}
\begingroup
\setlength{\tabcolsep}{10pt} 
\renewcommand{\arraystretch}{1.5}
\begin{table}[h!]
    \caption{\footnotesize Cases of triples $(d,A,H)$ such that \Cref{contained in spanning full exc} is not satisfied.}
    \label{tab:exceptions spanning}
    \centering
    {\footnotesize
    \begin{tabular}{l|l|l}
         \multicolumn{1}{c|}{$d$} & \multicolumn{1}{c|}{$A$}&  \multicolumn{1}{c}{$H$}\\ \hline 
         $1$ &  $\Ma_2(\qa{-1}{-3}{\Q})$  & \texttt{[48,33]} or \texttt{[96,67]} \\ \hline 
         $3$ & $\Ma_2(\qa{-2}{-5}{\Q})$ & $S(\Ma_2(\mc{I}_3)) \setminus \{\SL_2(\F_3), \: C_3 \rtimes D_8 = $ \texttt{[24,8]}$\}$.
    \end{tabular}}
\end{table}
\endgroup

\subsubsection{Possible subgroups which are Frobenius complements}

Recall that a finite subgroup $H$ of a finite-dimensional division algebra $D$ is a Frobenius complement, which have been classified by Amitsur. Moreover, by \Cref{subalgebra of exceptional} if $\Span_{\Q}\{ H\} =D$ is contained in an exceptional matrix algebra, then $D$ is a field or a quaternion algebra. It turns out that this restricts $H$ strongly, which we make precise in the following proposition.

\begin{proposition}\label{span division}
Let $A$ be an exceptional matrix algebra and $\O$ an order in $A$. Suppose that $\mc{S}(\mc{O})$ is non-empty and let $H$ be a finite subgroup of $\U(\O)$ such that $\Span_{\Q} \{H\}$ is a division algebra. Then the following hold:
\begin{enumerate}
    \item\label{it: nc, the prim} If $H$ is not cyclic, then $H$ is primitive or conjugated to $\{ \begin{psmallmatrix}
        u & 0 \\ 0 & 1
    \end{psmallmatrix} \mid u \in N\}$ with $N$ a non-cyclic subgroup of $\U (\O)$.
    \item\label{it: field case} If $\Span_{\Q} \{H\}$ is a field, $H \cong C_n$, with $n$ a divisor of $8,10$ or $12$.
    \item\label{it: div alg case} If $\Span_{\Q} \{H\}$ non-commutative, $H$ is isomorphic to $\SL_2(\F_3), \SU_2(\F_3), \SL_2(\F_5)$ or $Q_{4m}$, with $m$ a divisor of $4,5$ or $6$, and all of these cases occur for some $A$.
    \item \label{it: in spanning} There exists a $\Gamma \in \mc{S}(\O)$ containing an isomorphic copy of $H$.
\end{enumerate}
Moreover, whether $\mc{O}$ actually contains such an $H$ is recorded in \Cref{table iso types H}.
\end{proposition}

\begin{table}[h!]
\caption{\footnotesize Per isomorphism type of the exceptional matrix algebra $A$, the isomorphism types of all finite groups $H$ which span a division algebra that occur as a subgroup of a $\Gamma \in \mc{S}(\O)$.}
\label{table iso types H}
\begingroup
\setlength{\tabcolsep}{10pt} 
\renewcommand{\arraystretch}{1.5} 
{\footnotesize \begin{tabular}{l|l}
\multicolumn{1}{c|}{$A$}                  & \multicolumn{1}{c}{$H$}     \\ \hline
$\Ma_2(\Q)$ &  \hspace{2pt} $C_n$, for $n \mid 6$ \\ \hline
$\Ma_2(\Q(i))$    & \hspace{-9pt} \begin{tabular}{l}
$C_{n}$, for $n \mid 8$ or $12$ \\ \hline
$Q_8$, $Q_{12}$, $\SL_2(\F_3)$  
\end{tabular}\\ \hline
$\Ma_2(\Q(\sqrt{-2}))$ & \hspace{-9pt} \begin{tabular}{l} $C_n$, for $n \mid 8$ \\ \hline $Q_8$,  $\SL_2(\F_3)$ \end{tabular} \\ \hline                                                              
$\Ma_2(\Q(\sqrt{-3}))$ & \hspace{-9pt} \begin{tabular}{l}
$C_n$, for $n \mid 12$ \\ \hline 
$Q_8$, $Q_{12}$, $\SL_2(\F_3)$
\end{tabular} \\  \hline
$\Ma_2(\mathbb{H}_2)$,  $\Ma_2(\mathbb{H}_3)$ and $\Ma_2(\mathbb{H}_5)$ & \hspace{-9pt} \begin{tabular}{l} $C_n$, for $n \mid 8$, $10$ or $12$ \\ \hline $Q_8$, $Q_{12}$, $Q_{16}$, $Q_{20}$, $Q_{24}$, $\SL_2(\F_3)$, $\SU_2(\F_3)$, $\SL_2(\F_5)$\end{tabular}
\end{tabular}}
\endgroup
\end{table}

\begin{example}
\label{example faithful irreps SLSU}
The groups $\SL_2(\F_3), \SU_2(\F_3)$ and $ \SL_2(\F_5)$ have a faithful irreducible $\Q$-representation into a unique quaternion algebra, but also into an exceptional matrix  algebra. For example using the Wedderga package\footnote{Recall that the groups have \textsc{SmallGroupID's} respectively \texttt{[24,3]},  \texttt{[48,28]} and \texttt{[120,5]}.}, one verifies that these groups have the following faithful irreducible representations.
\[
\begin{array}{l}
     \SL_2(\F_3) \rightsquigarrow \qa{-1}{-1}{\Q} \text{ and }  \Ma_2(\Q(\sqrt{-3})), \\[0.2 cm]
     \SU_2(\F_3) \rightsquigarrow \qa{-1}{-1}{\Q(\zeta_8+\zeta_8^{-1})} \text{ and } \Ma_2(\mathbb{H}_3),\\[0.2cm]
      \SL_2(\F_5) \rightsquigarrow \qa{-1}{-1}{\Q(\zeta_5 + \zeta_5^{-1})} \text{ and } \Ma_2(\mathbb{H}_3).
\end{array}\]  
\end{example}

The isomorphism types in \Cref{span division} arise thanks to the following description of finite subgroups of a quaternion algebra, which follows from \cite[Prop. 32.4.1, Lemma 32.6.1 \& Prop. 32.7.1]{Voight}.

\begin{lemma}[\cite{Voight}]\label{spanning of quaternion}
    Let $\mathcal{O}$ be an order in $D =\qa{a}{b}{F}$ with $F$ a number field. If $H \leqslant \U (\mathcal{O})$ is a finite subgroup such that $\Span_{\Q}\{ H\} \cong D$, then $H$ is isomorphic to one of the following:
\begin{itemize}
    \item $Q_{4m} = \langle x,y \mid x^{2m} = 1, y^2= x^{m}, yx = x^{-1}y \rangle$, a generalised quaternion group,
    \item $\SL_2(\F_3)$, $\SU_2(\F_3)$ or $\SL_2(\F_5)$.
\end{itemize} 
\end{lemma}

Another important ingredient for \Cref{span division} are the following restrictions on the possible order of elements in an exceptional matrix algebra obtained in \cite{EKVG} (see also \cite[Proposition 12.1.1]{EricAngel1} for the version below).

\begin{theorem}[\cite{EKVG}]\label{recollection of props}
Let $A = \Ma_2(D)$ be an exceptional matrix algebra and $\O$ an order in $A$ such that $\mc{S}(\O)$ is non-empty. Consider $h \in \O$. If $D = \Q(\sqrt{-d})$, then the following hold:
\begin{itemize}
    \item For $h \in H$, the order $o(h)$ divides $8$ or $12$.
    \item If $o(h)=8$, then $D = \Q(\sqrt{-1})$ or $\Q(\sqrt{-2})$.
    \item If $o(h)=12$, then $D =\Q(\sqrt{-1})$ or $\Q(\sqrt{-3})$.
\end{itemize}
If $D= \qa{-a}{-b}{\Q}$, then $o(h)$ divides $8,10$ or $12$.
\end{theorem}

We can now proceed to the proof.

\begin{proof}[Proof of \Cref{span division}]
Suppose that $H$ is not cyclic. Suppose $H$ is imprimitive, and hence of the form given by \Cref{description imprimitive}, i.e. $H \cong (\Gamma \times \Gamma)\rtimes C_2$ or $H \cong \Gamma_1 \times \Gamma_2$, for finite $\Gamma,\Gamma_1,\Gamma_2 \leqslant \U(\O)$. Since $H$ spans a division algebra by assumption, it must be a Frobenius complement. Recall that a Frobenius complement does not have a subgroup of the form $C_p \times C_p$ for $p$ prime, see \cite[Theorem 18.1]{PassPerm}. Thus if $H$ is a Frobenius complement, it must be of the form $\Gamma_1 \times \Gamma_2$ with the orders $|\Gamma_i|$ for $i = 1,2$ relatively prime and $\Gamma_i \leqslant \U(\O)$. In that case $\Gamma_1 \times \Gamma_2$ is indeed a Frobenius complement by \cite[Exercise 11.4.3]{EricAngel1}. Now recall \eqref{when span set non-empty} which asserts in terms of $A$ when $\mc{S}(\O)$ is non-empty.  For these values the unit groups are recorded in \eqref{unit group ring of integers} and \eqref{eq:units_of_quat_orders}. From these it follows immediately that the only way for $H$ to be non-cyclic and of the form $\Gamma_1 \times \Gamma_2$, is that some $\Gamma_i$ is trivial, which finishes the proof of statement \eqref{it: nc, the prim}.

We now consider statements \eqref{it: field case} and \eqref{it: div alg case}. By \Cref{subalgebra of exceptional}, $\Span_{\Q} \{H\}$ must be a field or a quaternion algebra. In the former case, since $H$ is a subgroup of a field, it must be cyclic, say $H \cong C_n$. In the latter case, $H$ is isomorphic to one of the groups in \Cref{spanning of quaternion}, i.e. $H \cong Q_{4m}$ if it is a generalised quaternion group, or $H \cong \SL_2(\F_3)$, $\SU_2(\F_3)$ or $\SL_2(\F_5)$. In both cases, using \Cref{recollection of props} we obtain that both $m$ and $n$ must be a divisor of $8$, $10$ or $12$. In particular, we obtain part \eqref{it: field case}.  We now show that in the non-commutative case, $m$ is in fact at most a divisor of $4$, $5$ or $6$. 

In \cite[Example 3.5.7]{EricAngel1}, the strong Shoda pairs $(\Gamma, K)$ are computed for a generalised quaternion group $Q_{4n}$. Recall that in \Cref{form of SSP idempotent result}, we show for which $e = e(Q_{4n}, \Gamma,K)$, the associated representation 
$\psi \colon Q_{4n} \to \Q Q_{4n} e$ is faithful, namely when \[\bigcap_{g \in Q_{4n}} K^g = 1.\]
But from \cite[Example 3.5.7]{EricAngel1} it is then immediate that the only SSP for $Q_{4n} = \langle x,y \mid x^{2n} = 1, x^n = y^2, yx = x^{-1}y\rangle$ which yields a faithful representation, is $(\Gamma, K_{2n})$, with $\Gamma = \langle x\rangle$, and $K_{2n} = \langle x^{2n} \rangle = \{1\}$. Moreover in the authors in \emph{loc. cit.}\footnote{Note that there is a small typo, $A_d$ should read $\qa{\alpha_d^2 -4}{-1}{\Q(\alpha_d)}$ rather than $\qa{\alpha_d^2 -4}{-1}{\Q}$, as seen in eq.~(3.5.1) right below it.}, show that the associated component $\Q Q_{4n} e$ is given by $\qa{\alpha_{2n}^2-4}{-1}{\Q(\alpha_{2n})}$, where $\alpha_{2n} = \zeta_{2n}+\zeta_{2n}^{-1}$.

So suppose that for $m \in \{8,10,12\}$, $Q_{4m} \leqslant \GL_2(\O)$. Then in particular, there exists a faithful representation \[\psi \colon Q_{4m} \to \GL_2(\O) \subseteq \Ma_2(D).\]
By the previous paragraph, the image of $\psi$ is isomorphic to $\qa{-1}{-1}{\Q(\zeta_{2m}+\zeta_{2m}^{-1})}$. Moreover,
\[\ \dim_{\Q}(\Q(\zeta_{2m}+\zeta_{2m}^{-1})) = \frac{\phi(2m)}{2} = 4, \text{ for all } m \in \{8,10,12\}. \]
However, since we assume $\Span_{\Q}\{Q_{4m}\}\leqslant \GL_2(D)$, \cref{remark degree quat alg} implies that the degree over $\Q$ of the centre of $\Span_{\Q}\{Q_{4m}\}$ is at most $2$, a contradiction. We conclude that $m$ is a divisor of $4$, $5$ or $6$, which finishes the first part of \eqref{it: div alg case}.

For the second part of \eqref{it: div alg case}, as well as \eqref{it: in spanning}, we use the tables of finite spanning subgroups of exceptional matrix algebras provided by \cite[Corollary 12.12, Tab. 12.1 and Tab. 12.2]{EricAngel1}. In particular, in \Cref{table iso types H} we give, for each exceptional matrix algebra, which groups from the statement occur as a subgroup of some finite spanning subgroup. In \Cref{tableA} in \Cref{appendix section} we listed per spanning subgroup the groups from the statement which are isomorphically contained in it, and the content of \Cref{table iso types H} is then obtained by combining the information in \Cref{tableA} and \Cref{tableB}.

We now moreover prove for every possible isomorphism type of $H$ that whenever it does not occur as a subgroup of some finite spanning subgroup, then in fact $H$ cannot occur as a finite subgroup of $\U(\O)$ at all, which concludes the proof.

From \Cref{table iso types H}, we remark that the groups $Q_{4m}$ for $m \in \{4,5,6\}$ do not occur as a subgroup of a spanning subgroup of $\Ma_2(\Q(\sqrt{-d}))$ for $d \in \{0,1,2,3\}$. From the same reasoning as above, if these were subgroups of $\GL_2(\I_d)$, then this would mean that $Q_{4m}$ has a faithful representation into $\qa{-1}{-1}{\Q(\zeta_{2m}+\zeta_{2m}^{-1})}$. In particular the $\Q$-linear span of $H = Q_{4m}$ is isomorphic to the latter quaternion algebra. But 
\[\dim_{\Q}(\Q(\zeta_{2m} + \zeta_{2m}^{-1})) = \frac{\phi(2m)}{2} = 2, \text{ for all } m \in \{4,5,6\}.\]
It now follows from \Cref{remark degree quat alg} that this is only possible in case $H$ is a subgroup of $\GL_2(D)$ with $D$ a quaternion algebra over $\Q$, a contradiction since we assumed $H \leqslant \GL_2(\I_d)$.

We now show moreover that $Q_{12}$ cannot occur as a subgroup of $\GL_2(\Z(\sqrt{-2}))$. Recall from above that the unique faithful irreducible $\Q$-representation of $Q_{12}$ is $\qa{-3}{-1}{\Q}$. It is well-known that the latter division algebra splits over $\Q(\sqrt{-3})$ and $\Q(i)$. Now consider a square-free integer $d \in \N$. If $Q_{12}$ were a subgroup of $\GL_2(\Q(\sqrt{-d}))$ such that $\Span_{\Q}\{Q_{12}\}$ is a division $\Q$-subalgebra, then by the aforementioned faithful irreducible $\Q$-representation, \[\Span_{\Q}\{Q_{12}\} \cong \qa{-3}{-1}{\Q}.\] This would imply that 
\[\dim_{\Q} \Span_{\Q(\sqrt{-d})}\{Q_{12}\} = \dim_{\Q}\Q(\sqrt{-d}) \cdot \dim_{\Q}\qa{-3}{-1}{\Q} = 8.\]
Hence $\Span_{\Q(\sqrt{-d})}\{Q_{12}\} = \Ma_{2}(\Q(\sqrt{-d}))$. But then $\qa{-3}{-1}{\Q}$ would have to split over $\Q(\sqrt{-d})$. However, $\qa{-3}{-1}{\Q}$ does not split over $\Q(\sqrt{-2})$ (but does over $\Q(i)$ and $\Q(\sqrt{-3})$), a contradiction. In conclusion, $Q_{12}$ is not a subgroup of $\GL_2(\mc{I}_2)$, as stated.
If $\SL_2(\F_3)$ were a finite subgroup of $\GL_2(\Q)$ spanning a non-commutative division algebra, then this division algebra must be a quaternion algebra by \Cref{subalgebra of exceptional}. But such a quaternion algebra is of $\Q$-dimension at least $4$, a contradiction since in this case, $\Span_{\Q}(\SL_2(\F_3)) \subseteq \Ma_2(\Q)$, which is of dimension $4$ over $\Q$.

Finally, we prove that $\SU_2(\F_3)$ and $\SL_2(\F_5)$ cannot occur as a subgroup of $\GL_2(\Z[\sqrt{-d}])$, for $d \in \{0,1,2,3\}$, which finishes the proof. Note that when $\Span_{\Q}\{G\}$ is a division algebra, the latter must occur as a component of the Wedderburn--Artin decomposition of $\Q G$. Calculating the Wedderburn--Artin decompositions of $\Q[\SU_2(\F_3)]$ and $\Q[\SL_2(\F_5)]$ via the Wedderga package in GAP, one sees (see also \Cref{example faithful irreps SLSU}) that the only non-commutative division algebras occurring are respectively $\qa{-1}{-1}{\Q(\zeta_8+\zeta_8^{-1})}$ and $\qa{-1}{-1}{\Q(\zeta_5+\zeta_5^{-1})}$. But again as in the paragraph before last, the centres of these quaternion algebras have degree $2$ over $\Q$, meaning that from \Cref{remark degree quat alg} it follows that $\SL_2(\F_5)$ and $\SU_2(\F_3)$ cannot occur as finite subgroups of $\GL_2(\Z[\sqrt{-d}])$ for $d \in \{0,1,2,3\}$.
\end{proof}

\subsubsection{Some applications of the classification of finite subgroups}

To start we point out that all the results obtained can be packaged to obtain \Cref{all in a spanning}. 

\begin{proof}[Proof of \Cref{all in a spanning}]
Let $H$ be a finite subgroup of $\U(\O)$. In case that $H$ is imprimitive the desired statement was obtained in \Cref{contained in a spanning - imprimitive}.

Suppose that $H$ is primitive. Then by \Cref{subalgebra of exceptional}, either $\Span_{\Q} \{ H \}$ is a division algebra or $\Span_{\Q} \{ H \}$ is an exceptional matrix subalgebra of $A$. In the former case the statement was obtained in \Cref{span division}. In the latter case it holds by \Cref{contained in spanning full exc}.   
\end{proof}

Next, we explain how one can deduce an amalgamation for $\GL_2(\mc{I}_2)$.

\begin{remark}\label{amalgam GL2 for sqrt-2}
In \cite{Hat} it was proven that there is an isomorphism
\[\PGL_2(\Z[\sqrt{-2}]) \cong (S_4 \ast_{C_4} D_8) \ast_{(C_3 \ast C_2)} (D_6 \ast_{C_2} (C_2 \times C_2)),\]
where the amalgamating subgroup $C_3 \ast C_2 \cong \PSL_2(\Z)$. It is a classical fact that for the canonical epimorphism $\pi\colon \GL_2(\Z[\sqrt{-2}]) \rightarrow \PGL_2(\Z[\sqrt{-2}])$, one has that 
\[\GL_2(\Z[\sqrt{-2}]) \cong \pi^{-1}(S_4 \ast_{C_4} D_8) \ast_{\pi^{-1}(C_3 \ast C_2)} \pi^{-1}(D_6 \ast_{C_2} (C_2 \times C_2)).\]
We describe the inverse image $\pi^{-1}$ of each of the finite subgroups appearing in the decomposition. Firstly note that $|\pi^{-1}(H)|= 2 \cdot|H|$ for each of those subgroups. In particular for $H= S_4$, we have that $|\pi^{-1}(H)|=48$. Inspecting the obtained list of finite subgroups of $\GL_2(\Z[\sqrt{-2}])$, we see that $48$ is the largest order of a finite subgroup. Moreover, $\GL_2(\F_3)$ is the only one of that order. Hence $\pi^{-1}(S_4) = \GL_2(\F_3)$. Similarly, one deduces that $\pi^{-1}(D_8) = SD_{16}$, and the amalgamating $C_4$ is pulled back to a cyclic group $C_8$. 

For the other part of the amalgam we use that the centre of the subalgebra $\Ma_2(\Z)$ equals the centre of $\Ma_2(\Z[\sqrt{-2}])$. Hence, having in mind that the main amalgamation is over $\PSL_2(\Z)$, we see that the inverse image $\pi^{-1}$ of the amalgamating $C_2$ is isomorphic $C_4$. Therefore, looking again at the list of possible subgroups of $\GL_2(\Z[\sqrt{-2}])$, $ \pi^{-1}(D_6) = D_{12} \text{ and } \pi^{-1}(C_2 \times C_2) = Q_8$. In conclusion,
\[\GL_2(\mc{I}_2) \cong (\GL_2(\F_3) \ast_{C_8} SD_{16}) \ast_{\SL_2(\Z)} (D_{12} \ast_{C_2 \times C_2} Q_8).\]
\end{remark}

\appendix
\clearpage
\section{\texorpdfstring{Tables of groups with a faithful exceptional $2 \times 2$ embedding}{Tables of groups with a faithful exceptional 2 by 2 embedding}}\label{appendix section}

In this appendix we reproduce\footnote{Including the group with \textsc{SmallGroupID} \texttt{[24, 1]}, which was accidentally omitted in the original table.} \cite[Table~2]{EKVG}, listing those finite groups $G$ that have a faithful exceptional matrix component (see \Cref{def_exc_comps}) in the Wedderburn--Artin decomposition of the rational group algebra $\mathbb{Q}G$. In particular, these groups appear as a spanning subgroup of an exceptional $\Ma_2(D)$. For all these groups, we list properties relevant for \Cref{span division} in the first table. 
\setcounter{table}{0}
\begingroup
\renewcommand\thetable{\Alph{table}}

\begin{table}[h!]
\caption{\footnotesize Content of the first table in this appendix.}
\label{tableA}
{\footnotesize
\noindent\begin{tabular}{p{0.25\linewidth}p{0.7\linewidth}}
 \textsc{SmallGroupID}: & The identifier of the group $G$ in the \textsc{SmallGroup} library \\[.4cm] 
Spectrum & The spectrum of the group, i.e. the set of orders of elements in $G$. \vspace{6pt} \\ [.4cm]
 Quaternion subgroups: & The generalised quaternion subgroups $Q_{4n}$ isomorphically appearing as a subgroup of $G$.\vspace{6pt} \\[.4cm]
 Linear subgroups: & The groups $\SL_2(\F_3)$, $\SU_2(\F_3)$ and $\SL_2(\F_5)$ isomorphically appearing as a subgroup of $G$. (Recall that they have \textsc{SmallGroupID}'s respectively \texttt{[24,3]},  \texttt{[48,28]} and \texttt{[120,5]}). \vspace{6pt}\\[.4cm]
 Dihedral subgroups: & The dihedral groups $D_n$ isomorphically appearing as a subgroup of $G$. \vspace{6pt}\\[.4cm]
 Subgroups from list: & Isomorphism types of strict subgroups of $G$ which themselves have an exceptional matrix component, given in terms of their \textsc{SmallGroupID}.

\end{tabular}}
\end{table}
\vspace{-6pt} In the second table, we list properties relevant to property $\Mexc$. In particular we give the exceptional matrix algebras these groups span.

\begin{table}[h!]
\caption{\footnotesize Content of the second table in this appendix.}
\label{tableB}
{\footnotesize 
\noindent\begin{tabular}{p{0.25\linewidth}p{0.7\linewidth}}
\textsc{SmallGroupID}: & The identifier of the group $G$ in the \textsc{SmallGroup} library \\[.4cm] 
 $\Mexc$: & Indicates whether the group satisfies $\Mexc$.\\[.4cm]
 $\mathrm{cl}$: & The nilpotency class of the group; $\infty$ indicates that the group is not nilpotent (omitted for non-solvable groups) \vspace{6pt}\\[.4cm]
 d$\ell$: & Derived length of the group; $\infty$ for non-solvable groups \vspace{6pt}\\[.4cm]
 $[G : F(G)]$: & The index of the Fitting subgroup of $G$ in $G$ (omitted when $G$ itself is nilpotent).\\ [.4cm]
 $\mathrm{cl}(F(G))$: & The nilpotency class of the Fitting subgroup $F(G)$. \\[.4cm]
 Faithful exceptional components: & The exceptional matrix components (with multiplicity) of the Wedderburn--Artin decomposition of $G$ which are a faithful representation of $G$. In particular $G$ is a spanning subgroup of these.\vspace{2pt} \\[.4cm]
 $1 \times 1 $ components: & The division-algebra components appearing in the Wedderburn--Artin decomposition of $\Q G$ (with multiplicity); omitted for groups without $\Mexc$.  \\
 
\end{tabular}}
\end{table}
Recall that we use the following shorthands for the appearing quaternion algebras:
\[\mathbb{H}_2 = \qa{-1}{-1}{\mathbb{Q}}, \quad \mbox{ and } \quad \mathbb{H}_3 = \qa{-1}{-3}{\mathbb{Q}}, \quad \mbox{ and } \quad \mathbb{H}_5 = \qa{-2}{-5}{\Q},\] 
and that $\zeta_k$ represents a complex, primitive $k$-root of unity.
\endgroup

\newgeometry{right=1.75cm,left=1.75cm,top=0cm,bottom=0cm}
 \begin{landscape}
 \scriptsize

\begin{longtable}{@{}lcllll@{}} \\ \toprule[1.5pt]
\textsc{SmallGroupID} & Spectrum & Quaternion subgroups & Linear subgroups & Dihedral subgroups & Subgroups from list\
   \\ \midrule 
\endfirsthead \toprule[1.5pt] \textsc{SmallGroup} ID & Spectrum & Quaternion subgroups & Linear subgroups & Dihedral subgroups & Subgroups from list\
   \\ \midrule 
\endhead \hline \multicolumn{6}{c}{continued}\\ \midrule[1.5pt]\endfoot\bottomrule[1.5pt]\endlastfoot
\texttt{{[}6, 1{]}} &  $ \{1, 2, 3\} $ &  &  &  &  \\
\texttt{{[}8, 3{]}} &  $ \{1, 2, 4\} $ &  &  &  &  \\
\texttt{{[}12, 4{]}} &  $ \{1, 2, 3, 6\} $ &  &  & $D_6$ & $ \texttt{{[}6, 1{]}}$ \\
\texttt{{[}16, 6{]}} &  $ \{1, 2, 4, 8\} $ &  &  &  &  \\
\texttt{{[}16, 8{]}} &  $ \{1, 2, 4, 8\} $ & $Q_8$ &  & $D_8$ & $ \texttt{{[}8, 3{]}}$ \\
\texttt{{[}16, 13{]}} &  $ \{1, 2, 4\} $ & $Q_8$ &  & $D_8$ & $ \texttt{{[}8, 3{]}}$ \\
\texttt{{[}18, 3{]}} &  $ \{1, 2, 3, 6\} $ &  &  & $D_6$ & $ \texttt{{[}6, 1{]}}$ \\
\texttt{{[}24, 1{]}} &  $ \{1, 2, 3, 4, 6, 8, 12\} $ &  &  &  &  \\
\texttt{{[}24, 3{]}} &  $ \{1, 2, 3, 4, 6\} $ & $Q_8$ &  &  &  \\
\texttt{{[}24, 5{]}} &  $ \{1, 2, 3, 4, 6, 12\} $ & $Q_{12}$ &  & $D_6$, $D_{12}$ & $ \texttt{{[}12, 4{]}}$ \\
\texttt{{[}24, 8{]}} &  $ \{1, 2, 3, 4, 6\} $ & $Q_{12}$ &  & $D_6$, $D_8$, $D_{12}$ & $ \texttt{{[}8, 3{]}}$, $ \texttt{{[}12, 4{]}}$ \\
\texttt{{[}24, 10{]}} &  $ \{1, 2, 3, 4, 6, 12\} $ &  &  & $D_8$ & $ \texttt{{[}8, 3{]}}$ \\
\texttt{{[}24, 11{]}} &  $ \{1, 2, 3, 4, 6, 12\} $ & $Q_8$ &  &  &  \\
\texttt{{[}32, 8{]}} &  $ \{1, 2, 4, 8\} $ & $Q_8$ &  &  & $ \texttt{{[}16, 6{]}}$ \\
\texttt{{[}32, 11{]}} &  $ \{1, 2, 4, 8\} $ & $Q_8$ &  & $D_8$ & $ \texttt{{[}16, 6{]}}$, $ \texttt{{[}16, 13{]}}$ \\
\texttt{{[}32, 44{]}} &  $ \{1, 2, 4, 8\} $ & $Q_8$, $Q_{16}$ &  & $D_8$ & $ \texttt{{[}16, 6{]}}$, $ \texttt{{[}16, 8{]}}$, $ \texttt{{[}16, 13{]}}$ \\
\texttt{{[}32, 50{]}} &  $ \{1, 2, 4\} $ & $Q_8$ &  & $D_8$ & $ \texttt{{[}16, 13{]}}$ \\
\texttt{{[}36, 6{]}} &  $ \{1, 2, 3, 4, 6, 12\} $ & $Q_{12}$ &  &  &  \\
\texttt{{[}36, 12{]}} &  $ \{1, 2, 3, 6\} $ &  &  & $D_6$, $D_{12}$ & $ \texttt{{[}12, 4{]}}$, $ \texttt{{[}18, 3{]}}$ \\
\texttt{{[}40, 3{]}} &  $ \{1, 2, 4, 5, 8, 10\} $ & $Q_{20}$ &  &  &  \\
\texttt{{[}48, 16{]}} &  $ \{1, 2, 3, 4, 6, 8, 12\} $ & $Q_8$, $Q_{12}$, $Q_{24}$ &  & $D_8$ & $ \texttt{{[}16, 8{]}}$, $ \texttt{{[}24, 1{]}}$, $ \texttt{{[}24, 10{]}}$ \\
\texttt{{[}48, 18{]}} &  $ \{1, 2, 3, 4, 6, 8, 12\} $ & $Q_8$, $Q_{12}$, $Q_{16}$, $Q_{24}$ &  &  & $ \texttt{{[}24, 1{]}}$, $ \texttt{{[}24, 11{]}}$ \\
\texttt{{[}48, 28{]}} &  $ \{1, 2, 3, 4, 6, 8\} $ & $Q_8$, $Q_{12}$, $Q_{16}$ & $\SL_2(\F_3)$ &  & $ \texttt{{[}24, 3{]}}$ \\
\texttt{{[}48, 29{]}} &  $ \{1, 2, 3, 4, 6, 8\} $ & $Q_8$ & $\SL_2(\F_3)$ & $D_6$, $D_8$, $D_{12}$ & $ \texttt{{[}12, 4{]}}$, $ \texttt{{[}16, 8{]}}$, $ \texttt{{[}24, 3{]}}$ \\
\texttt{{[}48, 33{]}} &  $ \{1, 2, 3, 4, 6, 12\} $ & $Q_8$ & $\SL_2(\F_3)$ & $D_8$ & $ \texttt{{[}16, 13{]}}$, $ \texttt{{[}24, 3{]}}$ \\
\texttt{{[}48, 39{]}} &  $ \{1, 2, 3, 4, 6, 12\} $ & $Q_8$, $Q_{12}$, $Q_{24}$ &  & $D_6$, $D_8$, $D_{12}$ & $ \texttt{{[}16, 13{]}}$, $ \texttt{{[}24, 5{]}}$, $ \texttt{{[}24, 8{]}}$, $ \texttt{{[}24, 10{]}}$ \\
\texttt{{[}48, 40{]}} &  $ \{1, 2, 3, 4, 6, 12\} $ & $Q_8$, $Q_{12}$, $Q_{24}$ &  & $D_6$, $D_{12}$ & $ \texttt{{[}24, 5{]}}$, $ \texttt{{[}24, 11{]}}$ \\
\texttt{{[}64, 37{]}} &  $ \{1, 2, 4, 8\} $ & $Q_8$ &  &  & $ \texttt{{[}32, 8{]}}$ \\
\texttt{{[}64, 137{]}} &  $ \{1, 2, 4, 8\} $ & $Q_8$, $Q_{16}$ &  & $D_8$ & $ \texttt{{[}32, 8{]}}$, $ \texttt{{[}32, 11{]}}$, $ \texttt{{[}32, 44{]}}$, $ \texttt{{[}32, 50{]}}$ \\
\texttt{{[}72, 19{]}} &  $ \{1, 2, 3, 4, 6, 8\} $ & $Q_{12}$ &  &  &  \\
\texttt{{[}72, 20{]}} &  $ \{1, 2, 3, 4, 6, 12\} $ & $Q_{12}$ &  & $D_6$, $D_{12}$ & $ \texttt{{[}24, 5{]}}$, $ \texttt{{[}36, 6{]}}$, $ \texttt{{[}36, 12{]}}$ \\
\texttt{{[}72, 22{]}} &  $ \{1, 2, 3, 4, 6\} $ & $Q_{12}$ &  & $D_6$, $D_8$, $D_{12}$ & $ \texttt{{[}24, 8{]}}$, $ \texttt{{[}36, 12{]}}$ \\
\texttt{{[}72, 24{]}} &  $ \{1, 2, 3, 4, 6, 12\} $ & $Q_8$, $Q_{12}$, $Q_{24}$ &  &  & $ \texttt{{[}36, 6{]}}$ \\
\texttt{{[}72, 25{]}} &  $ \{1, 2, 3, 4, 6, 12\} $ & $Q_8$ & $\SL_2(\F_3)$ &  & $ \texttt{{[}24, 3{]}}$, $ \texttt{{[}24, 11{]}}$ \\
\texttt{{[}72, 30{]}} &  $ \{1, 2, 3, 4, 6, 12\} $ & $Q_{12}$ &  & $D_6$, $D_8$, $D_{12}$ & $ \texttt{{[}24, 8{]}}$, $ \texttt{{[}24, 10{]}}$, $ \texttt{{[}36, 6{]}}$, $ \texttt{{[}36, 12{]}}$ \\
\texttt{{[}96, 67{]}} &  $ \{1, 2, 3, 4, 6, 8, 12\} $ & $Q_8$ & $\SL_2(\F_3)$ & $D_8$ & $ \texttt{{[}24, 1{]}}$, $ \texttt{{[}32, 11{]}}$, $ \texttt{{[}48, 33{]}}$ \\
\texttt{{[}96, 190{]}} &  $ \{1, 2, 3, 4, 6, 8\} $ & $Q_8$, $Q_{12}$, $Q_{16}$ & $\SL_2(\F_3)$, $\SU_2(\F_3)$ & $D_6$, $D_8$, $D_{12}$ & $ \texttt{{[}24, 8{]}}$, $ \texttt{{[}32, 44{]}}$, $ \texttt{{[}48, 28{]}}$, $ \texttt{{[}48, 29{]}}$ \\
\texttt{{[}96, 191{]}} &  $ \{1, 2, 3, 4, 6, 8, 12\} $ & $Q_8$, $Q_{12}$, $Q_{16}$, $Q_{24}$ & $\SL_2(\F_3)$, $\SU_2(\F_3)$ & $D_8$ & $ \texttt{{[}32, 44{]}}$, $ \texttt{{[}48, 28{]}}$, $ \texttt{{[}48, 33{]}}$ \\
\texttt{{[}96, 202{]}} &  $ \{1, 2, 3, 4, 6, 12\} $ & $Q_8$ & $\SL_2(\F_3)$ & $D_8$ & $ \texttt{{[}24, 10{]}}$, $ \texttt{{[}32, 50{]}}$, $ \texttt{{[}48, 33{]}}$ \\
\texttt{{[}120, 5{]}} &  $ \{1, 2, 3, 4, 5, 6, 10\} $ & $Q_8$, $Q_{12}$, $Q_{20}$ & $\SL_2(\F_3)$ &  & $ \texttt{{[}24, 3{]}}$ \\
\texttt{{[}128, 937{]}} &  $ \{1, 2, 4, 8\} $ & $Q_8$, $Q_{16}$ &  & $D_8$ & $ \texttt{{[}64, 37{]}}$, $ \texttt{{[}64, 137{]}}$ \\
\texttt{{[}144, 124{]}} &  $ \{1, 2, 3, 4, 6, 8, 12\} $ & $Q_8$, $Q_{12}$, $Q_{16}$, $Q_{24}$ & $\SL_2(\F_3)$, $\SU_2(\F_3)$ &  & $ \texttt{{[}48, 18{]}}$, $ \texttt{{[}48, 28{]}}$, $ \texttt{{[}72, 25{]}}$ \\
\texttt{{[}144, 128{]}} &  $ \{1, 2, 3, 4, 6, 12\} $ & $Q_8$, $Q_{12}$, $Q_{24}$ & $\SL_2(\F_3)$ & $D_6$, $D_{12}$ & $ \texttt{{[}36, 12{]}}$, $ \texttt{{[}48, 40{]}}$, $ \texttt{{[}72, 25{]}}$ \\
\texttt{{[}144, 135{]}} &  $ \{1, 2, 3, 4, 6, 8\} $ & $Q_{12}$ &  &  & $ \texttt{{[}16, 6{]}}$, $ \texttt{{[}72, 19{]}}$ \\
\texttt{{[}144, 148{]}} &  $ \{1, 2, 3, 4, 6, 12\} $ & $Q_8$, $Q_{12}$, $Q_{24}$ &  & $D_6$, $D_8$, $D_{12}$ & $ \texttt{{[}48, 39{]}}$, $ \texttt{{[}72, 20{]}}$, $ \texttt{{[}72, 22{]}}$, $ \texttt{{[}72, 24{]}}$, $ \texttt{{[}72, 30{]}}$ \\
\texttt{{[}160, 199{]}} &  $ \{1, 2, 4, 5, 10\} $ & $Q_8$ &  & $D_8$ & $ \texttt{{[}32, 50{]}}$ \\
\texttt{{[}192, 989{]}} &  $ \{1, 2, 3, 4, 6, 8, 12\} $ & $Q_8$, $Q_{12}$, $Q_{16}$, $Q_{24}$ & $\SL_2(\F_3)$, $\SU_2(\F_3)$ & $D_8$ & $ \texttt{{[}48, 16{]}}$, $ \texttt{{[}64, 137{]}}$, $ \texttt{{[}96, 67{]}}$, $ \texttt{{[}96, 191{]}}$, $ \texttt{{[}96, 202{]}}$ \\
\texttt{{[}240, 89{]}} &  $ \{1, 2, 3, 4, 5, 6, 8, 10, 12\} $ & $Q_8$, $Q_{12}$, $Q_{16}$, $Q_{20}$, $Q_{24}$ & $\SL_2(\F_3)$, $\SU_2(\F_3)$, $\SL_2(\F_5)$ &  & $ \texttt{{[}40, 3{]}}$, $ \texttt{{[}48, 28{]}}$, $ \texttt{{[}120, 5{]}}$ \\
\texttt{{[}240, 90{]}} &  $ \{1, 2, 3, 4, 5, 6, 8, 10\} $ & $Q_8$, $Q_{12}$, $Q_{20}$ & $\SL_2(\F_3)$, $\SL_2(\F_5)$ & $D_6$, $D_8$, $D_{12}$ & $ \texttt{{[}24, 8{]}}$, $ \texttt{{[}40, 3{]}}$, $ \texttt{{[}48, 29{]}}$, $ \texttt{{[}120, 5{]}}$ \\
\texttt{{[}288, 389{]}} &  $ \{1, 2, 3, 4, 6, 8, 12\} $ & $Q_8$, $Q_{12}$, $Q_{24}$ &  & $D_6$, $D_8$, $D_{12}$ & $ \texttt{{[}32, 11{]}}$, $ \texttt{{[}144, 135{]}}$, $ \texttt{{[}144, 148{]}}$ \\
\texttt{{[}320, 1581{]}} &  $ \{1, 2, 4, 5, 8, 10\} $ & $Q_8$, $Q_{16}$, $Q_{20}$ &  & $D_8$ & $ \texttt{{[}64, 137{]}}$, $ \texttt{{[}160, 199{]}}$ \\
\texttt{{[}384, 618{]}} &  $ \{1, 2, 3, 4, 6, 8, 12\} $ & $Q_8$, $Q_{16}$ & $\SL_2(\F_3)$ & $D_8$ & $ \texttt{{[}96, 202{]}}$, $ \texttt{{[}128, 937{]}}$ \\
\texttt{{[}384, 18130{]}} &  $ \{1, 2, 3, 4, 6, 8\} $ & $Q_8$, $Q_{12}$, $Q_{16}$ & $\SL_2(\F_3)$, $\SU_2(\F_3)$ & $D_6$, $D_8$, $D_{12}$ & $ \texttt{{[}96, 190{]}}$, $ \texttt{{[}128, 937{]}}$ \\
\texttt{{[}720, 409{]}} &  $ \{1, 2, 3, 4, 5, 6, 8, 10\} $ & $Q_8$, $Q_{12}$, $Q_{16}$, $Q_{20}$ & $\SL_2(\F_3)$, $\SU_2(\F_3)$, $\SL_2(\F_5)$ &  & $ \texttt{{[}48, 28{]}}$, $ \texttt{{[}72, 19{]}}$, $ \texttt{{[}120, 5{]}}$ \\
\texttt{{[}1152, 155468{]}} &  $ \{1, 2, 3, 4, 6, 8, 12\} $ & $Q_8$, $Q_{12}$, $Q_{16}$ & $\SL_2(\F_3)$, $\SU_2(\F_3)$ & $D_6$, $D_8$, $D_{12}$ & $ \texttt{{[}72, 25{]}}$, $ \texttt{{[}72, 30{]}}$, $ \texttt{{[}384, 618{]}}$, $ \texttt{{[}384, 18130{]}}$ \\
\texttt{{[}1920, 241003{]}} &  $ \{1, 2, 3, 4, 5, 6, 8, 10, 12\} $ & $Q_8$, $Q_{12}$, $Q_{16}$, $Q_{20}$, $Q_{24}$ & $\SL_2(\F_3)$, $\SU_2(\F_3)$, $\SL_2(\F_5)$ & $D_8$ & $ \texttt{{[}120, 5{]}}$, $ \texttt{{[}192, 989{]}}$, $ \texttt{{[}320, 1581{]}}$, $ \texttt{{[}384, 618{]}}$ \\
\end{longtable}
	\end{landscape}
\restoregeometry
\newpage
 \newgeometry{right=1.75cm,left=1.75cm,top=0cm,bottom=0cm}
 \begin{landscape}
 \scriptsize

\begin{longtable}{@{}lcccccll@{}} \\ \toprule[1.5pt]
$\textsc{SmallGroupID} $ & $ \operatorname{(M_{exc})} $ & cl & d$\ell$ & ${[} G : F(G) {]}$  & cl$(F(G))$ & Faithful \
exceptional components & $1\times 1$ components   \\ \midrule 
\endfirsthead \toprule[1.5pt] $\textsc{SmallGroup} ID $ & $ \operatorname{(M_{exc})} $ & cl & d$\ell$ & ${[} G : F(G) {]}$  & cl$(F(G))$ & Faithful \
exceptional components & $1\times 1$ components   \\ \midrule 
\endhead \hline \multicolumn{6}{c}{continued}\\ \midrule[1.5pt]\endfoot\bottomrule[1.5pt]\endlastfoot
\texttt{{[}6, 1{]}} &$\checkmark$ & $\infty$ & $2$ & $2$  & $1$ & 1 $\times$  $M_2(\mathbb Q)$  & $2 {\times} \mathbb{Q}$,  \\
\texttt{{[}8, 3{]}} &$\checkmark$ & $2$ & $2$ &  &  & 1 $\times$  $M_2(\mathbb Q)$  & $4 {\times} \mathbb{Q}$,  \\
\texttt{{[}12, 4{]}} &$\checkmark$ & $\infty$ & $2$ & $2$  & $1$ & 1 $\times$  $M_2(\mathbb Q)$  & $4 {\times} \mathbb{Q}$,  \\
\texttt{{[}16, 6{]}} &$\checkmark$ & $2$ & $2$ &  &  & 1 $\times$  $M_2(\mathbb Q (i))$  & $4 {\times} \mathbb{Q}$, $2 {\times} \mathbb{Q}(i)$,  \\
\texttt{{[}16, 8{]}} &$\checkmark$ & $3$ & $2$ &  &  & 1 $\times$  $M_2(\mathbb Q (\sqrt{-2}))$  & $4 {\times} \mathbb{Q}$,  \\
\texttt{{[}16, 13{]}} &$\checkmark$ & $2$ & $2$ &  &  & 1 $\times$  $M_2(\mathbb Q (i))$  & $8 {\times} \mathbb{Q}$,  \\
\texttt{{[}18, 3{]}} &$\checkmark$ & $\infty$ & $2$ & $2$  & $1$ & 1 $\times$  $M_2(\mathbb Q (\sqrt{-3}))$  & $2 {\times} \mathbb{Q}$, $2 {\times} \mathbb{Q}(\sqrt{-3}) = \mathbb{Q}(\zeta_{3})$,  \\
\texttt{{[}24, 1{]}} &$\checkmark$ & $\infty$ & $2$ & $2$  & $1$ & 1 $\times$  $M_2(\mathbb Q (i))$  & $2 {\times} \mathbb{Q}$, $1 {\times} \mathbb{Q}(i)$, $1 {\times} \mathbb{Q}(\zeta_{8})$, $1 {\times} \mathbb{H}_3$,  \\
\texttt{{[}24, 3{]}} &$\times$ & $\infty$ & $3$ & $3$  & $2$ & 1 $\times$  $M_2(\mathbb Q (\sqrt{-3}))$  &  \\
\texttt{{[}24, 5{]}} &$\checkmark$ & $\infty$ & $2$ & $2$  & $1$ & 1 $\times$  $M_2(\mathbb Q (i))$  & $4 {\times} \mathbb{Q}$, $2 {\times} \mathbb{Q}(i)$,  \\
\texttt{{[}24, 8{]}} &$\checkmark$ & $\infty$ & $2$ & $2$  & $1$ & 1 $\times$  $M_2(\mathbb Q (\sqrt{-3}))$  & $4 {\times} \mathbb{Q}$,  \\
\texttt{{[}24, 10{]}} &$\checkmark$ & $2$ & $2$ &  &  & 1 $\times$  $M_2(\mathbb Q (\sqrt{-3}))$  & $4 {\times} \mathbb{Q}$, $4 {\times} \mathbb{Q}(\sqrt{-3}) = \mathbb{Q}(\zeta_{3})$,  \\
\texttt{{[}24, 11{]}} &$\checkmark$ & $2$ & $2$ &  &  & 1 $\times$  $M_2(\mathbb Q (\sqrt{-3}))$  & $4 {\times} \mathbb{Q}$, $4 {\times} \mathbb{Q}(\sqrt{-3}) = \mathbb{Q}(\zeta_{3})$, $1 {\times} \mathbb{H}_2$,  \\
\texttt{{[}32, 8{]}} &$\checkmark$ & $3$ & $2$ &  &  & 1 $\times$  $M_2(\mathbb{H}_{2})$  & $4 {\times} \mathbb{Q}$, $2 {\times} \mathbb{Q}(i)$,  \\
\texttt{{[}32, 11{]}} &$\checkmark$ & $3$ & $2$ &  &  & 2 $\times$  $M_2(\mathbb Q (i))$  & $4 {\times} \mathbb{Q}$, $2 {\times} \mathbb{Q}(i)$,  \\
\texttt{{[}32, 44{]}} &$\checkmark$ & $3$ & $2$ &  &  & 1 $\times$  $M_2(\mathbb{H}_{2})$  & $8 {\times} \mathbb{Q}$,  \\
\texttt{{[}32, 50{]}} &$\checkmark$ & $2$ & $2$ &  &  & 1 $\times$  $M_2(\mathbb{H}_{2})$  & $16 {\times} \mathbb{Q}$,  \\
\texttt{{[}36, 6{]}} &$\checkmark$ & $\infty$ & $2$ & $2$  & $1$ & 1 $\times$  $M_2(\mathbb Q (\sqrt{-3}))$  & $2 {\times} \mathbb{Q}$, $1 {\times} \mathbb{Q}(i)$, $2 {\times} \mathbb{Q}(\sqrt{-3}) = \mathbb{Q}(\zeta_{3})$, $1 {\times} \mathbb{Q}(\zeta_{12})$, $1 {\times} \mathbb{H}_3$,  \\
\texttt{{[}36, 12{]}} &$\checkmark$ & $\infty$ & $2$ & $2$  & $1$ & 1 $\times$  $M_2(\mathbb Q (\sqrt{-3}))$  & $4 {\times} \mathbb{Q}$, $4 {\times} \mathbb{Q}(\sqrt{-3}) = \mathbb{Q}(\zeta_{3})$,  \\
\texttt{{[}40, 3{]}} &$\times$ & $\infty$ & $2$ & $4$  & $1$ & 1 $\times$  $M_2(\mathbb{H}_{5})$  &  \\
\texttt{{[}48, 16{]}} &$\checkmark$ & $\infty$ & $2$ & $2$  & $2$ & 1 $\times$  $M_2(\mathbb{H}_{2})$  & $4 {\times} \mathbb{Q}$,  \\
\texttt{{[}48, 18{]}} &$\checkmark$ & $\infty$ & $2$ & $2$  & $2$ & 1 $\times$  $M_2(\mathbb{H}_{3})$  & $4 {\times} \mathbb{Q}$, $1 {\times} \left(\frac{-1, -1}{\mathbb{Q}(\sqrt{2})}\right)$,  \\
\texttt{{[}48, 28{]}} &$\times$ & $\infty$ & $4$ & $6$  & $2$ & 1 $\times$  $M_2(\mathbb{H}_{3})$  &  \\
\texttt{{[}48, 29{]}} &$\times$ & $\infty$ & $4$ & $6$  & $2$ & 1 $\times$  $M_2(\mathbb Q (\sqrt{-2}))$  &  \\
\texttt{{[}48, 33{]}} &$\times$ & $\infty$ & $3$ & $3$  & $2$ & 1 $\times$  $M_2(\mathbb Q (i))$  &  \\
\texttt{{[}48, 39{]}} &$\checkmark$ & $\infty$ & $2$ & $2$  & $2$ & 1 $\times$  $M_2(\mathbb{H}_{3})$  & $8 {\times} \mathbb{Q}$,  \\
\texttt{{[}48, 40{]}} &$\checkmark$ & $\infty$ & $2$ & $2$  & $2$ & 1 $\times$  $M_2(\mathbb{H}_{2})$  & $8 {\times} \mathbb{Q}$, $2 {\times} \mathbb{H}_2$,  \\
\texttt{{[}64, 37{]}} &$\times$ & $4$ & $2$ &  &  & 2 $\times$  $M_2(\mathbb{H}_{2})$  &  \\
\texttt{{[}64, 137{]}} &$\checkmark$ & $3$ & $2$ &  &  & 2 $\times$  $M_2(\mathbb{H}_{2})$  & $8 {\times} \mathbb{Q}$,  \\
\texttt{{[}72, 19{]}} &$\times$ & $\infty$ & $2$ & $4$  & $1$ & 2 $\times$  $M_2(\mathbb{H}_{3})$  &  \\
\texttt{{[}72, 20{]}} &$\times$ & $\infty$ & $2$ & $4$  & $1$ & 1 $\times$  $M_2(\mathbb{H}_{3})$  &  \\
\texttt{{[}72, 22{]}} &$\times$ & $\infty$ & $2$ & $4$  & $1$ & 1 $\times$  $M_2(\mathbb{H}_{3})$  &  \\
\texttt{{[}72, 24{]}} &$\times$ & $\infty$ & $2$ & $4$  & $1$ & 1 $\times$  $M_2(\mathbb{H}_{3})$  &  \\
\texttt{{[}72, 25{]}} &$\times$ & $\infty$ & $3$ & $3$  & $2$ & 3 $\times$  $M_2(\mathbb Q (\sqrt{-3}))$  &  \\
\texttt{{[}72, 30{]}} &$\checkmark$ & $\infty$ & $2$ & $2$  & $1$ & 2 $\times$  $M_2(\mathbb Q (\sqrt{-3}))$  & $4 {\times} \mathbb{Q}$, $4 {\times} \mathbb{Q}(\sqrt{-3}) = \mathbb{Q}(\zeta_{3})$,  \\
\texttt{{[}96, 67{]}} &$\times$ & $\infty$ & $4$ & $6$  & $2$ & 2 $\times$  $M_2(\mathbb Q (i))$  &  \\
\texttt{{[}96, 190{]}} &$\times$ & $\infty$ & $4$ & $6$  & $2$ & 1 $\times$  $M_2(\mathbb{H}_{2})$  &  \\
\texttt{{[}96, 191{]}} &$\times$ & $\infty$ & $4$ & $6$  & $2$ & 1 $\times$  $M_2(\mathbb{H}_{2})$  &  \\
\texttt{{[}96, 202{]}} &$\times$ & $\infty$ & $3$ & $3$  & $2$ & 1 $\times$  $M_2(\mathbb{H}_{2})$  &  \\
\texttt{{[}120, 5{]}} &$\times$ &  & $\infty$ & $60$  & $1$ & 1 $\times$  $M_2(\mathbb{H}_{3})$  &  \\
\texttt{{[}128, 937{]}} &$\times$ & $4$ & $3$ &  &  & 4 $\times$  $M_2(\mathbb{H}_{2})$  &  \\
\texttt{{[}144, 124{]}} &$\times$ & $\infty$ & $4$ & $6$  & $2$ & 3 $\times$  $M_2(\mathbb{H}_{3})$  &  \\
\texttt{{[}144, 128{]}} &$\times$ & $\infty$ & $3$ & $6$  & $2$ & 1 $\times$  $M_2(\mathbb{H}_{2})$  &  \\
\texttt{{[}144, 135{]}} &$\times$ & $\infty$ & $2$ & $4$  & $1$ & 4 $\times$  $M_2(\mathbb{H}_{3})$  &  \\
\texttt{{[}144, 148{]}} &$\times$ & $\infty$ & $2$ & $4$  & $1$ & 2 $\times$  $M_2(\mathbb{H}_{3})$  &  \\
\texttt{{[}160, 199{]}} &$\times$ & $\infty$ & $3$ & $5$  & $2$ & 1 $\times$  $M_2(\mathbb{H}_{2})$  &  \\
\texttt{{[}192, 989{]}} &$\times$ & $\infty$ & $4$ & $6$  & $2$ & 2 $\times$  $M_2(\mathbb{H}_{2})$  &  \\
\texttt{{[}240, 89{]}} &$\times$ &  & $\infty$ & $120$  & $1$ & 1 $\times$  $M_2(\mathbb{H}_{5})$  &  \\
\texttt{{[}240, 90{]}} &$\times$ &  & $\infty$ & $120$  & $1$ & 1 $\times$  $M_2(\mathbb{H}_{5})$  &  \\
\texttt{{[}288, 389{]}} &$\times$ & $\infty$ & $3$ & $8$  & $1$ & 2 $\times$  $M_2(\mathbb{H}_{3})$  &  \\
\texttt{{[}320, 1581{]}} &$\times$ & $\infty$ & $4$ & $10$  & $2$ & 2 $\times$  $M_2(\mathbb{H}_{2})$  &  \\
\texttt{{[}384, 618{]}} &$\times$ & $\infty$ & $3$ & $3$  & $4$ & 1 $\times$  $M_2(\mathbb{H}_{2})$  &  \\
\texttt{{[}384, 18130{]}} &$\times$ & $\infty$ & $4$ & $6$  & $2$ & 1 $\times$  $M_2(\mathbb{H}_{2})$  &  \\
\texttt{{[}720, 409{]}} &$\times$ &  & $\infty$ & $360$  & $1$ & 2 $\times$  $M_2(\mathbb{H}_{3})$  &  \\
\texttt{{[}1152, 155468{]}} &$\times$ & $\infty$ & $4$ & $18$  & $2$ & 1 $\times$  $M_2(\mathbb{H}_{2})$  &  \\
\texttt{{[}1920, 241003{]}} &$\times$ &  & $\infty$ & $60$  & $2$ & 1 $\times$  $M_2(\mathbb{H}_{2})$  &  \\
\end{longtable}
\end{landscape}
\restoregeometry

\bibliographystyle{plain}
\bibliography{LowDeg}

\end{document}